%% file: ATM-R1.1.tex
\documentclass [10pt, article, oneside, a4paper]{memoir}
\usepackage[T1]{fontenc}
\usepackage[utf8]{inputenc}
\usepackage{mathtools}
\usepackage{lmodern}
\usepackage{hyphenat}
\usepackage[stretch=10]{microtype}
\usepackage[all]{xy}
\RequirePackage{amsmath,amsfonts,amsthm,mathrsfs,paralist,amssymb,bm,amsthm,graphicx,color,url,mathrsfs,xcolor,tikz}

\usepackage[colorlinks=true,linkcolor=black,citecolor=blue]{hyperref}

\newcommand{\aquatre}{%
\stockaiv\pageaiv%
\settypeblocksize{22cm}{13cm}{*}%
\setlrmargins{4cm}{*}{1}%
\setmarginnotes{0pt}{0pt}{0pt}%
\setulmargins{3.5cm}{*}{1}%
\setheadfoot{3\baselineskip}{3\baselineskip}%
\setheaderspaces{2\baselineskip}{*}{1}%
\checkandfixthelayout%
}
\aquatre

\newcommand{\addpoint}[1]{#1\ ---\ }

\setsecnumdepth{subsubsection}

\setsecnumformat{\csname the#1\endcsname---}
\setsubsubsechook{\setsecnumformat{{\normalfont\sffamily\csname
      the##1\endcsname\ }}}
\setsubsechook{\setsecnumformat{\csname the##1\endcsname\ ---\ }}
\setsechook{\setsecnumformat{\csname the##1\endcsname---}}

\setsecheadstyle{\centering\Large\bfseries\sffamily}
\setsubsecheadstyle{\large\bfseries\sffamily}
\setsubsubsecheadstyle{\bfseries\itshape\addpoint}
\setsubsubsecindent{1em}
\setbeforesubsubsecskip{.5em plus .2em minus -.1em}
\setaftersubsubsecskip{-0em}
\setsubparaheadstyle{\itshape}
\renewcommand*{\beforepartskip}{\null\vskip 0pt}

\newtheoremstyle{thm}%
     {1.5ex plus .3ex minus .1ex}%
     {1ex plus .3ex minus .1ex}%
     {\itshape}%
     {}%
     {\sffamily}%
     {---}%
     {0em}%
     {$\bullet$\hbox{\ }#1\hbox{\ }#2}%

\theoremstyle{thm}

\newtheorem{definition}{Definition}[section]
\newtheorem{theorem}[definition]{Theorem}

\newtheorem{lemma}[definition]{Lemma}
\newtheorem{proposition}[definition]{Proposition}
\newtheorem{corollary}[definition]{Corollary}

\newtheoremstyle{note}%
     {1ex plus .3ex minus .1ex}%
     {1ex plus .3ex minus .1ex}%
     {}%
     {}%
     {\itshape}%
     {.}%
     {1em}%
     {}%
\theoremstyle{note}

\newtheorem{remark}[definition]{Remark}

\newlength{\remaining}
\newenvironment{intermediate}[1]
	{\par\medskip\noindent\makebox[2em][t]{\hfill#1\hfill}%
          \setlength{\remaining}{\textwidth}\addtolength{\remaining}{-2em}%
          \begin{minipage}[t]{\remaining}\itshape
          }
          {\end{minipage}\par\medskip}

\newcommand{\verticale}{\ar@{--}[d]}

\newcommand{\VJ}[1]{{#1}}

\def\calli#1{\expandafter\def\csname
  #1\endcsname{\mathcal{#1}}}	%
\def\sets#1{\expandafter\def\csname
  bb#1\endcsname{\mathbb{#1}}}	%
\def\rebar#1{\expandafter\def\csname #1bar\endcsname{\overline{\csname
      #1\endcsname}}}		%
\def\gothify#1{\expandafter\def\csname
  #1#1#1\endcsname{\mathfrak{#1}}}	%

\sets{R}	%
\sets{N}	%
\sets{Z}	%
\calli{I}	%
\calli{M}	%
\calli{R}
\calli{K}
\calli{U}
\calli{W}
\calli{E}
\calli{D}
\calli{B}
\calli{V}
\calli{P}
\gothify{F}	%
\gothify{G}	%
\rebar{M}	%

\newcommand{\FF}{\mathcal{F}}
\renewcommand{\SS}{\mathcal{S}}
\newcommand{\SSp}{\mathcal{S}'}

\newcommand{\Ch}{\mathscr{C}}
\newcommand{\bA}{\mathbf{A}}
\newcommand{\bB}{\mathbf{B}}
\newcommand{\bC}{\mathbf{C}}
\newcommand{\bG}{\mathbf{G}}
\newcommand{\calA}{\mathcal{A}}
\newcommand{\calB}{\mathcal{B}}

\newcommand{\Rbar}{\overline{R}}

\newcommand{\up}[1]{\,\uparrow #1}
\newcommand{\down}[1]{\downarrow #1}
\newcommand{\Up}[1]{\,\Uparrow #1}
\newcommand{\seq}[2]{(#1_{#2})_{#2\geq1}}
\newcommand{\pr}{\mathbb{P}}
\newcommand{\esp}{\mathbb{E}}

\newcommand{\norm}[1]{\|#1\|}

\newcommand{\slgb}{\mbox{$\sigma$-al}\-ge\-bra}

\newcommand{\BA}{\partial\bA}
\newcommand{\bAbar}{\overline{\bA}}

\newcommand{\subseteqlub}{\Subset}

\newcommand{\tq}{\;\big|\;}
\newcommand{\tqs}{\;:\;}
\newcommand{\T}{\mathrm{T}}

\newcommand{\ie}{\textsl{i.e.}}
\newcommand{\iid}{\textsl{i.i.d.}}
\newcommand{\AND}{\text{ \normalfont and }}

\newcommand{\rest}[1]{\bigl|_{#1}}

\newcommand{\un}{\mathbf{1}}
\newcommand{\qref}[1]{{\normalfont(P\ref{#1})}}

\newcommand{\height}{\tau}

\newcommand{\unit}{{\bm{e}}}
\newcommand{\lel}{<_\text{\normalfont l}}
\newcommand{\ler}{<_{\text{\normalfont r}}}

\newcommand{\leql}{\leq_\text{\normalfont l}}
\newcommand{\leqr}{\leq_{\text{\normalfont r}}}

\DeclareMathOperator{\bigveel}{\mathchoice%
  {\bigvee{}_{\raisebox{-.8ex}{\scriptsize\!\!\text{\normalfont l}}}}%
  {\bigvee{}_{\raisebox{-.3ex}{\scriptsize\!\!\text{\normalfont l}}}}%
  {\bigvee{}_{\raisebox{-.17ex}{\tiny\!\text{\normalfont l}}}}%
  {\bigvee{}_{\raisebox{-.17ex}{\tiny\!\text{\normalfont l}}}}%
}
\DeclareMathOperator{\bigveer}{\mathchoice%
  {\bigvee{}_{\raisebox{-.8ex}{\scriptsize\!\!\text{\normalfont r}}}}%
  {\bigvee{}_{\raisebox{-.3ex}{\scriptsize\!\!\text{\normalfont r}}}}%
  {\bigvee{}_{\raisebox{-.17ex}{\tiny\!\text{\normalfont r}}}}%
  {\bigvee{}_{\raisebox{-.17ex}{\tiny\!\text{\normalfont r}}}}%
}
\DeclareMathOperator{\bigwedgel}{\mathchoice%
  {\bigwedge{}_{\raisebox{-.8ex}{\scriptsize\!\text{\normalfont l}}}}%
  {\bigwedge{}_{\raisebox{-.3ex}{\scriptsize\!\text{\normalfont l}}}}%
  {\bigwedge{}_{\raisebox{-.17ex}{\tiny\!\text{\normalfont l}}}}%
  {\bigwedge{}_{\raisebox{-.17ex}{\tiny\!\text{\normalfont l}}}}%
}
\newcommand{\veel}{\vee_{\!\text{\normalfont l}}}
\newcommand{\veer}{\vee_{\!\text{\normalfont r}}}
\newcommand{\wedgel}{\wedge_{\text{\normalfont l}}}

\DeclareMathOperator{\Diag}{Diag}

\DeclareMathOperator{\supp}{supp}

\newcommand{\CC}{\mathcal{C}}

\newcommand{\GG}{\mathcal{G}}

\newcommand{\LL}{\mathcal{L}}

\newcommand{\cvlaw}[1]{\xrightarrow[#1]{\mathcal{L}}}
\newcommand{\NN}{\mathcal{N}}

\newcommand{\lev}{\ell}
\newcommand{\rev}{r}
\newcommand{\lpm}[2]{\mu_{#1,#2}^-}

\newcommand{\gbar}{\overline{g}}
\newcommand{\wtilde}{\widetilde{w}}
\newcommand{\rtilde}{\widetilde{r}}
\newcommand{\ytilde}{\widetilde{y}}
\newcommand{\xtilde}{\widetilde{x}}

\newcommand{\thetatilde}{\widetilde{\theta}}

\newcommand{\ii}{\mathbf{i}}

\newcommand{\Mtilde}{\widetilde{M}}
\newcommand{\Ptilde}{\widetilde{P}}
\newcommand{\Ftilde}{\widetilde{F}}
\newcommand{\Htilde}{\widetilde{H}}
\newcommand{\mtilde}{\widetilde{m}}
\newcommand{\elltilde}{\widetilde{\ell}}

\newcommand{\CWG}{{\normalfont\textsf{CWG}}}
\renewcommand{\H}{\mathcal{H}}

\DeclarePairedDelimiter{\abs}{|}{|}

\makeatletter
\newcommand{\@@setuppart}{%
  \if@openright
  \else
  \fi
  \thispagestyle{part}%
  \if@twocolumn
    \onecolumn
    \@tempswatrue
  \else
    \@tempswafalse
  \fi
  \beforepartskip}
\def\@apppage{%
  \@@setuppart
  \addappheadtotoc
  \partmark{\appendixpagename}%
  \memapppageinfo{\appendixpagename}%
  {\centering
   \interlinepenalty \@M
   \normalfont
   \printparttitle{\appendixpagename}\par}%
  \@endpart}
\def\@sapppage{%
  \@@setuppart
  \partmark{\appendixpagename}%
  \memapppagestarinfo{\appendixpagename}%
  {\centering
   \interlinepenalty \@M
   \normalfont
   \printparttitle{\appendixpagename}\par}%
  \@endpart}
\makeatother

\numberwithin{equation}{section}

\begin{document}

\begin{center}
  \huge\bfseries Asymptotic combinatorics\\ of Artin-Tits monoids\\
\Large and of some other monoids
\end{center}

\bigskip
\noindent\hspace{-.075\textwidth}
\begin{minipage}[t]{.25\textwidth}
  \begin{center}
  Samy Abbes    \\[.5em]
\tiny
Institut de  Recherche en Informatique Fondamentale
 (IRIF -- UMR 8243)\\
Universit\'e Paris Diderot\\
75205 Paris Cedex 13, 
 France
  \end{center}
\end{minipage}%
\hspace{.05\textwidth}%
\begin{minipage}[t]{.25\textwidth}
  \begin{center}
Sébastien Gouëzel\\[.5em]
\tiny
Laboratoire de Math\'ematiques Jean Leray (LMJL -- UMR 6629)\\
Université de Nantes, CNRS\\
F-44000 Nantes, France
  \end{center}
\end{minipage}%
\hspace{.05\textwidth}%
\begin{minipage}[t]{.25\textwidth}
  \begin{center}
Vincent Jugé\\[.5em]
\tiny
Laboratoire d'Informatique Gaspard Monge (LIGM -- UMR 8049)\\
Universit\'e Paris-Est, 
UPEM, F77454 Marne-la-Vallée, France
  \end{center}
\end{minipage}%
\hspace{.05\textwidth}%
\begin{minipage}[t]{.25\textwidth}
  \begin{center}
 Jean Mairesse\\[.5em]
\tiny Laborattoire d'Informatique de Paris~6
(LIP6 -- UMR 7606), Sorbonne University, CNRS,
75252 Paris Cedex 05, France
  \end{center}
\end{minipage}%
\par\vspace{.5em}\noindent\hspace{-.075\textwidth}
\begin{minipage}[t]{.25\textwidth}
\centering
\tiny
\texttt{abbes@irif.fr}
\end{minipage}%
\hspace{.05\textwidth}%
\begin{minipage}[t]{.25\textwidth}
\centering\tiny
\texttt{sebastien.gouezel@univ-nantes.fr}
\end{minipage}%
\hspace{.05\textwidth}%
\begin{minipage}[t]{.25\textwidth}
\tiny
\centering\texttt{vincent.juge@u-pem.fr}
\end{minipage}%
\hspace{.05\textwidth}%
\begin{minipage}[t]{.25\textwidth}
\centering
\tiny
\texttt{jean.mairesse@lip6.fr}
\end{minipage}%

\bigskip

\begin{center}
March 2018--January 2019
\end{center}

\bigskip
\begin{abstract}
  We introduce methods to study the combinatorics of the normal form of
  large random elements in Artin-Tits monoids. These methods also apply in
  an axiomatic framework that encompasses other monoids such as dual braid
  monoids.

\end{abstract}

\bigskip

\section{Introduction}
\label{sec:introduction}

Braid monoids are finitely presented monoids with a rich combinatorics,
mainly based on the existence of a normal form for their elements. Consider
the braid monoid $B_n^+$ on $n$ strands: it has $n-1$ \emph{Artin generators}
which represent the elementary twists of two neighbor strands. The
\emph{Artin length} of an element $x\in B_n^+$ is the number $m$ such that $x$
can be written as a product of $m$ Artin generators. The decomposition is not
unique, yet $m$ is well defined. The Garside normal form, also called the
greedy normal form, arises from the following fact: there exists a finite
subset $\SS$ of~$B_n^+$\,, that not only generates~$B_n^+$, but also such that
any $x\in B_n^+$ writes uniquely as a product $x_1\cdot\ldots\cdot x_k$ of
elements of~$\SS$, provided that the sequence $(x_1,\ldots,x_k)$ satisfies a
condition of the form $x_1\to x_2\to\dots\to x_k$, where $x\to y$ is some
binary relation on~$\SS$.

Fix an integer $j$ and pick an element $x$ in $B_n^+$ at random among elements
of Artin length~$m$, with $m$ large. Then consider the $j$ first elements
$(x_1,\ldots, x_j)$ in the normal form of~$x$.  This is a complex random
object: the sequence $(x_1,\dots,x_j)$ has no Markovian structure in general,
and the law of $(x_1,\dots,x_j)$ depends on~$m$. However, some regularity
appears when considering the limit with $m\to\infty$. Indeed, the limit law
of the sequence $(x_1,\dots,x_j)$ is \emph{Markovian}, which already brings
qualitative information on the typical behavior of ``large random elements''
in a braid monoid. First motivated by this kind of question, we have
introduced in a previous work~\cite{abbes17} the notion of uniform measure at
infinity for braid monoids. The marginals of the uniform measure give back
the limit form of the law of $(x_1,\dots,x_j)$ when $m\to\infty$, hence
yielding an appropriate notion of uniform measure for ``infinite braid
elements''. The notion of uniform measure at infinity allowed us to settle
open questions regarding the behavior of the normal form of large braids,
asked by Gebhardt and Tawn in~\cite{gebhardt14}.

The following questions were left open in~\cite{abbes17}: how much of these
techniques extend to more general monoids? Furthermore, on top of the
qualitative information brought by the Markovian properties of the uniform
measure at infinity, can we also obtain asymptotics on some combinatorial
statistics? Here is a typical statistics for which precise asymptotics might
be expected. Beside the Artin length of elements in~$B_n^+$\,, the
\emph{height} of $x\in B_n^+$ is the length $k$ of the normal form
$(x_1,\dots,x_k)$ of~$x$. What is the typical behavior of the ratio $m/k$
when $k$ is the height of a random element in $B_n^+$ of Artin length $m$
large? Assume that this ratio is close to some constant~$\gamma$. Is there a
limit law for the quantity $\sqrt k(m/k-\gamma)$ when $m\to\infty$? Hence, we
were left with the two following natural extensions to carry over:
\begin{inparaenum}[1)]
  \item extension to other monoids than braid monoids; and
  \item asymptotics for combinatorial statistics.
\end{inparaenum}

A natural candidate for the generalization of braid monoids is the class of
Artin-Tits monoids. This class encompasses both braid monoids and trace
monoids---the later are also called heap monoids~\cite{viennot86} and free
partially commutative monoids~\cite{cartier69} in the literature. Contrasting
with braid monoids, a trace monoid is not a lattice for the left divisibility
order; in particular, it has no Garside element. Actually, trace monoids and
braid monoids can be seen as two typical representatives of Artin-Tits
monoids of type~FC, which is a subclass of Artin-Tits monoids already large
enough to provide a good view on Artin-Tits monoids.

In this paper, we develop the notion of uniform measure at infinity for
Artin-Tits monoids. More generally, we consider the class of
\emph{multiplicative measures at infinity} for an Artin-Tits monoid, among
which is the uniform measure at infinity. We first define the \emph{boundary
at infinity} $\BA$ of an Artin-Tits monoid~$\bA$ as the topological space of
``$\bA$-infinite elements''. It has the property that $\bA\cup\BA$ is a
compactification of~$\bA$. Multiplicative measures at infinity are
probability measures on~$\BA$ with a purely algebraic definition. The
combinatorics of the monoid is reflected within these measures. For instance,
the explicit expression for the uniform measure at infinity involves the root
of smallest modulus of the Möbius polynomial of the monoid~$\bA$. The Möbius
polynomial encodes some information on the combinatorics of the monoid. Its
root of smallest modulus is the inverse of the growth rate of the monoid. By
analogy, multiplicative measures at infinity play in an Artin-Tits monoid a
role which corresponds to the role of the standard Bernoulli measures in a
free monoid, as illustrated below:
\begin{gather*}
\xymatrix@R=.5em{\txt<6em>{$\Sigma^*$ : monoid of  $\Sigma$-words}\ar[r]&\txt<9.5em>{$\Sigma^\infty$ : space of infinite $\Sigma$-sequences}\ar[r]&\txt<8em>{$m$: Bernoulli measure on $\Sigma^\infty$}\\
\txt<7em>{$\bA$ : Artin-Tits monoid}\ar[r]&\txt<8em>{$\BA$ : boundary at infinity of~$\bA$}\ar[r]&\txt<8em>{$m$: multiplicative measure on $\BA$}
}
\end{gather*}

The uniform measure at infinity describes the qualitative behavior of a large
random element of the monoid. As for braid monoids, the first $j$ elements in
the normal form of a random element $x\in\bA$ of large size converge in law
toward a Markov chain, for which the initial distribution and the transition
matrix are explicitly described. This Markov chain corresponds to the
marginal of the uniform measure at infinity. Multiplicative measures play an
analogous role when large elements, instead of being picked uniformly at
random, are chosen at random with multiplicative weights associated to the
elementary generators of the monoid.

A second device that we introduce is the conditioned weighted graph (\CWG)
associated with an Artin-Tits monoid. This is essentially a non-negative
matrix encoding the combinatorics of the monoid, together with weights
attributed to the generators of the monoid.  \CWG\ are reminiscent of several
tools and techniques found elsewhere in the literature and can almost be
considered as folklore; we show in particular the relationship with the
classical notion of survival process. Our contribution consists in applying
spectral methods to derive asymptotics for \CWG, including a concentration
result with a convergence in law and a Central Limit Theorem.

We have thus two devices for studying the asymptotic combinatorics of
Artin-Tits monoids. On the one hand, the uniform measure at infinity has a
purely algebraic definition, and naturally encodes some information on the
combinatorics of the monoid. On the other hand, the \CWG\ associated with the
monoid entirely encodes the combinatorics of the monoid and analytical
methods can be used to obtain information on its asymptotic. Our program
consists in using the analytical results for \CWG\ in order to
\begin{inparaenum}[1)]
\item derive asymptotic results for large random elements in the monoid; and
\item derive additional information regarding multiplicative measures at
    infinity.
\end{inparaenum}
In particular, the set of finite probabilistic parameters that entirely
describe multiplicative measures at infinity is shown to be homeomorphic to a
standard simplex, and by this way we prove the existence of the uniform
measure at infinity in all cases, which is a non-trivial task.

The paper brings original contributions in the area of pure combinatorics of
Artin-Tits monoids. In particular, the following results are of interest
\emph{per se}:
\begin{inparaenum}[1)]
\item If an Artin-Tits monoid is irreducible, its graph of proper simple elements
  is strongly connected---this graph contains as vertices all the simple elements of the monoid, excepted the unit element and the Garside element if it exists.
\item We introduce a generalized Möbius transform and we give an explicit
    form of its inverse for general Artin-Tits monoids.
\item We give a simple characterization of monoids of type FC among
    Artin-Tits monoids.
\end{inparaenum}

Finally, we observe that our methods also apply outside the framework of
Artin-Tits monoids. We isolate a working framework where the chain of
arguments that we use can be repeated \emph{mutatis mutandis}. We illustrate
this working framework by giving examples of monoids fitting into it, yet
living outside the class of Artin-Tits monoids. Of particular interest is the class of dual braid monoids and their generalizations for Artin-Tits monoids of spherical type. We show that they fit indeed into our general framework.

\paragraph{Outline.}

In Section~\ref{sec:posit-dual-posit}, we first recall the basic definitions
and some results regarding the combinatorics of Artin-Tits monoids. The
notion of normal form of elements of an Artin-Tits monoid leads us to the
definition of the boundary at infinity of the monoid. We describe the
relationship between measures at infinity and Möbius transform on the monoid.
We particularize then our study to the class of multiplicative measures at
infinity, among which is the uniform measure at infinity. We characterize
multiplicative measures at infinity by a finite number of probabilistic
parameters with suitable normalization conditions, and we state a uniqueness
result for the uniform measure.

In Section~\ref{sec:centr-limit-theor} we introduce a technical device: the
notion of conditioned weighted graph (\CWG). Based on the Perron-Frobenius
theory for primitive matrices, we observe that a natural notion of weighted
measure at infinity can be attached to any \CWG.

In Section~\ref{sec:appl-posit-braids}, we show how the theory of \CWG\
applies in particular to Artin-Tits monoids. The two notions of uniform
measure at infinity introduced above are related to each other. We obtain in
particular a parametrization of all multiplicative measures at infinity for
an Artin-Tits monoid, and the existence of the uniform measure at infinity.

Section~\ref{sec:asympt-conc-ergod} is devoted to the of asymptotic study of
combinatorial statistics, first in the general case of \CWG, and then applied
to the case of Artin-Tits monoids. Using spectral methods, we state a
concentration result and a Central Limit Theorem in the framework of \CWG,
relatively to weak convergences. We state the corresponding convergence
results for the asymptotic combinatorics of Artin-Tits monoids, yielding very
general answers to the questions raised in the first part of this
introduction, for braids.

Finally, in Section~\ref{sec:simil-results-other}, we extract some minimal
properties that were needed to carry our analysis. Provided that these
minimal properties hold indeed, one states the corresponding results for more
general monoids. We also give examples of such monoids outside the class of
Artin-Tits monoids, including the class of dual monoids of Artin-Tits monoids of spherical type.

\section{Probability measures on the boundary of Artin-Tits monoids}
\label{sec:posit-dual-posit}

In this section we review some background on the combinatorics of Artin-Tits
monoids. It turns out that the notion of Garside family is of great interest,
both from a theoretical and from a computational viewpoint. Therefore we
focus on Garside subsets for Artin-Tits monoids and related notions: simple
elements and the associated normal form in particular. We shall see later
that the probabilistic viewpoint also takes great advantage of these notions,
which were originally devised mainly to explore combinatorial aspects of
Artin-Tits monoids.

\subsection{Definitions and examples of Artin-Tits monoids}
\label{sec:definitions-examples}

\subsubsection{Definition of Artin-Tits monoids}
\label{sec:defin-artin-tits}

Let a finite, non empty alphabet $\Sigma$ be equipped with a symmetric
function $\ell : \Sigma \times \Sigma \mapsto \{2,3,4,\ldots\} \cup
\{\infty\}$, \ie, such that $\ell(a,b) = \ell(b,a)$ for all $(a,b) \in \Sigma
\times \Sigma$.  Associate to the pair $(\Sigma,\ell)$ the binary relation
$I$ on~$\Sigma$, and the monoid $\bA = \bA(\Sigma,\ell)$
defined by the following presentation~\cite{brieskorn72}:
\begin{align*}
  \bA&=\bigl\langle\Sigma\ \big|\ ababa\ldots = babab \ldots \text{
       for $(a,b)\in I$}\bigr\rangle^+\,,\\
\text{where }&\text{$ababa\ldots$ and $babab\ldots$ both have length $\ell(a,b)$},\\
\text{and }
I& = \bigl\{(a,b)\in\Sigma\times\Sigma\tq a \neq b \text{ and } \ell(a,b) <
\infty\bigr\}.
\end{align*}
Note that the values of $\ell$ on the
diagonal of $\Sigma\times\Sigma$ are irrelevant. Such a monoid is called an \emph{Artin-Tits monoid}.

\subsubsection{Examples of Artin-Tits monoids}
\label{sec:examples-artin-tits}

The free monoid generated by $\Sigma$ is isomorphic to $\bA(\Sigma,\ell)$
with $\ell(\cdot,\cdot)=\infty$. The free Abelian monoid generated by
$\Sigma$ is isomorphic to $\bA(\Sigma,\ell)$ with $\ell(\cdot,\cdot)=2$. More
generally, considering $\bA(\Sigma,\ell)$ with $\ell$ ranging only
over~$\{2,\infty\}$, but assuming possibly both values, yields the class of
so-called \emph{heap monoids} on~$\Sigma$, also called \emph{trace monoids}
on~$\Sigma$. They are the monoids analogous to the so-called right-angled
Artin groups~\cite{charney07}.

Braid monoids are also specific instances of Artin-Tits monoids. Indeed, for
every integer $n\geq3$, the braid monoid on $n$ strands is the monoid $B_n^+$
generated by $n-1$ elements $\sigma_1,\dots,\sigma_{n-1}$ with the following
relations~\cite{garside1969braid}: the braid relations
$\sigma_i\sigma_j\sigma_i=\sigma_j\sigma_i\sigma_j$ for $1 \leq i \leq n-2$
and $j=i+1$, and the commutativity relations
$\sigma_i\sigma_j=\sigma_j\sigma_i$ for $1 \leq i, j \leq n-1$ and
$\abs{i-j}>1$. Hence $B_n^+$ is isomorphic to $\bA(\Sigma,\ell)$ by choosing
$\Sigma=\{\sigma_1,\dots,\sigma_{n-1}\}$ and $\ell(\sigma_i,\sigma_j)=3$ for
$\abs{i-j}=1$ and $\ell(\sigma_i,\sigma_j)=2$ for $\abs{i-j}>1$.

\subsubsection{Length and orders}
\label{sec:length-orders}

Let $\bA=\bA(\Sigma,\ell)$ be an Artin-Tits monoid. The \emph{length} of
$x\in\bA$, denoted by~$\abs{x}$, is the length of any word in the equivalence
class~$x$, with respect to the congruence defining~$\bA$ (it does not depend
on the choice of the word as the relations do not modify the length).

In particular, elements of $\Sigma$ are characterized as those $x\in\bA$ such
that $\abs{x}=1$. It follows that, when considering an Artin-Tits
monoid~$\bA$, we may refer to the pair $(\Sigma,\ell)$ such that
$\bA=\bA(\Sigma,\ell)$. In particular, elements of $\Sigma$ are called
\emph{generators} of~$\bA$.

The monoid $\bA$ is equipped with the \emph{left} and with the \emph{right
divisibility} relations, denoted respectively by $\leql$ and by~$\leqr$\,,
which are both partial orders on~$\bA$\,, defined by:
\begin{align*}
  x\leql y&\iff\exists z\in\bA\quad y=x\cdot z\,,\\
  x\leqr y&\iff\exists z\in\bA\quad y=z\cdot x\,.
\end{align*}
We also denote by $\lel$ and $\ler$ the associated strict orders.

The results stated in this paragraph and the next one are proved in~\cite[Ch.~IX, Prop.~1.26]{dehornoy2013foundations}. Every Artin-Tits monoid $\bA$ is both left and right cancellative, meaning:
\begin{gather*}
  \forall x,y,z\in\bA\quad (z\cdot x=z\cdot y\implies x=y)\wedge(x\cdot
  z=y\cdot z\implies x=y).
\end{gather*}

The partially ordered set $(\bA,\leql)$ is a lower semilattice, \ie, any finite set has a greatest lower bound. Furthermore, any $\leql$-bounded set has a $\leql$-least upper bound. 

We denote by $\bigwedgel X$ and by $\bigveel X$ the greatest $\leql$-lower bound and the least upper bound of a subset $X\subseteq\bA$, if they exist. We use the standard notations $a\wedgel b=\bigwedgel\{a,b\}$
for $a,b\in\bA$.
We also write $a\veel b=\bigveel\{a,b\}$ when defined. The analogous
notations~$\bigveer$, \emph{etc}, are defined with respect to the right
divisibility relation.

Unlike braid monoids, $(\bA,\leql)$~is not necessarily a lattice. For instance, distinct generators $a$ and $b$ of a free monoid have no common multiple and thus $\{a,b\}$ has no $\leql$-upper bound.

If $a$ and $b$ are two generators of~$\bA$, then $a\veel b$ exists if and only if $\ell(a,b)<\infty$, if and only if $a\veer b$ exists, and then $a\veel b=a\veer b=ababa\ldots$\,, where the rightmost member has length $\ell(a,b)$. More generally, let $S\subseteq\Sigma$ be a set of generators of~$\bA$. Then it follows from \cite[Ch.~IX, Proposition~1.29]{dehornoy2013foundations} that the following conditions are equivalent:
\begin{inparaenum}[(i)]
  \item $S$~is $\leql$-bounded; 
  \item $\bigveel S$~exists;
  \item $S$~is $\leqr$-bounded;
  \item $\bigveer S$ exists.
\end{inparaenum}
If these conditions are fulfilled, the element $\Delta_S=\bigveel S$ has the following property:
\begin{gather}
  \label{eq:28}
\forall x\in\bA\quad x\leql\Delta_S\iff x\leqr\Delta_S\,.
\end{gather}
It implies in particular:
\begin{align}
  \label{eq:23}
\Delta_S&=\bigveer S,&\text{and}&&\forall x\in S\quad x&\leqr\Delta_S\,.
\end{align}

Furthermore, it follows from \cite[Ch.~IX, Cor.~1.12]{dehornoy2013foundations} that the sub-monoid of $\bA$ generated by $S$ is isomorphic to the Artin-Tits monoid $\bA(S,\ell\rest{S\times S})$. In particular:
\begin{gather}
  \label{eq:24}
\Delta_S\in\langle S\rangle,\qquad\text{the sub-monoid of $\bA$ generated by~$S$.}
\end{gather}

If $X$ is a subset of~$\bA$, we say that $X$ is \emph{closed under existing
$\veel$} to mean that, if $a$ and $b$ are elements of $X$ such that $a\veel
b$ exists in~$\bA$, then $a\veel b\in X$.

\subsubsection{Artin-Tits monoids of spherical type}
\label{sec:artin-tits-monoids}

The monoid $\bA=\bA(\Sigma,\ell)$ is said to have \emph{spherical type} if
$\bigveel \Sigma$ exists. In this case, the partially ordered set
$(\bA,\leql)$ is a lattice. The element $\bigveel \Sigma$ is denoted
by~$\Delta$ and is called the \emph{Garside element of\/~$\bA$}.

An equivalent definition,  oftenly found in the literature, is the following. Let $\bC=\bC(\Sigma,\ell)$ be the presented group with $\Sigma$ as set of generators, and with the same relations given in Section~\ref{sec:defin-artin-tits} in the definition of $\bA(\Sigma,\ell)$, together with all the relations $s^2=1$ for $s$ ranging over~$\Sigma$. The group $\bC$ is the Coxeter group~\cite{coxeter1935complete} associated with~$\bA$. Then $\bA$ has spherical type if and only if $\bC$ is finite, see~\cite[Th.~5.6]{brieskorn72}. 

For instance, it is well known that the braid monoid $B_n^+$ defined as in
Section~\ref{sec:examples-artin-tits} has spherical type, with Garside
element given by $\Delta=(\sigma_1\cdot\sigma_2\cdot\ldots\cdot\sigma_{n-1})
\cdot
(\sigma_1\cdot\sigma_2\cdot\ldots\cdot\sigma_{n-2})\cdot\ldots\cdot(\sigma_1\cdot\sigma_2)\cdot\sigma_1$.
A heap monoid is of spherical type if and only if it is a free Abelian
monoid.

\subsubsection{Irreducibility of Artin-Tits monoids and Coxeter graph}
\label{sec:irred-artin-tits}

\begin{definition}
  \label{def:13}
An Artin-Tits monoid $\bA$ is called \emph{irreducible} if it is not
isomorphic to the direct product of two Artin-Tits monoids.
\end{definition}

For instance, braid monoids and free monoids are all irreducible. A free
Abelian monoid is irreducible if and only if it has only one generator. For a
heap monoid $\M=\bA(\Sigma,\ell)$ with $\ell(\cdot,\cdot)\in\{2,\infty\}$,
define $D=\{(a,b)\in\Sigma\times\Sigma\tq a=b\text{ or }\ell(a,b)=\infty\}$.
Then $\M$ is irreducible if and only if the undirected graph $(\Sigma,D)$ is
connected.

More generally, the irreducibility is related to the Coxeter graph of the
monoid, defined as follows.

\begin{definition}
The \emph{Coxeter graph} of an Artin-Tits monoid $\bA(\Sigma,\ell)$ is the
undirected graph $\bG=(\Sigma,E)$, with $E=\{(s,t)\in\Sigma\times\Sigma\tq
s=t\text{
  or }\ell(s,t)\geq3\text{ or }\ell(s,t)=\infty\}$.
\end{definition}

As observed in~\cite[\S~7.1]{brieskorn72}, we have  the following result.

\begin{proposition}
  \label{prop:1}
An Artin-Tits monoid is irreducible if and only if its Coxeter graph is
connected.
\end{proposition}

\subsection{Normal sequences and normal form of elements}
\label{sec:simple-braids-gars-2}

We fix an Artin-Tits monoid $\bA=\bA(\Sigma,\ell)$.

\subsubsection{Garside subsets}
\label{sec:garside-subsets}

A subset $\GG$ of $\bA$ is a \emph{Garside subset} of $\bA$ if it contains
$\Sigma$ and if it is closed under existing $\veel$ and downward closed
under~$\leqr$, the latter meaning:
  \begin{align*}
\forall x\in\GG\quad\forall y\in\bA\quad
                   y\leqr x\implies y\in\GG.
  \end{align*}

The following result is proved in~\cite{dehornoy2015garside}.

\begin{proposition}
  \label{pro:FC-0}
  Any Artin-Tits monoid admits a \emph{finite} Garside subset.
\end{proposition}

The class of Garside subsets of $\bA$ is obviously closed by intersection,
hence $\bA$ admits a \emph{smallest} Garside subset, which we denote
throughout the paper by~$\SS$. The subset $\SS$ is the closure of $\Sigma$
under $\leqr$ and existing~$\veel$\,. Proposition~\ref{pro:FC-0} tells us
that the set $\SS$ thus constructed is finite. By construction,
$\SS$~contains $\Sigma\cup\{\unit\}${, where $\unit$ is the unit element of
the monoid}.

\begin{definition}
  \label{def:10}
The elements of the smallest Garside subset of an Artin-Tits monoid are
called its \emph{simple elements}.
\end{definition}

Assume that $\bA$ is of spherical type. Then, according to~\cite[Ch.~IX, Prop.~1.29]{dehornoy2013foundations}, the set $\SS$ coincides with the set of left divisors of~$\Delta$, which is also the set of right divisors of~$\Delta$. Hence: $\Delta=\bigveel\Sigma=\bigveer\Sigma=\bigveel\SS=\bigveer\SS$, and $\Delta$ is the maximum of the sub-lattice~$(\SS,\leql)$.

\subsubsection{Artin-Tits monoids of type FC}
\label{sec:artin-tits-monoids-1}

If the smallest Garside subset of an Artin-Tits monoid is closed under
\emph{left} divisibility, then the combinatorics of the monoid is a bit more
simple to study.

\begin{definition}
  \label{def:7}
An Artin-Tits monoid\/~$\bA$, with smallest Garside subset~$\SS$, is of
\emph{type FC} if $\SS$ is closed under left divisibility.
\end{definition}

It is proved in~\cite[Th.~2.85]{juge16:_combin} that this definition is
indeed equivalent to the one found in the
literature~\cite{godelle03,dehornoy2013foundations}. In particular, heap
monoids and braid monoids are monoids of type FC.

Not all Artin-Tits monoids are of type FC: see an example in
Section~\ref{sec:an-example} below.

\subsubsection{Normal sequences}
\label{sec:simple-elements}

Most of what is known on the combinatorics of Artin-Tits monoids is based on
the notion of normal sequence.

\begin{definition}
  \label{def:1ppopoqqjhgff}
  Let\/ $\bA$ be an Artin-Tits monoid with smallest Garside
  subset~$\SS$. A sequence $(x_1,\ldots,x_k)$, with $k\geq1$, of
  elements of $\bA$ is \emph{normal} if it satisfies:
  \begin{enumerate}
  \item $x_i\in\SS$ for all $i=1,\dots,k$.
  \item   $x_i = \bigveel\left\{\zeta \in \SS \tq \zeta \leql x_i \cdot
      \ldots \cdot x_k\right\}$ for all $i =1,\ldots,k$.
  \end{enumerate}
\end{definition}

A fundamental property of normal sequences is the following.

\begin{lemma}\label{lem:2-FC}
  Let\/ $\bA$ be an Artin-Tits monoid.
\begin{enumerate}
\item\label{item:9} A sequence $(x_1,\ldots,x_k)$, with $k\geq1$, of
simple elements of\/ $\bA$ is normal if and only if it satisfies:
\begin{gather*}
x_i = \bigveel\left\{\zeta \in \SS \tq \zeta \leql x_i \cdot
    x_{i+1}\right\}\quad\text{for all $i=1,\ldots,k-1$.}
\end{gather*}
\item\label{item:10} A sequence $(x_1,\ldots,x_k)$, with $k\geq1$, of
simple elements of\/ $\bA$ is normal if and only if all sequences
  $(x_i,x_{i+1})$ are normal, for $i\in\{1,\dots,k-1\}$.
\end{enumerate}
\end{lemma}

\begin{proof}
  Point~\ref{item:9} is proved in~\cite{dehornoy2015garside}, and
  point~\ref{item:10} follows at once from point~\ref{item:9}.
\end{proof}

Let $x\to y$ denote the relation $x = \bigveel\left\{\zeta \in \SS\tq \zeta
\leq x \cdot y\right\}$, defined for $(x,y)\in\SS\times\SS$.
Point~\ref{item:10} of Lemma~\ref{lem:2-FC} reduces the study of normality of
sequences to the study of the binary relation $\to$ on~$\SS$.

The unit element $\unit$ of $\bA$ satisfies:
\begin{align}
\label{eq:21}
  \forall x\in\SS\quad x&\to\unit,&\forall x\in\SS\quad \unit\to
                                    x&\iff x=\unit.
\end{align}
Hence $\unit$ can only occur at the end of normal sequences, and $\unit$ is
the only element of $\SS$ satisfying the two properties in~\eqref{eq:21}.

Dually, if $\bA$ is of spherical type, the Garside element $\Delta=\bigveel\Sigma$ satisfies:
\begin{align}
\label{eq:22}
  \forall x\in\SS \quad \Delta&\to x,&\forall x\in\SS \quad
  x\to\Delta&\iff x=\Delta.
\end{align}
Hence $\Delta$ can only occur at the beginning of normal sequences.
Obviously, $\Delta$~is the only simple element satisfying the two properties
in~\eqref{eq:22}.

\subsubsection{Charney graph}
\label{sec:charney-graph}

We define the Charney graph $(\Ch,\to)$ of an Artin-Tits monoid $\bA$ as
follows. If $\bA$ is of spherical type, we put
$\Ch=\SS\setminus\{\Delta,\unit\}$, where $\Delta$ is the Garside element
of~$\bA$. If not, we put $\Ch=\SS\setminus\{\unit\}$. In all cases, the edge
relation $\to$ is the restriction to $\Ch\times\Ch$ of the relation $\to$
defined above on~$\SS\times\SS$.

The relevance of this definition will appear in
Section~\ref{sec:conn-charn-graph-1} below.

\subsubsection{Normal form of elements and height}
\label{sec:normal-form-general}

Let $\bA$ be an Artin-Tits monoid. Then, for every element $x\in\bA$ with
$x\neq\unit$, there exists a unique integer $k\geq1$ and a unique normal
sequence $(x_1,\ldots,x_k)$ of \emph{non unit simple elements} such that
$x=x_1\cdot\ldots\cdot x_k$\,.  This sequence is called the \emph{normal
form} of~$x$. By convention, we decide that $(\unit)$ is the normal form of
the unit element of~$\bA$.

The integer $k$ is called the \emph{height} of~$x$, and we denote it:
\begin{gather*}
  k=\height(x).
\end{gather*}

The following result is standard~\cite[Chap.~III,
Prop.~1.36]{dehornoy2013foundations}. It shows that comparing elements can be
done after `cutting' at the right height.

\begin{lemma}
 \label{lem:1}
Let\/ $\bA$ be an Artin-Tits monoid, and let $\SS$ be the smallest Garside subset of\/~$\bA$. Then:
\begin{enumerate}
\item\label{item:3} If\/ $(x_1,\ldots,x_k)$ is the normal form of and element $x\in\bA$, then
    $x_1\cdot\ldots\cdot x_j= \bigveel\{z \in \SS^j \tq z \leql x\}$ for
    all $j\in\{1,\ldots,k\}$, where $\SS^j=\{u_1\cdot\ldots\cdot u_j\tq
    (u_1,\dots,u_j)\in\SS\times\dots\times\SS\}$.
\item\label{item:4} The height $\height(x)$ of an element $x\in\bA$ is the
    smallest integer $j\geq1$ such that $x\in\SS^j$.
\end{enumerate}
\end{lemma}

The normal form of elements does not behave `nicely' with respect to the
monoid multiplication. For instance, the multiplication of an element
$x\in\bA$, of normal form $(x_1,\dots,x_k)$, by an element $y\in\bA$, yields
in general an element $z$ of normal form $(z_1,\dots,z_{k'})$ with no simple
relation between $x_j$ and~$z_j$.

It is even possible that the multiplication $x\cdot\sigma$ of $x\in\bA$ with
a generator $\sigma\in\Sigma$, satisfies $\height(x\cdot\sigma)<\height(x)$
(see an example in Section~\ref{sec:an-example}). This contrasts with monoids
of type FC, where $\height(x\cdot\sigma)\geq\height(x)$ always holds.

\subsubsection{Conditions for normality of sequences}
\label{sec:conn-charn-graph}

Aiming at studying the connectedness of the Charney graph in the next
section, one needs theoretical tools to construct normal sequences. Such
tools exist in the literature. They include the \emph{letter set}, the
\emph{left set} and the \emph{right set} of an element $x \in \bA$, which are
respectively defined as the following subsets of~$\Sigma$:
\begin{align*}
  \LL(x)& = \{\sigma \in \Sigma\tq \exists y,z \in \bA\quad x = y\cdot \sigma\cdot z\}, \\
  L(x) & = \{\sigma \in \Sigma\tq \sigma \leql x\}, \\
  R(x) & = \{\sigma \in \Sigma \tq \sigma \leqr x\quad \text{or}\quad (\exists \eta \in \mathcal{L}(x)\quad \sigma \neq \eta \text{ and } \ell(\sigma,\eta) = \infty)\}.
\end{align*}
The letter set $\LL(x)$ of $x$ is the set of letters that appear in some word
representing $x$ or, equivalently, in any word representing $x$, as relations
in Artin-Tits monoids use the same letters on both sides of the relation.

\begin{lemma}
\label{lem:3-FC} Let\/ $\bA=\bA(\Sigma,\ell)$ be an Artin-Tits monoid.  The
subsets
\begin{align*}
  \GG &= \{x \in \bA \tq \forall \sigma \in \Sigma\quad \forall y,z \in \bA\quad x \neq y\cdot \sigma^2\cdot z\}, \\
  \I &= \{x \in \bA \tq \forall \sigma, \eta \in
       \LL(x)\quad\sigma\neq\eta\implies \ell(\sigma,\eta)<\infty\},
\end{align*}
are Garside subsets of\/~$\bA$.
\end{lemma}

\begin{proof}
  It is proved in~\cite[Theorem 6.27]{dehornoy2013foundations} that $\GG$
  is a Garside subset.  Since $\I$ is clearly downward closed
  under~$\leqr$ and contains~$\Sigma$, we focus on proving that $\I$
  is closed under~$\veel$. In passing, we also note that $\I$ is
  downward closed under~$\leql$.

  Seeking a contradiction, assume the existence of $x,y\in\I$ such that
  $z=x\veel y$ exists in $\bA$ but $z\notin\I$.
Without loss of generality, we assume that $z$ is such an element
  of minimal length and that the element
  $w = x \wedgel y$ is maximal among all the elements of the set
  $\{u \wedgel v \tq u,v \in \I \text{ and } u \veel v = z\}$.
 Consequently, a contradiction,
  and therefore a proof of the lemma, is obtained by proving
  the following claim:
\begin{intermediate}
    {\normalfont$(\dag)$}%
   There exist elements $x'$ and $y'$ in~$\I$
    such that $x' \veel y' = z$ and $w \lel (x' \wedgel y')$.
  \end{intermediate}

  Since $z \notin \{x,y\}$, observe that neither $x\leql y$ nor
  $y\leql x$ hold, whence $w \lel x$ and $w\lel y$.  Thus we pick
  $\sigma,\eta\in\Sigma$ such that $w\cdot\sigma\leql x$ and
  $w\cdot\eta\leql y$. Then $w\cdot \sigma\leql z$ and
  $w\cdot\eta\leql z$, which implies that $\sigma$ and $\eta$ have a
  common $\leql$-upper bound, and thus $\sigma\veel\eta$ exists by the
  remarks made in Section~\ref{sec:length-orders}, and it is equal to
  $\sigma \eta \sigma \dotsb$.

  Since $\I$ is $\leql$-downward closed, and since $x$ and $y$ belong
  to~$\I$, the elements $w\cdot\sigma$ and $w\cdot\eta$ both belong
  to~$\I$. It follows that $\ell(a,b)<\infty$ for all pairs $(a,b)$
  with $a\neq b$ and $a,b\in\LL(w)\cup\{\sigma,\eta\}$.  Therefore,
  putting $t=w\cdot(\sigma\veel\eta)$, we have that $t\in\I$.
  Since $t$ also writes as $t=(w\cdot\sigma)\veel(w\cdot\eta)$, it
  is clear that $t\leql z$.

  Hence, the element $u = t \veel x$
  exists, and $u \leql z$. If $u = z$, then setting $x' = x$ and $y' = t$
  gives us $w \lel (w \cdot \sigma) \leql (x' \wedgel y')$ and
  $z = x' \veel y'$. If $u \lel z$, then by minimality of $z$ we have
  $u \in \I$, and therefore setting $x' = u$ and $y' = y$ gives us
  $w \lel (w \cdot \eta) \leql (t \wedgel y) \leql( x' \wedgel y')$,
  $z =( x \veel y) \leql (x' \veel y')$ and $(x' \veel y') \leql z$,
  whence $z = x' \veel y'$.
  This completes the proof of the claim and of the lemma.
\end{proof}

We obtain the following \emph{sufficient} criterion for detecting normal
sequences.

\begin{corollary}
\label{cor:4-FC} Let\/ $\bA$ be an Artin-Tits monoid.
\begin{enumerate}
\item\label{item:13} If $x$ and $y$ are two non unit simple elements of\/
    $\bA$ satisfying $L(y)\subseteq R(x)$, then $(x,y)$ is normal.
\item\label{item:14} If $(x_1,\dots,x_k)$, with $k\geq1$, is a sequence of
    non unit simple elements of\/~$\bA$ satisfying $L(x_{i+1})\subseteq
    R(x_i)$ for all $i\in\{1,\dots,k-1\}$, then it is normal.
\end{enumerate}
\end{corollary}

\begin{proof}
  Point~\ref{item:14} follows from point~\ref{item:13}, in
  view of Lemma~\ref{lem:2-FC}, point~\ref{item:10}. To prove
  point~\ref{item:13}, let $x$ and $y$ be as in the statement and let
  $u = \bigveel \left\{z \in \SS \tq z \leql x \cdot y\right\}$; we
  prove that $x=u$.

  Clearly, $x\leql u$. Seeking a contradiction, assume that $u\neq x$.
  Then there exists $z\in\SS$ such that $x \lel z\leql x\cdot y$.
  Let $\sigma\in\Sigma$ be such that
  $x\cdot \sigma\leql z$. Then $x\cdot\sigma\leql x\cdot y$ and
  thus $\sigma\leql y$ since $\bA$ is left cancellative, hence
  $\sigma\in L(y)\subseteq R(x)$.

  Let $\GG$ and $\I$ be the Garside subsets of $\bA$ introduced in
  Lemma~\ref{lem:3-FC}. Discussing the property $\sigma\in R(x)$, one
  has: (1)~if $\sigma\leqr x$ then $x\cdot\sigma\notin\GG$, and (2)~if
  $\ell(\sigma,\eta)=\infty$ for some $\eta\in\LL(x)$ with
  $\eta\neq\sigma$, then $x\cdot\sigma\notin\I$. In both cases, we
  have thus $x\cdot\sigma\notin\GG\cap\I$. Since $\GG$ and $\I$ are
  both closed under~$\leql$, so is $\GG\cap\I$, and thus
  $z\notin\GG\cap\I$ since $x\cdot \sigma\leql z$. But
  $\SS\subseteq\GG\cap\I$ according to Lemma~\ref{lem:3-FC}, which
  contradicts that $z\in\SS$ and completes the proof.
\end{proof}

\subsubsection{Connectedness of the Charney graph}
\label{sec:conn-charn-graph-1}

Recall that we have defined the Charney graph of an Artin-Tits monoid in
Section~\ref{sec:charney-graph}. We aim to prove the following result.

\begin{theorem}
  \label{thr:1qqlknaa}
The Charney graph of an irreducible Artin-Tits monoid is strongly connected.
\end{theorem}

We postpone the proof of this theorem to the end of the section, and prove
two intermediate results first.

\begin{lemma}\label{lem:5-FC}
Let\/ $\bA=\bA(\Sigma,\ell)$ be an Artin-Tits monoid, let $S$ be a subset
of\/ $\Sigma$ such that the element $\Delta_S = \bigveel S$ exists, and let
$\sigma\in \Sigma\setminus S$.
\begin{enumerate}
\item\label{item:1} Let $\LL^\ast(\sigma,S) = \{\eta \in S\tq
    \ell(\sigma,\eta) = 2\}$. Then:
\begin{align*}
L(\sigma\cdot \Delta_S) &= \{\sigma\} \cup \LL^\ast(\sigma,S),
& R(\sigma\cdot \Delta_S) &\supseteq S \cup \{\eta \in \Sigma \tq\ell(\sigma,\eta) = \infty\}.
\end{align*}
\item\label{item:2} If $\ell(\sigma,\eta) < \infty$ for all $\eta \in S$,
    then $\sigma\cdot \Delta_S $ is simple.
\end{enumerate}
\end{lemma}

\begin{proof}
  \ref{item:1}.\quad The rightmost inclusion follows from the following two observations: the relation   $R(\sigma\cdot \Delta_S) \supseteq \{\eta \in \Sigma \tq
  \ell(\sigma,\eta) = \infty\}$
is immediate from the definitions stated at the beginning of Section~\ref{sec:conn-charn-graph}, and the relation $R(\sigma \cdot\Delta_S) \supseteq S$ derives from the property~\eqref{eq:23}, Section~\ref{sec:length-orders}. 

For the leftmost equality, we first observe that the inclusion $\{\sigma\}\cup\LL^*(\sigma,S)\subseteq L(\sigma\cdot\Delta_S)$ is obvious. For the converse inclusion, the relation $L(\sigma\cdot \Delta_S) \subseteq \LL(\sigma\cdot \Delta_S)$ is immediate, and $\LL(\Delta_S) =
 S$ follows from~\eqref{eq:24}, hence we obtain $L(\sigma\cdot\Delta_S)\subseteq\{\sigma\}\cup S$. Thus, it is enough to prove that $\eta \notin L(\sigma\cdot \Delta_S)$ for every $\eta \in S$ such that $\ell(\sigma,\eta) > 2$.

Suppose, for the sake of contradiction, that
$\eta \leql \sigma\cdot \Delta_S$ for such an element~$\eta$.  Then
$\sigma$ and $\eta$ have the common $\leql$-upper
bound~$\sigma\cdot\Delta_S$. This implies that
$\ell(\sigma,\eta)<\infty$ according to the remarks of
Section~\ref{sec:length-orders}, and that
$\sigma\veel\eta\leql\sigma\cdot\Delta_S$. Furthermore,
$\sigma\cdot \eta\cdot \sigma \leql (\sigma\veel\eta)$ since
$\ell(\sigma,\eta)>2$, whence
$\sigma\cdot\eta\cdot\sigma\leql \sigma\cdot \Delta_S$, and thus
$\eta\cdot \sigma \leql \Delta_S$.  Since $\sigma \notin S$, the
latter relation is impossible.

  \ref{item:2}.\quad Let $\SS$ denote the set of simple elements
  of~$\bA$. Then for all $\eta \in S$, it follows from the relation
  $\ell(\sigma,\eta) < \infty$ that $\sigma \veel \eta \in \SS$, and
  from the relation $ \sigma\cdot \eta\leqr (\sigma \veel \eta) $ that
  $\sigma\cdot \eta \in \SS$.  Since
  $\sigma\cdot \Delta_S = \bigveel\{\sigma\cdot \eta \tq \eta \in
  S\}$, it implies that $\sigma\cdot\Delta_S\in\SS$.
\end{proof}

For the next result, we follow the same lines as in the proof
of~\cite[Proposition 4.9]{bestvina1999non}.

\begin{proposition}\label{pro:main}
  Let\/ $\bA$ be an irreducible Artin-Tits monoid, and let $\SS$ be
  the smallest Garside subset of\/~$\bA$.  Let $a$ and $b$ be
  non unit elements of\/~$\SS$.  If either\/ $\bA$ does not have
  spherical type, or $a = \Delta$, or $b \neq \Delta$, then there
  exists some normal sequence $(x_1,\ldots,x_k)$ such that $a = x_1$
  and $b = x_k$.
\end{proposition}

\begin{proof}
  Let $\bG$ be the Coxeter graph of~$\bA$ (see
  Section~\ref{sec:irred-artin-tits}), and let $d_\bG(\cdot,\cdot)$
  denote the graph metric in~$\bG$.  For each set
  $S \subseteq \Sigma$, we denote by $c(S)$ the number of connected
  components of $S$ in the graph $\mathbf{G}$, and by $d_\bG(\cdot,S)$
  the sum $\sum_{s \in S} d_\bG(\cdot,s)$.

  First, observe that if $L(b) = \Sigma$, then $b$ is a $\leql$-upper
  bound of~$\Sigma$, hence $\bA$ has spherical type and $\Delta=b$
  since $\Delta=\bigveel\Sigma$. According to our assumptions, this
  only occurs with $a = \Delta$, and then the normal sequence
  $(\Delta)$ has first and last letters $a$ and~$b$. Hence, we assume
  that $L(b)\neq \Sigma$.

  Next, since $a\neq\unit$, the set $R(a)$ is non empty. Hence, let
  $\rho \in R(a)$ be some node of $\mathbf{G}$.  Since $(a,\rho)$ is a
  normal sequence according to Corollary~\ref{cor:4-FC},
  point~\ref{item:13}, we assume without loss of generality that
  $a$ is an element of~$\Sigma$.

  Now, we put $S=L(b)$, and we discuss different cases. We take into
  account that $\bG$ is connected, which follows from the assumption
  that $\bA$ is irreducible \emph{via} Proposition~\ref{prop:1}. We
  also note that $\bigveel S$ exists, since $b$ is a $\leql$-upper
  bound of~$S$.

  \begin{enumerate}
  \item If $d_\bG(a,S) = 0$, then $S = \{a\}$, hence $L(b)\subseteq R(a)$
      and therefore the sequence $(a,b)$ is normal by
      Corollary~\ref{cor:4-FC}, point~\ref{item:13}.

  \item If $c(S) = \abs{S}$ and $d_\bG(a,S) > 0$. Then the set $S$ is
      $\bG$-independent and contains some vertex $\sigma \neq a$.

Let $n=d_\bG(\sigma,a)\geq1$, and let $s_0,s_1,\ldots,s_n$ be a path in
$\mathbf{G}$ such that $s_0 = a$, $s_n = \sigma$.  Since $S$ is
$\bG$-independent, the vertex $s_{n-1}$ does not belong to~$S$. Consider
the sets:
  \begin{align*}
    Q &= \{s_{n-1}\} \cup \{t \in S\tq d_\mathbf{G}(t,s_{n-1}) \geq 2\},
   & T &= \{t \in S \tq \ell(s_{n-1},t) < \infty\}.
  \end{align*}

 We note that $\Delta_T=\bigveel T$ exists since $\bigveel S$ exists.
  Lemma~\ref{lem:5-FC}, point~\ref{item:1}, proves:
  \begin{align*}
   L(s_{n-1}\cdot \Delta_T) &\subseteq Q ,& R(s_{n-1}\cdot \Delta_T) &\supseteq T \cup \{t \in S :
  \ell(s_{n-1},t) = \infty\} = S,
  \end{align*}
  and Lemma~\ref{lem:5-FC}, point~\ref{item:2}, proves that $s_{n-1}\cdot
  \Delta_T \in \SS$.  By construction, the set $Q$ is $\bG$-independent,
  whence $c(Q) = \abs{Q}$. Finally, we observe:
\begin{align*}
d_\bG(a,S) - d_\bG(a,Q) & =  \sum_{t \in S} \mathbf{1}\bigl({d_\mathbf{G}(t,s_{n-1})=1}\bigr)
 d_\mathbf{G}(a,t) - d_\mathbf{G}(a,s_{n-1}) \\
& \geq  d_\mathbf{G}(a,\sigma) - d_\mathbf{G}(a,s_{n-1}) = 1,
\end{align*}
and thus $d_\bG(a,Q)<d_\bG(a,S)$.

Since we have observed that $L(s_{n-1}\cdot\Delta_T)\subseteq Q$, we may
assume as an induction assumption on $d_\bG(a,S)$ the existence of a normal
sequence $(w_1,\ldots,w_j)$ such that $w_1 = a$ and $w_j = s_{n-1}\cdot
\Delta_T$.  Since we have also observed that $R(w_j)\supseteq S=L(b)$, it
follows from Corollary~\ref{cor:4-FC}, point~\ref{item:13}, that $(w_j,b)$
is a normal sequence, and thus the sequence $(w_1,\ldots,w_j,b)$ is also
normal according to Lemma~\ref{lem:2-FC}, point~\ref{item:10}.

\item If $c(S) < \abs{S}$, then, for every subset $T$ of~$\Sigma$, let
    $p(T)$ be the maximal cardinal of a connected component of $T$
    in~$\bG$, and let $q(T)$ be the number of connected components of $T$
    with this maximal cardinal.

Consider some connected component $S'$ of $S$ in $\bG$ of maximal cardinal,
and let $\eta_0\notin S$ be some neighbor of $S'$ in~$\mathbf{G}$. Such a
vertex $\eta_0$ exists since $S\neq\Sigma$. In addition, consider the sets:
\begin{align*}
  Q &= \{\eta_0\} \cup \{s \in S \tq d_\mathbf{G}(\eta_0,s) \geq 2\},
  &  T &= \{t \in S \tq \ell(\eta_0,t) < \infty\}.
\end{align*}

As in case~2 above, we note that $\Delta_T=\bigveel T$ exists since
$\bigveel S$ exists, and we apply Lemma~\ref{lem:5-FC} to obtain:
\begin{align*}
L(\eta_0\cdot\Delta_T) &\subseteq Q,&
R(\eta_0\cdot \Delta_T) &\supseteq S,&\eta_0\cdot
\Delta_T &\in \SS.
\end{align*}

It is obvious that $p(Q)\leq p(S)$.  Moreover, since $\eta_0$ is a neighbor
of~$S'$, $S'\cap Q\varsubsetneq S'$.  Therefore, since $S'$ has been chosen
of maximal cardinal among the connected components of~$S$ on the one hand,
and since $\abs{S'}\geq2$ by the assumption $c(S)<\abs{S}$ on the other
hand, at least one of the inequalities $p(Q)<p(S)$ and $q(Q)<q(S)$ holds.
It implies that $\bigl(p(Q),q(Q)\bigr)<\bigl(p(S),q(S)\bigr)$ holds in the
lexicographical order on~$\bbN\times\bbN$. We may thus assume as an
induction hypothesis (using Case 2 if $p(Q) = 1$) the existence of a normal
sequence $(w_1,\ldots,w_j)$ such that $w_1 = a$ and $w_j = \eta_0\cdot
\Delta_T$. As in case~2 above, we use that $L(b) = S \subseteq R(w_j)$ to
conclude that $(w_j,b)$, and thus $(w_1,\dots,w_j,b)$ are normal sequences.
\end{enumerate}

The proof is complete.
\end{proof}

Theorem~\ref{thr:1qqlknaa} follows at once from Proposition~\ref{pro:main}.

\subsection{Finite measures on the completion
of Artin-Tits monoids} \label{sec:finite-meas-bound}

\subsubsection{Boundary at infinity of an Artin-Tits monoid}
\label{sec:boundary-elements-an}

Let $\bA$ be an Artin-Tits monoid, with smallest Garside subset~$\SS$.
Elements of $\bA\setminus\{\unit\}$ are in bijection with normal sequences
according to the results recalled in Section~\ref{sec:normal-form-general}.
Hence, they identify with finite paths in the graph
$(\SS\setminus\{\unit\},\to)$. It is therefore natural to introduce
\emph{boundary elements} of $\bA$ as infinite paths in the very same graph.

In order to uniformly treat elements and boundary elements of the monoid, we
extend the notion of normal form as follows. If $x$ is an element of~$\bA$,
with height $k\geq1$ and normal form $(x_1,\dots,x_k)$, we put $x_j=\unit$
for all integers $j>k$. The sequence $(x_j)_{j\geq1}$ thus obtained is the
\emph{generalized
  normal form} of~$x$.

We say that an infinite sequence $(x_k)_{k\geq1}$ of elements of $\SS$ is
\emph{normal} if $x_k\to x_{k+1}$ holds for all $k\geq1$. Among infinite
normal sequences, those hitting $\unit$ at least once actually stay in
$\unit$ forever because of~\eqref{eq:21}, and these sequences correspond
bijectively to the usual elements of~$\bA$. And those normal sequences never
hitting $\unit$ correspond to the boundary elements of~$\bA$.

\begin{definition}
  \label{def:1poq}
  The \emph{generalized elements} of an Artin-Tits monoid\/ $\bA$ are
  the infinite normal sequences of simple elements of\/~$\bA$. Their
  set is called the \emph{completion of\/~$\bA$}, and is denoted
  by\/~$\bAbar$.  The \emph{boundary elements of\/~$\bA$} are the
  generalized elements that avoid the unit~$\unit$. Their set is
  called the \emph{boundary at infinity of\/~$\bA$}, or shortly
  \emph{boundary of\/~$\bA$}, and is denoted by\/~$\BA$. Identifying
  elements of\/~$\bA$ with their generalized normal form, we have
  thus: $\bAbar=\bA\cup\BA$.

  Both sets\/ $\BA$ and\/ $\bAbar$ are endowed with their natural
  topologies, as subsets of the countable product
  $\SS\times\SS\times\cdots$, where $\SS$ is the smallest Garside
  subset of\/~$\bA$, and with the associated Borelian \slgb s.

For every element $x$ of\/~$\bA$, of height $k$ and with normal form\/
$(x_1,\dots,x_k)$, the \emph{Garside cylinder} of base $x$ is the open and
closed subset of\/ $\bAbar$
defined by:
\begin{gather*}
\CC_x=\{y=(y_j)_{j\geq1}\in\bAbar\tq y_1=x_1,\dots,y_k=x_k\}.
\end{gather*}
Note that $\CC_{\unit}= \{\unit\}$.

The partial ordering\/ $\leql$ is extended on\/ $\bAbar$ by putting, for
$x=\seq xk$ and $y=\seq yk$:
\begin{gather*}
  x\leql y\iff
(\forall k\geq 1\quad\exists j\geq1\quad x_1\cdot\ldots\cdot x_k \leql
y_1\cdot\ldots\cdot y_j).
\end{gather*}
For every $x\in\bA$, the \emph{visual cylinder} $\up x\subseteq\BA$ and the
\emph{full visual cylinder} $\Up x\subseteq\bAbar$ are:
\begin{align*}
  \up x&=\{y\in\BA\tq x\leql y\},&\Up x&=\{y\in\bAbar\tq x\leql y\}.
\end{align*}
\end{definition}

Both spaces $\bAbar$ and $\BA$ are metrisable and compact, and $\bAbar$ is
the topological closure of~$\bA$. More generally, for any $x\in\bA$, the full
visual cylinder $\Up x$ is the topological closure in $\bAbar$ of
$\{y\in\bA\tq x\leql y\}$.

The collection $\bigl\{\CC_x\tq x\in\bA\bigr\} \cup \{\emptyset\}$ is closed
under intersection and generates the \slgb\ on~$\bAbar$.
Therefore, any finite measure $\nu$ on $\bAbar$
is entirely determined by the family of
values~$\bigl(\nu(\CC_x)\bigr)_{x\in\bA}$\,.

The following lemma provides an alternative description of the relation
$x\leql y$ for $x\in\bA$ and $y\in\bAbar$.

\begin{lemma}
  \label{lem:2}
  Let $x$ be an element of an Artin-Tits monoid\/~$\bA$, and let
  $k=\height(x)$.  Let $y=(y_j)_{j\geq1}$ be a generalized element
  of\/~$\bA$. Then $x\leql y$ in\/ $\bAbar$ if and only if
  $x\leql y_1\cdot\ldots\cdot y_k$ in\/~$\bA$.
  \VJ{In particular, if $y \in \bA$, then the relation $x \leql y$ holds
  in\/~$\bA$ if and only if it holds in\/~$\bAbar$.}
\end{lemma}

\begin{proof}
  Let $x\in\bA$, $k=\height(x)$ and $y\in\bAbar$ with
  $y=(y_j)_{j\geq1}$. Assume that $x\leql y$. Then there is an integer
  $j\geq k$ such that $x\leql y'$, with $y'= y_1\cdot\ldots\cdot
  y_j$. Applying Lemma~\ref{lem:1}, point~\ref{item:3}, we obtain
  that $x\leql y_1\cdot\ldots\cdot y_k$, which is what we wanted to prove. The
  converse implication also follows from Lemma~\ref{lem:1}:
  if $x\leql y_1\cdot\ldots\cdot y_k$ then $x_1\cdot\ldots\cdot x_j \leql
  y_1\cdot\ldots\cdot y_j$ for all $j \leq k$, and
  $x_1\cdot\ldots\cdot x_j = x \leql y_1\cdot\ldots\cdot y_k \leql
  y_1\cdot\ldots\cdot y_j$ for all $j > k$.
\end{proof}

\subsubsection{Relating Garside cylinders and visual cylinders}
\label{sec:relat-gars-cylind}

Visual cylinders are natural from the point of view of the algebraic
structure of the monoid, whereas Garside cylinders have a more operational
presentation as they rely on the normal form of elements. Since both points
of view are interesting, it is important to relate the two kinds of cylinders
to one another, which we do now.

% New part, used to avoid useless notation

Given any two Garside cylinders
$\CC_x$ and $\CC_y$, either $\CC_x \cap \CC_y = \emptyset$
or $\CC_x \subseteq \CC_y$ or $\CC_y \subseteq \CC_x$.
Furthermore, for all $x \in \bA$, the Garside cylinder
$\CC_x$ is contained into the full visual cylinder
$\Up x$.
Consequently, for every open set $\calA \subseteq \bAbar$,
\ie, every union of sets of the form $\Up x$ with $x \in \bA$,
there exists a unique set $\calB$ such that
$\calA$ is the disjoint union of the Garside cylinders
$\CC_x$ for $x \in \calB$.
We say that $\calB$ is a \emph{Garside base}
of $\calA$.

A case of special interest is that of the set
$\calA = \Up x$, where $x \in \bA$.

\begin{definition}
  \label{prop:2qqpjaazlx}
  Let\/~$\bA$ be an Artin-Tits monoid. For every $x\in\bA$, we denote by $\bA[x]$ the Garside base of the set
  $\Up x$, i.e., the unique set such that
  the full cylinder\/ $\Up x$
  has the following decomposition as a disjoint union:
\begin{gather}
\label{eq:1qaz}
  \Up x=\bigcup_{y\in\bA[x]}\CC_y\,.
\end{gather}
\end{definition}

Note that, equivalently, we might define the set $\bA[x]$ as
\begin{gather}
\label{eq:18}
  \bA[x]=\bigl\{y \in \bA \/~\cap \Up x \tq \forall z \in \Up x \quad \CC_y \subseteq \CC_z \implies y = z\bigr\}.
\end{gather}

It follows from Lemma~\ref{lem:1qpasoa} below that
the set $\bA[x]$ is finite,
and therefore that the disjoint union of~\eqref{eq:1qaz} is finite.

\begin{lemma}
  \label{lem:1qpasoa}
  Let\/ $\bA$ be an Artin-Tits monoid.  For every $x\in\bA$ and $y\in\bA[x]$, it holds:
  $\height(y)\leq\height(x)$.
\end{lemma}

\begin{proof}
From $x\leql y$ with $y=(y_j)_{j\geq1}$, we know from Lemma~\ref{lem:2}
that $x\leq y_1\cdot\ldots\cdot y_k$ with $k=\height(x)$, whence the result.
\end{proof}

% 
% Let us introduce some notations, given an Artin-Tits monoid
% $\bA=\bA(\Sigma,\ell)$. For $x\in\bA$ and $\xi\in\bAbar$
% with $\xi=(y_j)_{j\geq1}$\,, and such that $x\leql\xi$, we let $\beta_x(\xi)$
% denote the element of $\bA$ defined by:
% \begin{align*}
%   \beta_x(\xi)&=y_1\cdot\ldots\cdot y_{T_x(\xi)}\,,&
% \text{with } T_x(\xi)&=\min\{j\geq1\tq x\leql y_1\cdot\ldots\cdot
% y_j\}.
% \end{align*}
% 
% \begin{lemma}
%   \label{lem:1qpasoa}
%   Let\/ $\bA$ be an Artin-Tits monoid.  For every $x\in\bA$ and
%   $\xi\in\up x$, it holds:
%   $\height\bigl(\beta_x(\xi)\bigr)\leq\height(x)$.
% \end{lemma}
% 
% \begin{proof}
% From $x\leql\xi$ with $\xi=(y_j)_{j\geq1}$, we know from Lemma~\ref{lem:2}
% that $x\leq y_1\cdot\ldots\cdot y_k$ with $k=\height(x)$, whence the result.
% \end{proof}
% 
% It follows from Lemma~\ref{lem:1qpasoa} that the set of elements of the form
% $\beta_x(\xi)$ is finite for every $x\in\bA$. Let $\bA[x]$ denote this finite
% set:
% \begin{gather}
% \label{eq:18}
%   \bA[x]=\bigl\{\beta_x(\xi)\tq\xi\in\up x\bigr\}.
% \end{gather}

If $x\in\SS$, then every $y\in\bA[x]$ has height exactly~$1$, and thus
$\bA[x]=\{y\in\SS\tq x\leql y\}$. However, if $x$ is of height at least~$2$,
elements $y\in\bA[x]$ may satisfy the \emph{strict} inequality
$\height(y)<\height(x)$. This is illustrated in Section~\ref{sec:an-example}.
This situation, valid for general Artin-Tits monoids, contrasts with the case
of specific monoids such as heap monoids and braid monoids, or more generally
monoids of type FC; this is investigated in Section~\ref{sec:an-example}.

% 
% \begin{proposition}
%   \label{prop:2qqpjaazlx}
%   Let\/~$\bA$ be an Artin-Tits monoid. For every $x\in\bA$, the full visual
%   cylinder\/ $\Up x$
%   has the following decomposition as a finite and
%   disjoint union:
% \begin{gather}
% \label{eq:1qaz}
%   \Up x=\bigcup_{y\in\bA[x]}\CC_y\,.
% \end{gather}
% \end{proposition}
% 
% \begin{proof}
%   For every $\xi\in\Up x$,
%   one has $\xi\in\CC_y$ for $y=\beta_x(\xi)$,
%   which proves the $\subseteq$ inclusion. Note that for every
%   $y\in\bA[x]$ and $\xi\in\CC_y$, one has $\xi\in\Up x$
%   and
%   $\beta_x(\xi)=y$ by construction. This completes the proof of~\eqref{eq:1qaz}.
%   And it also proves that any two
%   Garside cylinders $\CC_y$ and $\CC_{y'}$ for $y,y'\in\bA[x]$ with
%   $y\neq y'$ are disjoint, completing the proof.
% \end{proof}

In addition, the very definition of $\bA[x]$
has immediate consequences when considering finite
measures on $\bAbar$.

\begin{proposition}
  \label{cor:1aapkla}
  Let $\nu$ be a finite measure on the completion
  of an Artin-Tits
  monoid\/~$\bA$. Then the measures of full visual cylinders and of Garside
  cylinders are related by the following formulas, for $x$ ranging
  over\/~$\bA$:
  \begin{gather}
    \label{eq:1qqapzs}
\nu(\Up x)=\sum_{y\in\bA[x]}\nu(\CC_y).
  \end{gather}
\end{proposition}

Assume given a finite measure $\nu$ for which the values $\nu(\Up x)$
are known---this will hold indeed for a family of probability measures that
we shall construct later in Sections~\ref{sec:mult-meas-bound}
and~\ref{sec:appl-posit-braids}. Let $f,h:\bA\to\bbR$ be the functions
defined by $f(x)=\nu(\Up x)$
and $h(x)=\nu(\CC_x)$. Then, in view of~\eqref{eq:1qqapzs}, expressing the
quantities $h(x)=\nu(\CC_x)$ by means of the values $f(x)=\nu(\Up x)$
amounts to inverting the linear operator $\T^*$ defined by:
\begin{gather}
\label{eq:1plaopa}
  \T^*h(x)=\sum_{y\in\bA[x]}h(y).
\end{gather}

Giving an explicit expression for the inverse of $\T^*$ is the topic of the
next section.

\subsubsection{Graded Möbius transform}
\label{sec:mobi-transf-appl}

Measure-theoretic reasoning provides a hint for guessing the right
transformation. Let $\nu$ be a finite measure defined on the boundary $\BA$
of an Artin-Tits monoid $\bA=\bA(\Sigma,\ell)$. For $x$ an element of~$\bA$,
we put:
\begin{gather}
\label{eq:16}
\begin{aligned}
\E(x)&=\bigl\{u\in\bA\tq \height(x\cdot
u)\leq\height(x)\bigr\}\setminus\{\unit\},\\
\D(x)&=\{\text{$\leql$-minimal elements of $\E(x)$}\}.
\end{aligned}
\end{gather}
The set $\E(x)$ may be empty, in which case $\D(x)=\emptyset$ as well.

We claim that, for every $x\in\bA$,  the following equality of sets holds:
\begin{gather}
\label{eq:1papkama}
  \CC_x=\Up x\setminus\bigcup_{u\in\D(x)}\Up(x\cdot u).
\end{gather}

\begin{proof}[Proof of~\eqref{eq:1papkama}.]
Let $(x_1,\dots,x_k)$ be the normal form of~$x$, and let $y\in\CC_x$. Let $(y_j)_{j\geq1}$ be the extended normal form of~$y$, with $y_j=x_j$ for all $j\in\{1,\dots,k\}$. Hence $x\leql (y_1\cdot\ldots\cdot y_j)$ for all $j\geq k$, and thus $y\in\Up x$ according to the definition of the partial ordering $\leql$ on~$\bAbar$ given in Definition~\ref{def:1poq}. Let $u\in\bA$ be such that $\height(x\cdot u)\leq\height(x)$ and  $y\in\Up(x\cdot u)$. It follows from Lemma~\ref{lem:2} that $x\cdot u\leql y_1\cdot\ldots\cdot y_k$, and since $y_1\cdot\ldots\cdot y_k=x$ it implies that $u=\unit$. Hence $u\notin\D(x)$ and the $\subseteq$ inclusion in~\eqref{eq:1papkama} follows.

Conversely, let $y$ be an element of the right-hand set of~\eqref{eq:1papkama}, and let $(y_j)_{j\geq1}$ be the generalized normal form of~$y$. Let us prove that $y_1\cdot\ldots\cdot y_k=x$, which entails that $y\in\CC_x$\,. For this, we use the characterization given in Lemma~\ref{lem:1}, point~\ref{item:3}:
\begin{gather*}
  y_1\cdot\ldots\cdot y_k=\bigveel U,\qquad\text{with }U=\{z\in S^k\tq z\leql y\}.
\end{gather*}

We have $x\in U$ since $x\leql y$ and $\height(x)=k$ by assumption. Let $z=\bigveel U$ and, seeking a contradiction, assume that $x\neq z$. Then $z=x\cdot u$ with $u\neq\unit$ and $\height(x\cdot u)\leq\height(x)$, thus $u\in\D(x)$. But then $y\in\Up (x\cdot u)$, contradicting the assumption on~$y$.
\end{proof}

We observe that for any elements $u,u'\in\D(x)$, one has that $(x\cdot
u)\veel(x\cdot u')$ exists in $\bA$ if and only if $u\veel u'$ exists
in~$\bA$, if and only if $\Up(x\cdot u)\,\cap\Up(x\cdot u')\neq\emptyset$.
If these conditions are fulfilled, we have:
\begin{gather*}
  \Up(x\cdot u)\,\cap\Up(x\cdot u')=\Up\bigl(x\cdot(u\veel u')\bigr),
\end{gather*}
and the latter generalizes to intersections of the form $\Up(x\cdot
u_1)\,\cap\dots\cap\Up(x\cdot u_k)$ for $u_1,\dots,u_k\in\D(x)$.
Taking the measure of both members in~\eqref{eq:1papkama}, and applying the
inclusion-exclusion principle, which is basically the essence of Möbius
inversion formulas, we obtain thus:
\begin{equation*}
  \nu(\CC_x)=\nu(\Up x)-\sum_{D\subseteqlub\D(x),\;D\neq\emptyset}
  (-1)^{\abs{D}+1}\nu\bigl(\Up\bigl(x\cdot\bigveel D\bigr)\bigr),
\end{equation*}
where $D\subseteqlub \D(x)$ means that $D$ is a
              subset of $\D(x)$ \emph{and} that $\bigveel D$ exists
              in~$\bA$. Observing that $\emptyset\subseteqlub\D(x)$ with $\bigveel \emptyset=\unit$, we
  get:
\begin{align}
\label{eq:1ppplalmka}
  \nu(\CC_x)&=\sum_{D\subseteqlub\D(x)}(-1)^{\abs{D}}\nu\bigl(\Up\bigl(x\cdot\bigveel
              D\bigr)\bigr).
\end{align}

We are thus brought to introduce the following definition.

\begin{definition}
  \label{def:1qqpa}
Let\/ $\bA$ be an Artin-Tits monoid. The\/ \emph{graded Möbius
  transform} of a function $f:\bA\to\bbR$ is the function $\T
f:\bA\to\bbR$, defined as follows for every $x\in\bA$:
\begin{align*}
  \T f(x)&=\sum_{D\subseteqlub\D(x)}(-1)^{\abs{D}}f\bigl(x\cdot\bigveel D\bigr),
\end{align*}
where $\D(\cdot)$ has been defined in\/~\eqref{eq:16} and $D\subseteqlub
\D(x)$ means that $D$ is a subset of $\D(x)$ such that\/ $\bigveel D$ exists
in~$\bA$.

The\/ \emph{inverse graded Möbius transform} of a function $h:\bA\to\bbR$ is
the function $\T^\ast h:\bA\to\bbR$, defined as follows for every $x\in\bA$:
\begin{align*}
  \T^\ast h(x)&=\sum_{y\in\bA[x]}h(y),
\end{align*}
where $\bA[\cdot]$ has been defined in\/~\eqref{eq:18}.
\end{definition}

Section~\ref{sec:an-example} below details the expression of the graded
Möbius transform for some specific examples. Before that, we bring additional
information on the range of summation in the definition of the graded Möbius
transform.

\begin{lemma}
  \label{lem:9}
  In an Artin-Tits monoid\/~$\bA$, for any $x\in\bA$, one has
  $\D(x)=\D(u)$, where $u$ is the \emph{last} simple element in the
  normal form of~$x$.
\end{lemma}

\begin{proof}
Let $(x_1,\dots,x_k)$ be the normal form of~$x$. We first show the following:
\begin{gather}
  \label{eq:17}
\forall y\in\SS\quad \D(x)\;\cap\down y=\emptyset\iff\D(x_k)\;\cap\down y=\emptyset,
\end{gather}
where $\down y=\{z\in\bA\tq z\leql y\}$. Indeed, for $y\in\SS$, one has
according to Lemma~\ref{lem:1}:
\begin{align*}
  x_k\to y&\iff x=\bigveel\{z\in\SS^k\tq z\leql x\cdot y\}
  \\&
\iff \E(x)\;\cap\down y=\emptyset
\iff\D(x)\;\cap\down y=\emptyset.
\end{align*}

Applying the above equivalence with $x=x_k$ on the one hand, and with $x$ on
the other hand, yields~\eqref{eq:17}.  From~\eqref{eq:17}, and since $\D(x)$
and $\D(x_k)$ are two $\leql$-antichains, we deduce that $\D(x)=\D(x_k)$, as
expected.
\end{proof}

With Definition~\ref{def:1qqpa}, we reformulate~\eqref{eq:1ppplalmka} as
follows.

\begin{proposition}
\label{prop:2}
  Consider a finite measure $\nu$ on the completion
  of an Artin-Tits
  monoid\/~$\bA$. For every $x\in\bA$, let $f(x)=\nu(\Up x)$
  and let
  $h(x)=\nu(\CC_x)$. Then $h$ is the graded Möbius transform of~$f$.

  If $\nu$ and $\nu'$ are two finite measures such that
  $\nu(\Up x)=\nu'(\Up x)$
  for all $x\in\bA$, then $\nu=\nu'$.
\end{proposition}

\begin{proof}
  The first part of the statement is a simple rephrasing
  of~\eqref{eq:1ppplalmka}. For the second part, if
  $\nu(\Up x)=\nu'(\Up x)$ for all $x\in\bA$, then
  $\nu(\CC_x)=\nu'(\CC_x)$ for all $x\in\bA$. We have already observed
  that this implies $\nu=\nu'$.
\end{proof}

\VJ{Propositions~\ref{prop:2qqpjaazlx} and~\ref{prop:2} show}
that the transformations $\T$ and $\T^*$ are inverse of each other when
operating on functions $f$ and $h$ of the form $f(x)=\nu(\Up x)$
and $h(x)=\nu(\CC_x)$. We actually have the following more general result, of
a purely combinatorial nature.

\begin{theorem}
  \label{thr:1mobiusyta}
  Let\/ $\bA$ be an Artin-Tits monoid. Then the graded Möbius
  transform\/ $\T$ and the inverse graded Möbius transform\/
  $\T^\ast$ are two endomorphisms inverse of each
  other, defined on the space of functions $\bA\to\bbR$.
\end{theorem}
We first prove the following lemma.

\begin{lemma}
  \label{lem:1mobiusdoisjq}
  Let\/ $\bA$ be an Artin-Tits monoid, let $x,y\in\bA$, and let:
\begin{gather*}
  \B(x,y)=\bigl\{D\subseteqlub\D(x)\tq y\in\bA\bigl[x\cdot\bigveel D\bigr]\bigr\}.
\end{gather*}
Then:
\begin{gather}
\label{eq:1}
  \sum_{D\in\B(x,y)}(-1)^{\abs{D}}=\un(x=y).
\end{gather}
\end{lemma}

\begin{proof}
For $x,y\in\bA$, we put:
\begin{align*}
  K(x,y)&=\bigl\{z\in\bA\tq x\cdot z\leql y\text{ and }\height(x\cdot
          z)\leq\height(x)\bigr\},\\
\K(x,y)&=\bigl\{D\subseteqlub\D(x)\tq\bigveel D\in K(x,y)\bigr\}.
\end{align*}
Then we claim:
\begin{gather}
  \label{eq:2}
\sum_{D\in\K(x,y)}(-1)^{\abs{D}}=
\begin{cases}
  1,&\text{if } \VJ{y \in \CC_x},\\
0,&\text{otherwise.}
\end{cases}
\end{gather}

Let us first prove~\eqref{eq:2}. If $x\leql y$ does not hold, then
$K(x,y)=\K(x,y)=\emptyset$ and~\eqref{eq:2} is true. Hence we assume that
$x\leql y$. Let $(y_1,\dots,y_k)$ be the normal form of~$y$. We put
$y'=y_1\cdot\ldots\cdot y_j$, where $j=\min\{i\geq 1\tq x\leql
y_1\cdot\ldots\cdot y_i\}$, and thus $x\leql y'$ holds, with equality if and
only if
\VJ{$y \in \CC_x$}. We observe that $K(x,y)$ is a lattice: this follows from
Lemma~\ref{lem:1}; and thus $\K(x,y)$ is a Boolean lattice.

We discuss two cases.  If $x=y'$, then $K(x,y)=\{\unit\}$ and
$\K(x,y)=\{\emptyset\}$, and thus~\eqref{eq:2} is true.  However, if $x\neq
y'$, then Lemma~\ref{lem:2} proves that $\height(y') \leq \height(x)$. Hence,
the set $K(x,y')$ contains a minimal non unit element~$z$. This element
satisfies $\{z\}\in\K(x,y)$, which is thus a non trivial Boolean lattice.
Hence~\eqref{eq:2} is true in all cases.

We now come to the proof of~\eqref{eq:1}. The case where $x=y$ or $\neg (x
\leql y)$ holds is easily settled, hence we assume below that $x \lel y$. In
particular, observe that $y \neq \unit$. Let $(y_1,\dots,y_k)$ be the normal
form of~$y$, and let $\ytilde=y_1\cdot\ldots\cdot y_{k-1}$\,. Then we have
$\K(x,\ytilde)\subseteq\K(x,y)$ and $\B(x,y)=\K(x,y)\setminus\K(x,\ytilde)$.
Therefore, using~\eqref{eq:2}, we have:
\begin{align}
\label{eq:19}
  \sum_{D\in\B(x,y)}(-1)^{\abs{D}}&=A(x,y)-A(x,\ytilde),\\
\notag
&\text{with }
A(u,v)=
\begin{cases}
  1,&\text{if } {v \in \CC_u},\\
0,&\text{otherwise.}
\end{cases}
\end{align}

If
$\ytilde$ belongs to the Garside cylinder~$\CC_x$, so does~$y$. Conversely,
since $x \neq y$, if
$y \in \CC_x$ then the normal form of $x$ is a strict prefix of the normal
form of~$y$, and therefore $\ytilde \in \CC_x$. Together with~\eqref{eq:19}
this completes the proof of~\eqref{eq:1}.
\end{proof}

\begin{proof}[Proof of Theorem~\ref{thr:1mobiusyta}.]
  Let $h:\bA\to\bbR$ be given. We compute, exchanging the order of
  summation:
\begin{gather*}
  \T\bigl(\T^*
  (h)\bigr)(x)=\sum_{y\in\bA}h(y)\Bigl(\sum_{D\in\B(x,y)}(-1)^{\abs{D}}\Bigr)=h(x),
\end{gather*}
where we have put $\B(x,y)=\bigl\{D\subseteqlub\D(x)\tq y\in\bA[x\cdot\bigvee
D]\bigr\}$, and where the last equality follows from
Lemma~\ref{lem:1mobiusdoisjq}.

Let $E$ denote the vector space of functions $\bA\to\bbR$. Then $\T$ is an
endomorphism of~$E$, and it remains to prove that $\T$ is injective. Let $f$
be a non zero function $f \in E$, let $k$ be the smallest integer such that
$f$ is non zero on~$\SS^k$, and let $x$ be a maximal element of $\SS^k$ such
that $f(x) \neq 0$. For all non empty sets $D \subseteqlub \D(x)$, we have $x
\lel x \cdot \bigveel D$, and $x \cdot \bigveel D \in \SS^\ell$ for some
$\ell \leq k$, hence $f(x \cdot \bigveel D) = 0$. It follows that $\T(f)(x) =
f(x) \neq 0$, which completes the proof.
\end{proof}

\subsubsection{Examples and particular cases}
\label{sec:an-example}

In this section, we first develop the expression of the graded Möbius
transform for monoids of type FC, pointing out the simplification that arises
in this case. By contrast, we introduce then the example of an Artin-Tits
monoid which is not of type FC. Finally, we give the expression of the graded
Möbius transform and its inverse evaluated on the simple elements of a
general Artin-Tits monoid.

Let us first examine the case of Artin-Tits monoids of type FC.

\begin{proposition}
  \label{prop:5}
Let\/ $\bA$ be an Artin-Tits monoid of type FC. Then, for every $x\in\bA$,
the sets\/ $\bA[x]$ and $\D(x)$ have the following expressions:
\begin{align*}
  \bA[x]&=\{y\in\bA\tq x\leql y\AND\height(y)\leq\height(x)\}
  \\& =\{y\in\bA\tq x\leql y\AND\height(y)=\height(x)\} \\
\D(x)&=\{\sigma\in\Sigma\tq\height(x\cdot\sigma)=\height(x)\}.
\end{align*}
\end{proposition}

\begin{proof}
Let $x,y \in \bA$ such that
$x \leql y$ and $\height(y) \leq \height(x)$. Let $k = \height(y)$, and let
$(y_1,\ldots,y_k)$ be the normal form of $y$.
The following statement is proved in ~\cite[Prop.~2.95]{juge16:_combin}:
  \begin{intermediate}
    {\normalfont$(\ddag)$}
    Let $\bA$ be an Artin-Tits monoid of type FC.
    For all $a,b \in \bA$ such that $a \leql b$, we have $\height(a) \leq \height(b)$.
  \end{intermediate}
% It follows from $(\ddag)$ that $\height(x) = \height(y) = k$ and that $T_x(y)
% = k$,
% \ie, that $x \leql y_1 \cdot \ldots \cdot y_k$ and that $\neg (x \leql y_1
% \cdot \ldots \cdot y_j)$ for all $j \leq k-1$. It follows that $\beta_x(y) =
% y_1 \cdot \ldots \cdot y_k = y \in \bA[x]$. This proves the inclusion $\bA[x]
% \supseteq \{y\in\bA\tq x\leql y\AND\height(y)\leq\height(x)\}$, whereas the
% converse inclusion is immediate.
It follows from $(\ddag)$ that $\height(x) = \height(y) = k$, \ie, that $x \leql y_1 \cdot \ldots \cdot y_k$ and that $\neg (x \leql y_1
\cdot \ldots \cdot y_j)$ for all $j \leq k-1$. This proves the inclusion $\bA[x]
\supseteq \{y\in\bA\tq x\leql y\AND\height(y)\leq\height(x)\}$, whereas the
converse inclusion follows from Lemma~\ref{lem:1qpasoa}. Then, the claim $(\ddag)$ also proves the
inclusion $\{y\in\bA\tq x\leql y\AND\height(y)\leq\height(x)\} \subseteq
\{y\in\bA\tq x\leql y\AND\height(y) =\height(x)\}$, and the converse
inclusion is immediate.

Finally, assume that $\D(x)$ is non empty, and let $z$ be an element of
$\D(x)$. Since $z\neq\unit$, we can write  $z = \sigma \cdot z'$ for some
$\sigma \in \Sigma$ and $z' \in \bA$. Since $x \leql x \cdot \sigma \leql x
\cdot z$ and $k = \height(x) = \height(x \cdot z)$, it follows from $(\ddag)$
that $k = \height(x \cdot \sigma)$, \ie, that $x \cdot \sigma \in \bA[x]$. By
minimality of $z$, we must have $z = \sigma$. It follows that $\D(x)
\subseteq \{\sigma \in \Sigma \tq x \cdot \sigma \in \bA[x]\}$, whereas the
converse inclusion is immediate. Recalling that $\bA[x]=\{y\in\bA\tq x\leql
y\AND\height(y)=\height(x)\}$, the proof is complete.
\end{proof}

Recalling that $\D(x)=\D(u)$ according to Lemma~\ref{lem:9}, where $u$ is the
last element in the normal form of~$x$, it follows that the graded Möbius
transform, in a monoid of type FC, has the following form:
\begin{gather*}
  \T f(x)=\sum_{D\subseteqlub\Sigma\tqs u\cdot(\bigveel
    D)\in\SS}(-1)^{\abs{D}}f\bigl(x\cdot \bigveel D\bigr).
\end{gather*}

Compared to monoids of type FC, general Artin-Tits monoids have the pathology
that the smallest Garside subset $\SS$ is not necessarily closed by left
divisibility. It entails that the strict inequality
$\height(x\cdot\sigma)<\height(x)$ may occur. This is illustrated in the
following example, where we point out two consequences of this fact.

Let $\bA$ be the Artin-Tits monoid defined below, known as the Artin-Tits groups of type affine~$\tilde{A}_2$---see, for instance,~\,\cite{digne06}. It admits the set $\SS$
described below as smallest Garside subset:
\begin{align*}
\bA &= \langle a,b,c \;\big|\; aba=bab,\ bcb=cbc,\ cac=aca \rangle  \\
\SS&= \{\unit,a,b,c,ab,ac,ba,bc,ca,cb,aba,bcb,cac,abcb,bcac,caba\}.
\end{align*}

Hence, the normal forms of $x=abc$ and $y=abcb$ are, respectively, $(ab,c)$
and $(abcb)$, which shows that $\height (x\cdot b)=1<\height(x)=2$. The Hasse
diagram of $(\SS,\leql)$ is depicted on Figure~\ref{fig:atmbiz}.

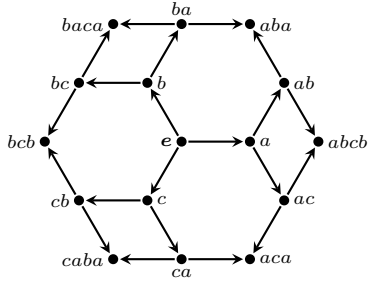
\begin{figure}
\centering
\begin{tikzpicture}[scale=0.9]
\node[anchor=east] at (0,0) {\scriptsize $\unit$};
\node[anchor=west] at (1,0) {\scriptsize $a$};
\node[anchor=west] at (-0.5,0.866) {\scriptsize $b$};
\node[anchor=west] at (-0.5,-0.866) {\scriptsize $c$};
\node[anchor=east] at (-1.5,0.866) {\scriptsize $bc$};
\node[anchor=east] at (-1.5,-0.866) {\scriptsize $cb$};
\node[anchor=west] at (1.5,0.866) {\scriptsize $ab$};
\node[anchor=west] at (1.5,-0.866) {\scriptsize $ac$};
\node[anchor=west] at (2,0) {\scriptsize $abcb$};
\node[anchor=east] at (-2,0) {\scriptsize $bcb$};
\node[anchor=east] at (-1,1.732) {\scriptsize $baca$};
\node[anchor=south] at (0,1.732) {\scriptsize $ba$};
\node[anchor=west] at (1,1.732) {\scriptsize $aba$};
\node[anchor=east] at (-1,-1.732) {\scriptsize $caba$};
\node[anchor=north] at (0,-1.732) {\scriptsize $ca$};
\node[anchor=west] at (1,-1.732) {\scriptsize $aca$};

\draw[draw=black,fill=black,very thick] (0,0) circle (0.05);
\draw[draw=black,fill=black,very thick] (1,0) circle (0.05);
\draw[draw=black,fill=black,very thick] (-0.5,0.866) circle (0.05);
\draw[draw=black,fill=black,very thick] (-0.5,-0.866) circle (0.05);
\draw[draw=black,fill=black,very thick] (-1.5,-0.866) circle (0.05);
\draw[draw=black,fill=black,very thick] (-1.5,0.866) circle (0.05);
\draw[draw=black,fill=black,very thick] (1.5,0.866) circle (0.05);
\draw[draw=black,fill=black,very thick] (1.5,-0.866) circle (0.05);
\draw[draw=black,fill=black,very thick] (2,0) circle (0.05);
\draw[draw=black,fill=black,very thick] (-2,0) circle (0.05);
\draw[draw=black,fill=black,very thick] (-1,1.732) circle (0.05);
\draw[draw=black,fill=black,very thick] (0,1.732) circle (0.05);
\draw[draw=black,fill=black,very thick] (1,1.732) circle (0.05);
\draw[draw=black,fill=black,very thick] (-1,-1.732) circle (0.05);
\draw[draw=black,fill=black,very thick] (0,-1.732) circle (0.05);
\draw[draw=black,fill=black,very thick] (1,-1.732) circle (0.05);

\draw[->,draw=black,thick,>=stealth] (0.1,0) -- (0.9,0);
\draw[->,draw=black,thick,>=stealth] (-0.05,0.0866) -- (-0.45,0.7794);
\draw[->,draw=black,thick,>=stealth] (1.05,0.0866) -- (1.45,0.7794);
\draw[->,draw=black,thick,>=stealth] (1.55,0.7794) -- (1.95,0.0866);
\draw[->,draw=black,thick,>=stealth] (1.45,0.9526) -- (1.05,1.6454);
\draw[->,draw=black,thick,>=stealth] (-0.45,0.9526) -- (-0.05,1.6454);
\draw[->,draw=black,thick,>=stealth] (0.1,1.732) -- (0.9,1.732);
\draw[->,draw=black,thick,>=stealth] (-0.1,1.732) -- (-0.9,1.732);
\draw[->,draw=black,thick,>=stealth] (-0.6,0.866) -- (-1.4,0.866);
\draw[->,draw=black,thick,>=stealth] (-1.55,0.7794) -- (-1.95,0.0866);
\draw[->,draw=black,thick,>=stealth] (-1.45,0.9526) -- (-1.05,1.6454);
\draw[->,draw=black,thick,>=stealth] (-0.05,-0.0866) -- (-0.45,-0.7794);
\draw[->,draw=black,thick,>=stealth] (1.05,-0.0866) -- (1.45,-0.7794);
\draw[->,draw=black,thick,>=stealth] (1.55,-0.7794) -- (1.95,-0.0866);
\draw[->,draw=black,thick,>=stealth] (1.45,-0.9526) -- (1.05,-1.6454);
\draw[->,draw=black,thick,>=stealth] (-0.45,-0.9526) -- (-0.05,-1.6454);
\draw[->,draw=black,thick,>=stealth] (0.1,-1.732) -- (0.9,-1.732);
\draw[->,draw=black,thick,>=stealth] (-0.1,-1.732) -- (-0.9,-1.732);
\draw[->,draw=black,thick,>=stealth] (-0.6,-0.866) -- (-1.4,-0.866);
\draw[->,draw=black,thick,>=stealth] (-1.55,-0.7794) -- (-1.95,-0.0866);
\draw[->,draw=black,thick,>=stealth] (-1.45,-0.9526) -- (-1.05,-1.6454);
\end{tikzpicture}
\caption{Hasse diagram of $(\SS,\leql)$ for the Artin-Tits monoid
  $\bA=\langle a,b,c\ |\ aba=bab,\ bcb=cbc,\ cac=aca\rangle$. In this example, $\SS$~is not closed by left divisibility since, for instance,
  $caba\in\SS$ whereas $cab\notin\SS$.}
\label{fig:atmbiz}
\end{figure}

As a first consequence of $\SS$ not being closed by left divisibility, the
Garside cylinders $\CC_y$ may not be disjoint, for $x\in\bA$ fixed and for
$y$ ranging over $\{z\in\bA\tq x\leql z\AND \height(z)\leq\height(x)\}$.
Indeed, consider again $x=abc$, of height~$2$, and $y=abcb$ and $y'=abcba$,
of height $1$ and $2$ respectively, the latter since the normal form of $y'$
is $(abcb,a)$. Then both $y$ and $y'$ belong to the set described above; yet,
$\CC_{y}\subseteq\CC_{y'}$ and thus $\CC_y\cap\CC_{y'}\neq\emptyset$.
Whereas, combining Proposition~\ref{prop:2qqpjaazlx} and
Proposition~\ref{prop:5}, one sees that this situation cannot happen in
monoids of type FC.

As a second consequence of $\SS$ not being closed by left divisibility, the
sets $\D(x)$ are not in general subsets of~$\Sigma$, but of $\SS$ itself.
Indeed, one has for instance: $\D(ab)=\{a,cb\}$, contrary to the second
statement of Proposition~\ref{prop:5} for monoids of type FC.

We note however that $\D(\unit)\subseteq\Sigma$ in any Artin-Tits monoid.
Therefore, evaluated at~$\unit$, the graded Möbius transform has the
following expression:
\begin{gather}
  \label{eq:3}
\T f(\unit)=\sum_{D\subseteqlub\Sigma}(-1)^{\abs{D}}f\bigl(\bigveel D\bigr).
\end{gather}

Consider in particular the case of a function $f$ of the form
$f(x)=p^{\abs{x}}$, for some real number~$p$. Then:
\begin{gather}
  \label{eq:5}
\T f(\unit)=\sum_{D\subseteqlub\Sigma}(-1)^{\abs{D}}p^{\abs{\bigveel D}}.
\end{gather}
This is a polynomial expression in~$p$. We shall see in
Section~\ref{sec:mobius-polynomial} that this polynomial corresponds to the
Möbius polynomial of the monoid, which is a simplified version of the Möbius
function in the sense of Rota~\cite{rota64} associated with the partial
order.

Furthermore, if $x$ is simple in a general Artin-Tits monoid, then $\bA[x]$
has the following expression: $\bA[x]=\{y\in\SS\tq x\leql y\}$. Therefore,
restricted to~$\SS$, the inverse graded Möbius transform takes the following
form:
\begin{gather}
  \label{eq:4}
\text{for $x\in\SS$}\qquad\T^\ast h(x)=\sum_{y\in\SS\tqs x\leql y}h(y).
\end{gather}

\subsection{Multiplicative probability measures on the boundary}
\label{sec:mult-meas-bound}

In this section, we introduce a particular class of measures defined on the
boundary of Artin-Tits monoids, the class of \emph{multiplicative measures}.
Measures defined on the boundary of an Artin-Tits monoid $\bA$ can be seen as
those measures $\nu$ on the completion $\bAbar$ with support in $\BA$, \ie,
such that $\nu(x) = 0$ for all $x \in \bA$. In
Section~\ref{sec:char-mult-meas}, we entirely characterize multiplicative
measures through a finite family of probabilistic parameters. Each
multiplicative measure induces a natural finite Markov chain. Finally, we
investigate in Section~\ref{sec:exist-uniq-unif} the particular case of the
uniform measure.

\subsubsection{Multiplicative measures and valuations}
\label{sec:mult-meas-valu}

Multiplicative measures on the boundary of Artin-Tits monoids are probability
measures that generalize the classical Bernoulli measures on infinite
sequences of letters.

\begin{definition}
  \label{def:2}
Let\/ $\bA$ be an Artin-Tits monoid. A finite measure $\nu$ on the boundary
$\BA$ is \emph{multiplicative} if it satisfies the two following properties:
\begin{align*}
\forall x\in\bA\quad\nu(\up x)&>0\\
\forall x\in\bA\quad\forall y\in\bA\quad\nu(\up(x\cdot y))&=\nu(\up
                                                            x)\cdot\nu(\up y).
\end{align*}

A \emph{valuation on\/ $\bA$} is any function $f:\bA\to(0,+\infty)$
satisfying $f(x\cdot y)=f(x)\cdot f(y)$ for all $x,y\in\bA$.

If $\nu$ is a multiplicative measure on~$\BA$, then the \emph{valuation
associated with~$\nu$} is the function $f:\bA\to(0,+\infty)$ defined by
$f(x)=\nu(\up x)$ for all $x\in\bA$.
\end{definition}

\begin{remark}
  \begin{enumerate}
  \item Assuming that multiplicative measures exist, which is not obvious
      to prove, any such measure satisfies $\nu(\BA)=\nu(\up\unit)=1$,
      hence is a probability measure.
  \item If $\bA$ is a free monoid, multiplicative measures correspond to
      the usual Bernoulli measures characterizing \iid\ sequences. If $\bA$
      is a heap monoid or a braid monoid, multiplicative measures have been
      introduced respectively in~\cite{abbes15a} and in~\cite{abbes17}.
  \end{enumerate}
\end{remark}

Let us put aside a trivial multiplicative measure which is found in
Artin-Tits monoids of spherical type.

\begin{definition}
  \label{def:11}
  Let\/ $\bA$ be an Artin-Tits monoid of spherical type. Let
  $\Delta_\infty$ denote the boundary element
  $\Delta_\infty=(\Delta,\Delta,\dots)$, where $\Delta$ is the Garside
  element of\/~$\bA$.  Then the constant valuation $f(x)=1$ on\/ $\bA$
  corresponds to the multiplicative measure
  $\nu=\delta_{\Delta_\infty}$\,, which we call\/
  \emph{degenerate}. Any other multiplicative measure on the boundary
  of an Artin-Tits monoid, either of spherical type or not, is
  \emph{non degenerate}.
\end{definition}

By Proposition~\ref{prop:2}, a multiplicative measure is entirely determined
by its associated valuation. We note that valuations always exist: it
suffices to consider for instance $f$ to be constant on the set $\Sigma$ of
generators of~$\bA$, say equal to~$r$. Then $f$ extends uniquely to a
valuation on~$\bA$, given by $f(x)=r^{\abs{x}}$ for all $x\in\bA$. Since we
will give a special attention to this kind of valuation and associated
multiplicative measures whenever they exist, we introduce the following
definition.

\begin{definition}
  \label{def:6}
  A valuation $f:\bA\to(0,+\infty)$ on an Artin-Tits monoid\/
  $\bA=\bA(\Sigma,\ell)$ is \emph{uniform} if it is constant
  on\/~$\Sigma$, or equivalently if it is of the form $f(x)=p^{\abs{x}}$
  for some $p\in(0,+\infty)$. A multiplicative measure is
  \emph{uniform} if it is non degenerate and if its associated
  valuation is uniform.
\end{definition}

\subsubsection{Characterization and realization of
  multiplicative measures}
\label{sec:char-mult-meas}

Our aim is to characterize the valuations that correspond to multiplicative
measures. To this end, we introduce the following definition.

\begin{definition}
\label{def:12}
  Let $f:\bA\to(0,+\infty)$ be a valuation on an Artin-Tits
  monoid\/~$\bA$, and let $h=\T f$ be the graded Möbius transform
  of~$f$. We say that $f$ is a \emph{Möbius valuation} whenever:
\begin{align*}
  h(\unit)&=0,&\text{and}&&\forall x\in\SS\setminus\{\unit\}\quad h(x)&>0,
\end{align*}
where $\SS$ denotes as usual the smallest Garside subset of\/~$\bA$.
\end{definition}

This definition is motivated by the following result, which shows in
particular that being Möbius is a necessary and sufficient condition for a
valuation to be associated with some non degenerate multiplicative measure.
It also details the nature of the probabilistic process associated with the
boundary elements.

\begin{theorem}
  \label{thr:1}
  Let\/ $\bA$ be an irreducible Artin-Tits monoid, and let $\nu$ be a
  non degenerate multiplicative measure on\/~$\BA$. Then the
  valuation $f(\cdot)=\nu(\up\cdot)$ is a Möbius valuation.

Conversely, if $f$ is a Möbius valuation on~$\bA$, then there exists a non
degenerate multiplicative measure on~$\BA$, necessarily unique, say~$\nu$,
such that $f(\cdot)=\nu(\up\cdot)$. Hence, non degenerate multiplicative
measures on $\BA$ correspond bijectively to Möbius valuations on~$\bA$.

Furthermore, let $(f,\nu)$ be such a corresponding pair. Let $h$ be the
graded Möbius transform of~$f$. For every integer $n\geq1$, let
$X_n:\BA\to\SS\setminus\{\unit\}$ denote the $n^\text{th}$~canonical
projection of boundary elements, which maps a  boundary element
$\xi=(y_j)_{j\geq1}$ to the simple element~$y_n$. Then, under the probability
measure~$\nu$, the sequence of random variables $(X_n)_{n\geq1}$ is a
homogeneous Markov chain with values in the finite set\/
$\SS\setminus\{\unit\}$, with initial distribution and with transition matrix
$P$ given by:
  \begin{align*}
    \forall x\in\SS\setminus\{\unit\}&\quad\nu(X_1=x)=h(x),\\
\forall x,y\in\SS\setminus\{\unit\}&\quad P_{x,y}=\un(x\to y)f(x)\frac{h(y)}{h(x)}\,.
  \end{align*}

  This Markov chain is ergodic if\/ $\bA$ is not of spherical type;
  and has $\SS\setminus\{\unit,\Delta\}$ as unique ergodic component
  if\/ $\bA$ is of spherical type.
\end{theorem}

The %
following lemmas are central in the proof of Theorem~\ref{thr:1}, and also
for subsequent results.

\begin{lemma}
\label{lem:3.5} Let\/ $\bA$ be an Artin-Tits monoid, of smallest Garside
subset\/~$\SS$. Let $f:\bA\to(0,+\infty)$ be a valuation, and let $h=\T f$ be
the graded Möbius transform of~$f$. For an element $x \in \bA$, let $u \in
\SS$ be the last element in the normal form of $x$, and let $\xtilde \in \bA$
be such that $x = \xtilde \cdot u$. It holds that:
\begin{gather*}
 h(x) = f(\xtilde) h(u).
\end{gather*}
\end{lemma}

\begin{proof}
 Lemma~\ref{lem:9} proves that $\D(x) = \D(u)$,
 and since $f$ is a valuation it follows that
 \begin{align*}
  h(x) & = \sum_{D \subseteqlub \D(x)} (-1)^{\abs{D}} f\bigl(x \cdot \bigveel D \bigl) \\
  & = \sum_{D \subseteqlub \D(u)} (-1)^{\abs{D}} f(\xtilde) f\bigl(u \cdot \bigveel D \bigl)
  = f(\xtilde) h(u).
 \qedhere
 \end{align*}
\end{proof}

\begin{lemma}
\label{lem:4} Let\/ $\bA$ be an Artin-Tits monoid, of smallest Garside
subset\/~$\SS$. Let $f:\bA\to(0,+\infty)$ be a valuation, and let $h=\T f$ be
the graded Möbius transform of~$f$. Let also $g:\bA\to\bbR$ be the function
defined by:
  \begin{gather*}
    g(x)=\sum_{y\in\SS\tqs u\to y}h(y),
  \end{gather*}
where $u\in\SS$ is the last element in the normal form of~$x$. Then
$h(x)=f(x)\cdot g(x)$ holds for all $x\in\bA$.
\end{lemma}

\begin{proof}
  Let $\P(\SS)$ denote the powerset of~$\SS$, and let
  $F,G:\P(\SS)\to\bbR$ be the two functions defined by:
  \begin{align*}
    F(U)&=\sum_{D\subseteqlub U}(-1)^{\abs{D}}f\bigl(\bigveel D\bigr),&
G(U)&=\sum_{y\in\SS}\un(U\cap \down y=\emptyset) h(y),
  \end{align*}
where $\down y =\{z\in\bA\tq z\leql y\}$.

We first prove the equality $F(U)=G(U)$ for all $U\in\P(\SS)$. For any
$x\in\SS$, we have $f(x)=\T^* h(x)$ according to
Theorem~\ref{thr:1mobiusyta}. According to~\eqref{eq:4}, this writes as
follows:
\begin{gather*}
  f(x)=\sum_{y\in\SS\tqs x\leql\, y}h(y).
\end{gather*}
This works in particular for $x=\bigveel D$ for any $D\subseteqlub\SS$, since
then $x\in\SS$, yielding:
\begin{align}
\notag
  F(U)&=\sum_{D\subseteqlub U}(-1)^{\abs{D}}\,\Bigl(\,\sum_{y\in\SS}\un\bigl(\bigveel
        D\leql y\bigr)h(y)\Bigr)\\
\label{eq:9}
      &=\sum_{y\in\SS}\;\Bigl(\,
\sum_{D\subseteqlub U}\un\bigl(\bigveel D\leql y\bigr)(-1)^{\abs{D}}
        \Bigr)\cdot h(y)\\
  &=\sum_{y\in\SS}\;\Bigl(\,\sum_{D\subseteq(U\cap\down y)}(-1)^{\abs{D}}\Bigr)\cdot h(y)
        =G(U).
\end{align}

Next, we observe that for every $u,y\in\SS$, one has $u\to y$ if and only if
$\D(u)\cap\down y=\emptyset$, which implies:
\begin{align}
\label{eq:6}
  g(u)=G\bigl(\D(u)\bigr).
\end{align}
And, using that $f$ is a valuation, we have:
\begin{align}
\notag
  h(u)&=\sum_{D\subseteqlub \D(u)}(-1)^{\abs{D}}f\bigl(u\cdot\bigveel
        D\bigr)\\
\label{eq:7}
      &=f(u)\cdot\Bigl(\sum_{D\subseteqlub\D(u)}(-1)^{\abs{D}}f\bigl(\bigveel
        D\bigr)\Bigr)=f(u)\cdot F\bigl(\D(u)\bigr).
\end{align}
Putting together~\eqref{eq:7}, \eqref{eq:6} and~\eqref{eq:9} with $U=\D(u)$,
we obtain $h(u)=f(u)\cdot g(u)$.

So far, we have proved the statement of the lemma for $x=u\in\SS$. For an
element $x\in\bA$, let $u\in\SS$ be the last element in the normal form
of~$x$, and let $\xtilde\in\bA$ be such that $x=\xtilde\cdot u$.
Lemma~\ref{lem:3.5} states that $h(x)=f(\xtilde)\cdot h(u)$. We also have
$g(x)=g(u)$, and thus by the first part of the proof: $h(x)=f(\xtilde)\cdot
f(u)\cdot g(u)=f(x)\cdot g(x)$, which completes the proof.
\end{proof}

\begin{lemma}
  \label{lem:3}
  Let $f:\bA\to(0,+\infty)$ be a valuation defined on an irreducible
  Artin-Tits monoid\/~$\bA$ with at least two generators. Let $h=\T f$
  be the graded Möbius transform of~$f$, that we assume to satisfy
  $h(\unit)=0$.

  We consider the Charney graph $(\Ch,\to)$, and the non-negative
  square matrix $B=(B_{x,y})_{(x,y)\in\Ch\times\Ch}$ defined by
  $B_{x,y}=\un(x\to y)f(y)$.  Let also $g:\bA\to\bbR$ be the
  function defined as in Lemma~{\normalfont\ref{lem:4}}.  Then:
\begin{enumerate}
\item\label{item:5} The matrix $B$ is primitive, and the column vector
    $\gbar=(g(x))_{x\in\Ch}$ satisfies $B\cdot \gbar=\gbar$.
\item\label{item:6} Furthermore, if $f(\cdot)=\nu(\up\cdot)$ for some
    multiplicative measure~$\nu$, then the assumption $h(\unit)=0$ is
    necessarily satisfied, and one and only one of the following two
    propositions is true:
  \begin{enumerate}
  \item\label{item:7} $h$ and $g$ are identically zero on~$\Ch$. In this
      case, $\bA$~is necessarily of spherical type, and $\nu$ is the
      degenerate measure~$\delta_{\Delta_\infty}$\,.
  \item\label{item:8} $h$ and $g$ are positive on~$\Ch$, $B$~has spectral
      radius~$1$, $\gbar$~is the Perron eigenvector of~$B$ and $\nu$ is
      non degenerate.
  \end{enumerate}
\end{enumerate}
\end{lemma}

\begin{proof}
  The graph $(\Ch,\to)$ is non empty
  since $\bA$ is assumed to have at least two generators.
  The function $f$ is positive on $\Ch$ by definition of valuations;
  the graph $(\Ch,\to)$ is strongly connected according to
  Theorem~\ref{thr:1qqlknaa}, and has loops since every element
  $\sigma\in\Sigma$ belongs to $\Ch$ and satisfies $\sigma\to \sigma$.
  It follows
  that $B$ is indeed primitive. Finally, to see that $\gbar$ is a
  fixed point of~$B$, we compute:
  \begin{align*}
    (B\cdot\gbar)_x&=\sum_{y\in\Ch\tqs x\to y}f(y)g(y)=\sum_{y\in\Ch\tqs x\to y}h(y),
  \end{align*}
  the latter equality by Lemma~\ref{lem:4}. Now the set $\Ch$ differs
  from $\SS$ by at most two elements: either~$\unit$, if $\bA$ is not
  of spherical type, or $\unit$ and $\Delta$ if $\bA$ is of spherical
  type. Since $h(\unit)=0$ by assumption, and since $x\to\Delta$ does
  not hold for any $x\in\Ch$, the above equality writes
  $(B\cdot\gbar)_x=\sum_{y\in\SS\tqs x\to y}h(y)=g(x)$, which completes
  the proof of point~\ref{item:5}.

For point~\ref{item:6}, assume that $f(\cdot)=\nu(\up\cdot)$ for some
multiplicative measure $\nu$ on~$\BA$. Write $\xi=(X_1,X_2,\dots)$ for a
generic element $\xi\in\BA$. Then, by Proposition~\ref{prop:2} applied to
elements of~$\SS$, one has $h(x)=\nu(X_1=x)$ for all $x\in\SS$. It implies on
the one hand that $h$ is non-negative on~$\SS$. On the second hand, since
$X_1$ takes its values in $\SS\setminus\{\unit\}$ only,  the total
probability law yields:
\begin{gather}
\label{eq:10}
  \sum_{y\in\SS\setminus\{\unit\}}h(y)=1.
\end{gather}
Applying formula~\eqref{eq:4} to $x=\unit$ yields:
\begin{gather}
\label{eq:11}
  \sum_{y\in \SS}h(y)=\nu(\up \unit)=\nu(\BA)=1.
\end{gather}
Comparing~\eqref{eq:10} and~\eqref{eq:11} yields $h(\unit)=0$.

Since $h$ is non-negative, the vector $\gbar$ is also non-negative. Since $B$
is primitive, it follows from the Perron-Frobenius Theorem for primitive
matrices~\cite{seneta81} that $\gbar$ is either identically zero or positive.

Assume that $\gbar=0$. Then Lemma~\ref{lem:4} implies that $h=0$ on~$\Ch$. It
follows from~\eqref{eq:10} that $\Ch\setminus\{\unit\}\neq\emptyset$, and
thus that $\bA$ is of spherical type and that $h(\Delta)=1$, where $\Delta$
is the Garside element of~$\bA$. For any element $x\in\bA$, if $u\in\SS$ is
the last element of the normal form of~$x$, and if $\xtilde\in\bA$ is such
that $x=\xtilde\cdot u$, Lemma~\ref{lem:3.5} states that
$h(x)=f(\xtilde)h(u)$. Hence, for all normal sequences $(x_1,\dots,x_k)$, if
$x_i \neq \Delta$ for some $i \leq k$, then by Proposition~\ref{prop:2} we
have:
\begin{gather*}
  \nu(X_1=x_1,\dots,X_k=x_k) \leq \nu(X_1=x_1,\dots,X_i=x_i) = h(x_1\cdot\ldots\cdot x_i) = 0,
\end{gather*}
from which follows $\nu\bigl(\xi=(\Delta,\Delta,\dots)\bigr)= \nu\bigl(\BA) =
1$, or in other words, $\nu=\delta_{\Delta_\infty}$, as claimed in
point~\ref{item:7}.

Assume now that $\gbar\neq0$, and thus that $\gbar$ is positive. Since
$\gbar$ is a fixed point of~$B$, it follows from the Perron-Frobenius Theorem
for primitive matrices that $\gbar$ is a Perron eigenvector of~$B$, and thus
$B$ is of spectral radius~$1$. From $\gbar>0$ and from Lemma~\ref{lem:4}, we
deduce that $h>0$ on~$\Ch$, and this implies that $\nu$ is non degenerate as
claimed in point~\ref{item:8}.
\end{proof}

\begin{proof}[Proof of Theorem~\ref{thr:1}.]
We first prove that, if $\nu$ is a non degenerate multiplicative measure
on~$\bA$, then $f:\bA\to(0,\infty)$ defined by $f(\cdot)=\nu(\up\cdot)$ is
Möbius. Let $h=\T f$. It follows from
  Lemma~\ref{lem:3} that $h(\unit)=0$, and since $\nu$ is non
  degenerate, point~\ref{item:8} shows that $h>0$ on~$\Ch$. To obtain
  that $f$ is a Möbius valuation, it remains only to prove, in case
where $\bA$ is of spherical type, that $h(\Delta)>0$ for $\Delta$ the
  Garside element of~$\bA$. But, since $\Delta$ satisfies
  $\Delta\to x$ for all $x\in\SS$, one has $\D(\Delta)=\emptyset$ and
  thus $h(\Delta)=f(\Delta)>0$. This proves the first statement of
  Theorem~\ref{thr:1}.

Conversely, let $f$ be a Möbius valuation on~$\bA$. Define a non-negative
matrix $Q=(Q_{x,y})_{(x,y)\in
  (\SS\setminus\{\unit\})\times(\SS\setminus\{\unit\})}$ by
\begin{gather*}
  Q_{x,y}=\un(x\to y)f(x)\frac{h(y)}{h(x)}.
\end{gather*}
Since $h(\unit)=0$, Lemma~\ref{lem:4} shows that $Q$ is stochastic.
Furthermore, the non-negative vector $(h(x))_{x\in\SS\setminus\{\unit\}}$
satisfies, using that $h(\unit)=0$ and Theorem~\ref{thr:1mobiusyta}:
\begin{align*}
  \sum_{x\in\SS\setminus\{\unit\}}h(x)&=\sum_{x\in\SS}h(x)=\T^\ast h(\unit)=f(\unit)=1.
\end{align*}
It is thus a probability vector.

Consider the canonical probability space $(\Omega,\FFF,\pr)$ associated to
the Markov chain $(X_n)_{n\geq1}$ with values in $\SS\setminus\{\unit\}$,
with initial distribution $(h(x))_{x\in\SS\setminus\{\unit\}}$ and with
transition matrix~$Q$. Let also $\pi:\Omega\to\BA$ be the canonical mapping,
defined with $\pr$-probability~$1$, and let $\nu=\pi_\ast\pr$, the image
probability measure on~$\BA$. Then we claim that $\nu(\up \cdot)=f(\cdot)$.

Indeed, by Theorem~\ref{thr:1mobiusyta} and Proposition~\ref{prop:2}, it is
enough to prove that, for every integer $n\geq1$, the law of
$(X_1,\dots,X_n)$ satisfies:
$\pr(X_1=x_1,\ldots,X_n=x_n)=h(x_1\cdot\ldots\cdot x_n)$ for every sequence
$(x_1,\dots,x_n)$. If the sequence $(x_1,\dots,x_n)$ is not normal, then both
members vanish. And if the sequence is normal, then the consecutive
cancellations yield:
\begin{align*}
  \pr(X_1=x_1,\ldots,X_n=x_n)=f(x_1\cdot\ldots\cdot x_{n-1})\cdot
  h(x_n)=h(x_1\cdot\ldots\cdot x_n),
\end{align*}
the latter equality using Lemma~\ref{lem:3.5}. This proves that
$\nu(\up\cdot)=f(\cdot)$, as expected.

This also proves that, for any non degenerate multiplicative measure $\nu$
on~$\BA$, the sequence $(X_n)_{n\geq1}$ defined in the statement of
Theorem~\ref{thr:1} is indeed a Markov chain with the specified transition
matrix and initial distribution.

If $\bA$ is of spherical type, then $f(\Delta) = \nu(\up \Delta) < 1$.
Indeed, otherwise one would have $\nu(\up \Delta^n) = \nu(\up \Delta)^n = 1$,
and therefore $\nu = \delta_{\Delta^\infty}$, contradicting its non
degeneracy. The ergodicity statements derive at once from the fact that the
Charney graph is irreducible on the one hand, and that $\Delta$ is initial
with $Q_{\Delta, \Delta} = f(\Delta) < 1$ if $\bA$ is of spherical type, on
the other hand. The proof is complete.
\end{proof}

\subsubsection{Uniqueness of the uniform measure for
  Artin-Tits monoids}
\label{sec:exist-uniq-unif}

Although we have not yet proved the existence of non degenerate uniform
measures on the boundary of Artin-Tits monoids---which is deferred to
Section~\ref{sec:param-mult-meas}---, we are ready to prove the following
uniqueness result.

\begin{theorem}
  \label{thr:2}
  Let\/ $\bA$ be an irreducible Artin-Tits monoid. Then there exists at
  most one uniform measure on\/~$\BA$.
\end{theorem}

Recall that, by Definition~\ref{def:6}, uniform measures are non degenerate.

\begin{proof}[Proof of theorem~\ref{thr:2}.]
  Let $\nu$ and $\nu'$ be two uniform measures, say
  associated to the two uniform valuations $f_1( x)=p_1^{\abs{x}}$ and
  $f_2( x)=p_2^{\abs{x}}$. Without loss of generality, we may assume that
  $p_1\leq p_2$.

  Consider the primitive matrices $B_1$ and $B_2$ constructed as
  in Lemma~\ref{lem:3}, associated to $f_1$ and to $f_2$
  respectively. Then $B_1\leq B_2$ since $p_1\leq p_2$, and both
  matrices have spectral radius~$1$. It follows from the
  Perron-Frobenius Theorem~\cite{seneta81} that $B_1=B_2$, and thus $p_1=p_2$ and
  $f_1=f_2$. Then Proposition~\ref{prop:2} implies that $\nu_1=\nu_2$.
\end{proof}

\section{Conditioned weighted graphs}
\label{sec:centr-limit-theor}

\subsection{General framework}
\label{sec:general-framework}

\subsubsection{Non negative matrices}
\label{sec:non-negative-matrices}

Although we already appealed to the concept of non-negative matrix and to the
Perron-Frobenius theory, we recall now some standard definitions from this
theory, see for instance~\cite{seneta81}.  A real square matrix $M$ is
\emph{non-negative}, denoted $M\geq0$, if all its entries are non-negative,
and \emph{positive}, denoted $M>0$, if all its entries are positive. The same
definitions apply to vectors. If $M\geq0$, it is \emph{primitive} if $M^K>0$
for some integer power $K>0$, and then $M^k>0$ for all $k\geq K$. The matrix
$M$ is \emph{irreducible} if for every pair $(i,j)$ of indices, there exists
an integer $k>0$ such that~$M^k_{i,j}>0$.

We interpret non-negative matrices as labeled oriented graphs, which we
simply call \emph{graphs} for brevity; vertices are represented by the
indices of the matrix, and there is an edge from $x$ to~$x'$\,, labeled by
the entry~$M_{x,x'}$\,, whenever $M_{x,x'}>0$. A \emph{path} in the graph is
any non empty sequence $(x_0,\dots,x_k)$ of indices such that
$M_{x_i,x_{i+1}}>0$ for all $i=0,\dots,k-1$. The non-negative integer $k$ is
the \emph{length} of the path. The path is a \emph{circuit} if, in addition,
$x_k=x_0$.

In this representation, irreducible matrices correspond to strongly connected
graphs and primitive matrices correspond to strongly connected graphs such
that circuits have $1$ as greatest common divisor of their lengths.

The \emph{spectral radius}, denoted~$\rho(M)$, of a non-negative matrix $M$
is defined as the largest modulus of its complex eigenvalues. In the seek of
completeness, we establish the following elementary lemma, which is probably
found elsewhere as a textbook exercise.

\begin{lemma}
  \label{lem:pokqpkqa}
  Let $M$ be a non-negative square matrix of size $N>1$. Assume that
  $M$ has a unique eigenvalue of maximal modulus, which is simple. Let
  $\lambda$ be this eigenvalue. Then:
  \begin{enumerate}
  \item\label{item:1pjas} $\lambda$ is real and positive, and thus
      $\lambda=\rho(M)$.
  \item\label{item:2asakln} There exists a pair $(\ell,r)$ of non-negative
      vectors, such that $\ell$ is a left (row) $\lambda$-eigenvector and
      $r$ is a right (column) $\lambda$-eigenvector of~$M$, and such that
      $\ell\cdot r=1$. Any other such pair is of the form $(t\ell,t^{-1}r)$
      for some $t>0$.
  \item\label{item:3asaspp} With $(\ell,r)$ a pair of non-negative vectors
      as above, the matrix $M$ has the following decomposition:
\begin{align*}
  M&=\lambda\Pi+Q,\quad\text{with\/ $\Pi=r\cdot\ell$},
\end{align*}
and where $Q$ is a matrix satisfying\/ $\Pi\cdot Q=Q\cdot\Pi=0$ and
$\rho(Q)<\lambda$. It entails the following convergence:
\begin{gather}
\label{eq:2qpoijak}
  \lim_{k\to\infty}\Bigl(\frac1\lambda M\Bigr)^k=\Pi.
\end{gather}
  \end{enumerate}
\end{lemma}

\begin{proof}
  Observe first that $\lambda\neq0$, otherwise $M$ could not have
  other eigenvalues than $\lambda=0$, contradicting that $\lambda$ is
  simple. Hence, considering $(1/\abs{\lambda})M$, which is still
  non-negative, instead of~$M$, we assume without loss of generality
  that $\abs{\lambda}=1$. Secondly, the spectral decomposition of $M$
  entails that $M$ writes as $M=\lambda\Pi+Q$, where $\Pi$ is the
  matrix of a projector of rank~$1$, $Q$~is a matrix with all
  eigenvalues less than $1$ in modulus, and $\Pi\cdot
  Q=Q\cdot\Pi=0$.
  Consequently, the powers $M^k$ form a bounded sequence of matrices.

Now, let $x$ be a right $\lambda$-eigenvector of~$M$, and let $\abs{x}$
denote the vector with $\abs{x}_i=\abs{x_i}$ for all $i\in\{1,\dots,N\}$.
Then $\sum_j M_{i,j}x_j = \lambda x_i$ yields $\abs{x}_i \leq
\sum_jM_{i,j}\abs{x_j}$, or in other words: $M\cdot\abs{x} \geq \abs{x}$.
Since $M$ is non-negative, it follows that $M^{k+1}\cdot\abs{x} \geq
M^k\cdot\abs{x}$, so that each coordinate of $M^k\cdot\abs{x}$ is a
non-decreasing sequence of reals. Since $(M^k)_{k>0}$ is bounded, it follows
that $(M^k\cdot\abs{x})_{k>0}$ converges toward a non-negative vector~$r$,
satisfying $r\geq \abs{x}$ and $M\cdot r=r$. Since $\abs{x}\neq0$, in
particular $r\neq0$ and thus $1$ is an eigenvalue of~$M$, which implies that
$\lambda=1$.

We have already found that $r$ is a non-negative fixed point of $M$ for its
right action on vectors. The same reasoning applied to the transpose of~$M$
(or to the left action of $M$ on vectors) yields the existence of a non
negative left fixed point, say~$\ell$, of~$M$.

The decomposition seen at the beginning of the proof now writes as $M=\Pi+Q$
with $\rho(Q)<1$ and $\Pi\cdot Q=Q\cdot\Pi=0$. It implies at once
$M^k=\Pi+Q^k$ for all integers $k>0$ and thus $\lim_{k\to\infty}M^k=\Pi$.

To obtain the existence of the pair $(\ell,r)$ with the normalization
condition $\ell\cdot r=1$, it suffices to prove that $\ell\cdot r>0$. For
this, we observe first that $r$ and $\ell$ are respectively right and left
fixed point of~$\Pi$, since $\Pi$ is associated to the
$\lambda$-characteristic subspace of~$M$. Being a rank~$1$ projector with $r$
as right fixed point, $\Pi$~writes as $\Pi=r\cdot \ell'$ for some non zero
row vector~$\ell'$. Hence $\ell=\ell\cdot\Pi=(\ell\cdot r)\ell'$. This
implies that $\ell\cdot r\neq0$, which was to be proved.

Assuming now that the normalization condition $\ell\cdot r=1$ holds, we
obtain $\ell'=\ell$ and thus $\Pi=r\cdot\ell$, completing the proof.
\end{proof}

\subsubsection{Conditioned weighted graphs}
\label{sec:cond-weight-graphs}

The central object of study of this section is the following.

\begin{definition}[\CWG]
  \label{def:3}
  A \emph{conditioned weighted graph (\CWG)} is given by a triple\/ $\GG=(M,w^-,w^+)$,
  where $M$ is a non-negative matrix of size $N\times N$ with $N>1$,
  and
  \begin{align*}
    w^-&:\{1,\ldots,N\}\to\bbR^+\,,&
    w^+&:\{1,\ldots,N\}\to\bbR^+\,,
  \end{align*}
  are two real-valued and non-negative functions, respectively
  called\/ \emph{initial} and\/ \emph{final}. We identify $w^-$ with
  the corresponding row vector of size~$N$, and we identify $w^+$ with
  the corresponding column vector of size~$N$. Furthermore, we assume
  that the triple $(M,w^-,w^+)$ satisfies the following conditions.
  \begin{enumerate}
  \item $M$ has a unique eigenvalue of maximal modulus, which is simple.
      Let $\lambda$ be this eigenvalue, which is real and positive
      according to Lemma~\ref{lem:pokqpkqa}.
  \item Let\/ $(\ell,r)$ be a pair of nonzero non-negative left and right
      $\lambda$-eigenvectors of~$M$. We assume that $w^-\cdot r>0$ and
      $\ell\cdot w^+>0$ both hold.
  \end{enumerate}
\end{definition}

\subsubsection{Two particular cases}
\label{sec:two-particular-cases}

For the study of Artin-Tits monoids, we shall be interested in triples
$(M,w^-,w^+)$ falling into one of the two following cases.

\begin{description}
  \item[Case A.] $M$ is primitive, and the functions $w^-$ and $w^+$ are
      non identically zero.
  \item[Case B.] For some integer $0< K< N$ and for some non-negative
      matrices~$A$, $T$ and $\Mtilde$ of sizes $K\times K$, $K\times(N-K)$,
      and $(N-K)\times(N-K)$ respectively, $M$~has the following form:
\begin{gather}
    \label{eq:67}
M=
\begin{pmatrix}
A&T\\0&\Mtilde
\end{pmatrix}
\end{gather}
where $\Mtilde$ is primitive, and the spectral radii $\rho(A)$ and
$\rho(\Mtilde)$ satisfy $\rho(A)<\rho(\Mtilde)$. The functions $w^-$ and
$w^+$ are assumed to be non identically zero on the $N-K$ last indices.
\end{description}

\begin{proposition}
  \label{prop:pokpala}
  In either case A or B described above, the triple $(M,w^-,w^+)$ is a
  \CWG.
\end{proposition}

\begin{proof}
  In case~A, this is a direct application of the Perron-Frobenius
  Theorem for primitive matrices. In case~B, let $M$ be as
  in~\eqref{eq:67}. Let $\lambda$ be the Perron eigenvalue
  of~$\Mtilde$, \ie, according to the
  Perron-Frobenius Theorem for primitive matrices, the simple,
  unique eigenvalue of $\Mtilde$ of largest modulus. Then $\lambda$ is
  the unique eigenvalue of $M$ of maximal modulus since
  $\rho(A)<\rho(\Mtilde)$, and it is simple as an eigenvalue of~$M$.

  It remains only to prove the existence of a pair $(\ell,r)$ of
  $\lambda$-eigenvectors of $M$ such that $w^-\cdot r>0$ and
  $\ell\cdot w^+>0$. For this, let $(\elltilde,\rtilde)$ be a pair of
  positive left and right $\lambda$-eigenvectors of~$\Mtilde$, and
  consider the vectors $\ell$ and $r$ defined by:
\begin{align*}
  \ell&=
             \begin{pmatrix}
               0&\elltilde\;
             \end{pmatrix}
&r&=
     \begin{pmatrix}
       (\lambda I-A)^{-1}\cdot T\cdot\rtilde\\\rtilde
     \end{pmatrix}
\end{align*}

The hypothesis $\rho(A)<\rho(\Mtilde)$ implies that $\lambda I-A$ is
invertible, hence $r$ is well defined, and $\ell$ and $r$ are left and right
$\lambda$-eigenvectors of~$M$, which, due to Lemma~\ref{lem:pokqpkqa}, have
non-negative entries. Since $w^-$ and $w^+$ are assumed to be non identically
zero on their last $(N-K)$ coordinates, and since $\elltilde$ and $\rtilde$
are positive, they satisfy $w^-\cdot r>0$ and $\ell\cdot w^+>0$.
\end{proof}

\subsection{Weak convergence of weighted distributions}
\label{sec:conv-unif-meas}

Let $\GG=(M,w^-,w^+)$ be a \CWG. Given a path $x=(x_0,\dots,x_{k})$ associated
to~$M$, we define its \emph{weight} $w(x)$ as the following non-negative
real:
\begin{gather*}
  w(x)=w^-(x_0)\cdot M_{x_0,x_1}\cdot\ldots\cdot
  M_{x_{k-1},x_k}\cdot w^+(x_{k})\,.
\end{gather*}

From now on, our study of conditioned weight graphs focuses on paths and on
probability distributions over sets of paths. As a first elementary result,
we show that paths of positive weight and of length $k$ exist for all $k$
large enough.

\begin{lemma}
  \label{lem:pojqwopjiq}
  Let $\GG=(M,w^-,w^+)$ be a conditioned weighted graph. Then there
  exists an integer $K$ such that, for each $k\geq K$, the set of
  paths of length $k$ and with positive weight is non empty.
\end{lemma}

\begin{proof}
  Let $Z_k$ be the sum of the weights of all paths of
  length~$k$. Then, identifying the functions $w^-$ and $w^+$ with the
  corresponding row and column vectors, one has:
  $Z_k=w^-\cdot M^{k}\cdot w^+$. Let $(\ell,r)$ be a pair of left and
  right $\lambda$-eigenvectors of~$M$ and satisfying $\ell\cdot r=1$,
  where $\lambda$ is the eigenvalue of maximal modulus of~$M$. Then, putting $\Pi=r\cdot\ell$, and
  according to Lemma~\ref{lem:pokqpkqa}:
  \begin{align*}
    \frac1{\lambda^{k}}Z_k&=w^-\cdot\frac1{\lambda^{k}}M^{k}\cdot
                          w^+\xrightarrow[k\to\infty]{}
w^-\cdot\Pi\cdot w^+=(w^-\cdot r)(\ell\cdot w^+)>0.
  \end{align*}
It follows that $Z_k>0$ for all $k$ large enough, and in particular the set
of paths with positive weight is non empty.
\end{proof}

\begin{definition}
  \label{def:4}
  Let $\GG=(M,w^-,w^+)$ be a conditioned weighted graph, and for each
  integer $k\geq1$, let $G_k$ denote the set of paths of length $k$
  in~$\GG$. The \emph{weighted distribution} on $G_k$ is the
  probability distribution $\mu_k$ on $G_k$ defined by:
\begin{align*}
\forall z\in G_k\quad  \mu_k(z)&=\frac{w(z)}{Z_k}\,,&\text{with }
Z_k&=\sum_{z\in G_k}w(z)\,,
\end{align*}
which is well defined at least for $k$ large enough according to
Lemma~{\normalfont\ref{lem:pojqwopjiq}}.

For each integer $j\geq1$, and for $k$ large enough, we denote by $\lpm kj$
the joint law of the first $j+1$ elements $(x_0,\ldots,x_{j})$ of a path
$(x_0,\ldots,x_{k})$ of length $k$ distributed according to~$\mu_k$\,. We
call $\lpm kj$ the \emph{left $j$-window distribution
  with respect to~$\mu_k$\,}.
\end{definition}

Let $\GG=(M,w^-,w^+)$ be a conditioned weighted graph, and for each integer
$k\geq0$, let $\Omega_k$ denote the set of paths of length \emph{at
most~$k$}. Let also $\Omega$ be the set of finite or infinite paths, with its
canonical topology (for which it is a compact space). The set of finite paths
$\bigcup_k \Omega_k$ is dense in~$\Omega$.

In general, the collection $(\mu_k)_{k\geq0}$ \emph{is
  not a projective system of probability measures}, since the measure
induced by $\mu_{k+1}$ on the set $G_k$ of paths of length $k$ does not
coincide with~$\mu_k$. Hence, the projective limit of~$(\mu_k)_{k\geq0}$ is
not defined in general.

Yet, weak limits of measures are an adequate tool to replace projective
limits in this case. Indeed, each $G_k$ is naturally embedded into~$\Omega$.
Through this embedding, the distribution $\mu_k$ identifies with a discrete
probability measure, still denoted by~$\mu_k$\,, on the space $\Omega$
equipped with its Borel \slgb.

\begin{theorem}
  \label{thr:4}
  Let $\GG=(M,w^-,w^+)$ be a conditioned weighted graph, and consider
  as in Definition~{\normalfont\ref{def:3}} the eigenvalue~$\lambda$ of maximal
  modulus together with the pair $(\lev,\rev)$ of associated
  eigenvectors.

The sequence $(\mu_k)_{k\geq0}$  of weighted distributions converges weakly
toward a probability measure $\mu$ on~$\Omega$, which is concentrated on the
set\/ $\Xi\subseteq\Omega$ of infinite paths.

For each integer $k\geq0$, let $X_k:\Xi\to\{1,\ldots,N\}$ denote the
$k^\text{th}$ natural projection. Then, under~$\mu$, $(X_k)_{k\geq0}$~is a
Markov chain. Its initial distribution, denoted~$h$, and its transition
matrix, denoted~$P=(P_{i,j})$ where $P_{i,j}$ is the probability to jump from
state $i$ to state~$j$, are given by:
\begin{align*}
  h(i)&=\frac{w^-(i)\rev(i)}{w^-\cdot\rev}\,,&
P_{i,j}&=\lambda^{-1}M_{i,j}\frac{\rev(j)}{\rev(i)}\,,\quad\text{if $r(i)\neq0$},
\end{align*}
independently of the choice of\/~$\rev$. If $r(i)=0$, the line
$P_{i,\bullet}$ is defined as an arbitrary probability vector. The chain can
only reach the set of states $i$ such that $r(i) \neq 0$.

Restricted to the set of reachable states, the chain has a unique stationary
measure, say~$\pi$, given by:
\begin{gather*}
  \forall i\in\{1,\dots,N\}\quad \pi(i)=\lev(i)\rev(i).
\end{gather*}

In Case~A introduced in Section~\ref{sec:two-particular-cases}, the chain
$(X_k)_{k\geq1}$ is ergodic. In Case~B, the chain has a unique ergodic
component, namely the $N-K$ last indices $\{K+1,\dots,N\}$, and hence the
states in $\{1,\dots,K\}$ are all transient.
\end{theorem}

In view of the above result, we introduce the following definition.

\begin{definition}
  \label{def:5}
  The probability measure $\mu$ on the space of infinite sequences
  characterized in Theorem~{\normalfont\ref{thr:4}} is called the
  \emph{limit weighted measure} of the conditioned weighted
  graph\/~$\GG$.
\end{definition}

\begin{proof}[Proof of Theorem~\ref{thr:4}.]
  For $x$ a finite path of length~$j$, denote by $\CC_x$ the
  elementary cylinder of base~$x$, \ie, the set of finite or
  infinite words that start with $x$. For all integers $k$ such that
  $k\geq j$ and such that $\mu_k$ is well defined, one has:
\begin{gather}
  \label{eq:34}
  \mu_k(\CC_x)=\frac1{Z_k}\sum_{z\in G_k\tqs\theta_{j}(z)=x}w(z)\,,
\end{gather}
where $\theta_j$ is the truncation map that only keeps the first $j$ steps of
a path.

Let $\wtilde$ be the real-valued function defined on finite paths by:
\begin{gather*}
  \wtilde(x_0,\ldots,x_{j})=w^-(x_0) M_{x_0,x_1}\dotsm M_{x_{j-1},x_{j}}\,.
\end{gather*}

Then both terms of the quotient in~\eqref{eq:34} can be written through
powers of the matrix~$M$:
\begin{align*}
  \mu_k(\CC_x)&=\frac1{w^-\cdot M^{k}\cdot w^+} \wtilde(x)
\un_{x_{j}}\cdot M^{k-j}\cdot w^+\,.
\end{align*}
where $\un_{x_{j}}$ denotes the row vector filled with~$0$s, except for the
entry $x_{j}$ where it has a~$1$.

According to Lemma~\ref{lem:pokqpkqa}, the following asymptotics holds for
the powers of~$M$:
\begin{gather*}
  M^k=\lambda^k(\rev\cdot\lev)\bigl(1+o(1)\bigr),\quad k\to\infty.
\end{gather*}
Therefore $(\mu_k(\CC_x))_{k\geq0}$ is convergent, with limit given by:
\begin{gather}
\label{eq:33}  \lim_{k\to\infty}\mu_k(\CC_x)=\lambda^{-j}\wtilde
(x)\frac{(\un_{x_{j}}\cdot\rev)(\lev\cdot w^+)}{(w^-\cdot\rev)(\lev\cdot w^+)}
=\lambda^{-j}\wtilde(x)\frac{\rev(x_{j})}{w^-\cdot\rev}\,.
\end{gather}

Elementary cylinders, together with the empty set, are stable under finite
intersections and generate the Borel \slgb\ on~$\Omega$. Elementary cylinders
are clopen sets and thus of empty topological boundary. And finally,
$\Omega$~is a compact metric space. According
to~\cite[Th.~25.8]{billingsley95}, this is enough to deduce the weak
convergence of $(\mu_k)_{k\geq0}$ toward a probability measure $\mu$ on
$\Omega$ such that $\mu(\CC_x)$ coincides with the value of the limit
in~\eqref{eq:33}.

It is clear that the support of $\mu$ only contains infinite paths since, for
every finite path $z=(x_0,\ldots,x_j)$, one has $\mu_k(\{z\})=0$ for $k$
large enough. It follows that $\mu\bigl(\{z\}\bigr)=0$, and thus $\mu(\Xi)=1$
since finite paths are
 countably many.

The vector $h$ and the matrix $P$ defined in the statement are indeed
respectively a probability vector and a stochastic matrix on
$\{1,\ldots,N\}$. The Markov chain with initial law $h$ and transition
matrix~$P$ gives to the cylinder $\CC_x$ the following probability:
\begin{multline*}
  h(x_0) P_{x_0,x_1} P_{x_1,x_2} \dotsm P_{x_{j-1}, x_{j}}=\\
 \frac{w^-(x_0) r(x_0)}{w^-\cdot r} \cdot \lambda^{-1} M_{x_0,x_1} \frac{r(x_1)}{r(x_0)}
     \cdot  \lambda^{-1} M_{x_1,x_2} \frac{r(x_2)}{r(x_1)} \dotsm
     \lambda^{-1} M_{x_{j-1},x_{j}} \frac{r(x_{j})}{r(x_{j-1})}\\
  = \frac{w^-(x_0) r(x_{j})}{w^-\cdot r} \lambda^{-j} M_{x_0,x_1} \dotsm M_{x_{j-1}, x_{j}}
  = \mu(\CC_x),
\end{multline*}
by~\eqref{eq:33}. This shows that this Markov chain has the same joint
marginals as $(X_k)_{k\geq0}$ under~$\mu$, or equivalently, that
$(X_k)_{k\geq0}$ under $\mu$ is the Markov chain with initial law $h$ and
transition matrix~$P$.

By the normalization condition $\lev\cdot\rev=1$, the vector $\pi$ is indeed
a probability distribution, which is readily seen to be left invariant
for~$P$.  Furthermore, left invariant vectors $\theta$ for $P$ and left
$\lambda$-eigenvectors $\theta'$ for~$M$, with support within the set of
reachable states, correspond to each others by $\theta'(i)=\theta(i)/r(i)$.
Since $M$ has a unique left $\lambda$-eigenvector~$\ell$, the unique ergodic
component of the chain corresponds to the support of~$\ell$. In case~A,
$\ell>0$ hence the chain is ergodic. In Case~B, the unique ergodic component
corresponds to the last $N-K$ states.
\end{proof}

\begin{corollary}
\label{cor:3} We keep the same notations as in
Theorem~{\normalfont\ref{thr:4}}. Let $j\geq1$ be an integer. Then, with
respect to~$\mu_k$\,, as $k\to\infty$, the left $j$-window distributions
$\lpm kj$ converge toward the joint law of\/ $(X_0,\ldots,X_{j})$ under the
uniform distribution at infinity~$\mu$.
\end{corollary}

\begin{proof}
  This is a rephrasing of the weak convergence stated in
  Theorem~\ref{thr:4}.
\end{proof}

\subsection{Related notions found in the literature}
\label{sec:relat-with-other}

The notion of conditioned weighted graph is often found in the literature
under disguised forms.  For instance, the transition matrix $P$ of the Markov
chain introduced in Theorem~\ref{thr:4} corresponds to the transformation of
an incidence matrix first introduced by
Parry~\cite{parry64,lind95,kitchens97} in its construction of a stationary
Markov chain reaching the maximum entropy.

This matrix $P$ also has the very same form as the transition matrix of the
\emph{survival process} of a discrete time, finite states absorbing Markov
chain~\cite{darroch65,collet13}. Actually, discrete time, finite states
absorbing Markov chains can be interpreted as a particular case of
conditioned weighted graph, as we briefly explain now.

A finite absorbing Markov chain $(Y_i)_{i\geq0}$ is a Markov chain on
$\{0,\ldots,N\}$ such that $P_{0,0}=1$, and such that $0$ can be reached in a
finite number of states from any other state. Usually, it is assumed that all
states in $\{1,\ldots,N\}$ are strongly connected, which we assume too. For
each $x\in\{1,\ldots,N\}$, we consider the conditioned weighted graph
$\GG_x=(M,w_x,w^+)$ defined as follows: $M$~is the restriction of $P$ to the
entries in $\{1,\ldots,N\}$; $w^+$~is the constant vector with entries~$1$;
and $w_x$~is the indicator function of~$x$.

Assume that $(Y_i)_{i\geq0}$ starts from~$1$. Let $T$ be the first hitting
time of $0$ of the chain $(Y_i)_{i\geq0}$\,. Then it is clear that the
marginal law of $(Y_0,\ldots,Y_j)$ conditioned on $\{T>k\}$ corresponds to
our left $j$-window distribution for the conditioned weighted
graph~$\GG_1$\,. The survival process, if it exists, is a process
$(X_i)_{i\geq0}$ such that:
\begin{gather}
  \label{eq:57}
  \pr(X_0=x_0,\ldots,X_j=x_j)=\lim_{k\to\infty}\pr(Y_0=x_0,\ldots,Y_j=x_j|T>k)\,.
\end{gather}

We recover thus the existence of the survival process, and its form as a
Markov chain, through Theorem~\ref{thr:4} (or Corollary~\ref{cor:3}); this is
established for instance in~\cite[Sections~3.1 and~3.2]{collet13} for
continuous-time Markov chains.

\section{Application to Artin-Tits monoids}
\label{sec:appl-posit-braids}

We apply the notion of Conditioned Weighted Graphs (\CWG) introduced in
Section~\ref{sec:centr-limit-theor} to the counting of elements of Artin-Tits
monoids, maybe with a multiplicative positive weight. The limit of the
associated weighted measures is found to be concentrated on the boundary of
the monoid and to be multiplicative. This yields another representation of
multiplicative measures, introduced in Section~\ref{sec:posit-dual-posit}, as
weak limits of finite probability distributions, and provides a proof of
existence for multiplicative measures. It also yields a parametrization of
multiplicative measures.

\subsection{\CWG\ associated to an irreducible Artin-Tits monoid}
\label{sec:cwg-associated-an}

\subsubsection{Uniform case}
\label{sec:uniform-case}

\begin{definition}
  \label{def:8}
  Let\ $\bA$ be an irreducible Artin-Tits monoid, and let $\SS$ be the
  smallest Garside subset of\/~$\bA$. Let
  $J=\{(x,i)\tq x\in\SS\setminus\{\unit\}\text{ and }1\leq i\leq \abs{x}\}$
  and $N=\#J$. The \CWG\ associated with\/ $\bA$ is the triple
  $(M,w^-,w^+)$, where $M$ is the non-negative square matrix of size
  $N\times N$ indexed by~$J$, and $w^-$ and $w^+$ are defined by:
  \begin{gather*}
    M_{(x,i),(y,j)}=  \begin{cases}
    1,&\text{if $x=y$ and $j=i+1$}\\
1,&\text{if $x\to y$ and $i=\abs{x}$ and $j=1$}\\
0,&\text{otherwise}
  \end{cases}\\
\begin{aligned}
w^-(x,i)&=\un(i=1),&w^+(x,i)&=\un(i=\abs{x}).
  \end{aligned}
\end{gather*}
\end{definition}

The motivation behind this definition is that elements of $\bA$ correspond
bijectively to paths in the graph associated with the matrix~$M$, or more
precisely with the triple $(M,w^-,w^+)$. Indeed, let $x$ be an element
of~$\bA$, with $x\neq\unit$, and let $(x_1,\dots,x_j)$ be the normal form
of~$x$. We associate to $x$ the sequence $\xtilde$ defined by:
\begin{gather}
  \label{eq:20}
\xtilde=\bigl((x_1,1),\dots,(x_1,\abs{x_1}),(x_2,1),\dots,(x_2,\abs{x_2}),
\dots,(x_j,1),\dots,(x_j,\abs{x_j})\bigr).
\end{gather}
Then $\xtilde$ is indeed a path in the graph corresponding to the triple
$(M,w^-,w^+)$, of length $\abs{x_1}+\dots+\abs{x_j}-1=\abs{x}-1$.

This correspondence is a bijection between elements of $\bA$ of length~$k>0$
and paths in $(M,w^-,w^+)$ of length~$k-1$, whence:
\begin{equation}
\label{eq:card_sphere}
  \#\{x\in\bA\tq\abs{x}=k\} = w^- \cdot M^{k-1}\cdot w^+.
\end{equation}

\begin{proposition}
\label{prop:8} Let\/ $\bA$ be an irreducible Artin-Tits monoid with at least
two generators. Then the triple $(M,w^-,w^+)$ introduced in
Definition~\ref{def:8} is indeed a \CWG, corresponding either to Case~A of
Section~\ref{sec:two-particular-cases} if\/ $\bA$ is not of spherical type,
or to case~B if\/ $\bA$ is of spherical type.
\end{proposition}

\begin{proof}
  Assume that $\bA$ is not of spherical type. Then the graph
  $(\SS\setminus\{\unit\},\to)$ is strongly connected according to
  Theorem~\ref{thr:1qqlknaa}, which yields that the graph associated
  with $M$ is strongly connected. Since $x\to x$ holds for every
  $x\in\Sigma$, the diagonal element $M\bigl((x,1),(x,1)\bigr)$ is~$1$, hence $M$ is primitive. The conditions on $w^-$ and $w^+$ are
  trivially satisfied, and thus $(M,w^-,w^+)$ is of type~A.

If $\bA$ is of spherical type, let $J_\Delta=\{(\Delta,i)\tq 1\leq
i\leq\abs{\Delta}\}$. Then the matrix $M$ has the following form:
\begin{align*}
  M&=
     \begin{pmatrix}
A&T\\0&\Mtilde
     \end{pmatrix}
&
\text{with\quad}A&=
  \begin{pmatrix}
    0&1&0&&\cdots&0\\
&0&1&0&\cdots&0\\
\vdots&&&&&\vdots\\
&&&\cdots&0&1\\
1&0&&&\cdots&0
  \end{pmatrix}
\end{align*}
where $A$ has size $\#J_\Delta\times\#J_\Delta$\,, and $\Mtilde$ is
irreducible according to Theorem~\ref{thr:1qqlknaa} and by the same reasoning
as above. The matrix $T$ is filled with~$0$s, except for its last line where
it has a $1$ at each column indexed by $(x,1)$ for
$x\in\SS\setminus\{\unit,\Delta\}$.

Since $A^{\#J_\Delta}=I$, all the eigenvalues of $A$ have modulus~$1$ and
thus the spectral radius of $A$ is~$1$. On the other hand, since $\bA$ is
assumed to have at least two generators and to be irreducible, there exist
elements $x,y,z \in \SS\setminus\{\unit,\Delta\}$ such that $y \neq z$, $x
\to y$ and $x \to z$. Hence, $\Mtilde$~is greater than a permutation matrix;
being primitive, $\Mtilde$~has a spectral radius greater than $1$ by the
Perron-Frobenius Theorem. Finally, the conditions on $w^-$ and $w^+$ are
trivially satisfied, hence $(M,w^-,w^+)$ falls into case~B.
\end{proof}

\begin{proposition}
  \label{prop:7}
  Let\/ $\bA$ be an irreducible Artin-Tits monoid with at least two
  generators. For each integer $k\geq0$, let
  $\lambda_k=\#\{x\in\bA\tq\abs{x}=k\}$. Then there exist two real
  constants $C>0$ and $p_0\in(0,1)$, that depend on\/~$\bA$, such
  that:
  \begin{gather}
    \label{eq:13}
    \lambda_k\sim_{k\to\infty} Cp_0^{-k}.
  \end{gather}
  The real $p_0$ is the inverse of the Perron eigenvalue of
  the \CWG\ associated to\/~$\bA$.
\end{proposition}
\begin{proof}
By~\eqref{eq:card_sphere}, we have $\lambda_k = w^- M^{k-1} w^+$. Hence,
putting $p_0=\lambda^{-1}$, where $\lambda$ is the Perron eigenvalue of~$M$,
we obtain the expected form according to Lemma~\ref{lem:pokqpkqa},
point~\ref{item:3asaspp}.
\end{proof}

\subsubsection{Möbius polynomial}
\label{sec:mobius-polynomial}

There is a nice combinatorial interpretation of the real $p_0$ introduced in
Proposition~\ref{prop:7}. It is much similar to the case of other monoids
such as braid monoids or trace monoids; see~\cite{abbes17} or~\cite{abbes15a}
for more details.

The \emph{Möbius polynomial} of an Artin-Tits monoid $\bA=\bA(\Sigma,\ell)$
is the polynomial $\mu_\bA\in\bbZ[T]$ defined by~:
\begin{gather*}
  \mu_\bA=\sum_{D\subseteqlub\Sigma}(-1)^{\abs{D}}T^{\abs{\bigveel D}},
\end{gather*}
where the notation $D\subseteqlub\Sigma$ has been introduced in
Definition~\ref{def:1qqpa}. Note that this is nothing but the polynomial
expression found for $h(\unit)$ in Section~\ref{sec:an-example}, where
$h(\cdot)$ was the graded Möbius transform of the uniform valuation
$f(x)=p^{\abs{x}}$ on~$\bA$.

Let also the \emph{growth series} $G\in\bbZ[[T]]$ be the formal series
defined by~:
\begin{gather}
\label{eq:8}
  G=\sum_{x\in\bA}T^{\abs{x}}=\sum_{k\geq0}\lambda_kT^k.
\end{gather}

Then $G$ is a rational series, inverse of the Möbius polynomial:
$G(T)=1/\mu_\bA(T)$; see a proof for a slightly more general result below in
Section~\ref{sec:multiplicative-case}. Since $G$ has non-negative terms, its
radius of convergence is one of its singularities by Pringsheim's
Theorem~\cite{flajolet09}. Since $G$ is rational with coefficients of the
form~\eqref{eq:13}, provided that $\bA$ is irreducible with at least two
generators, this singularity is necessarily of order~$1$, and there is no
other singularity of $G$ with the same modulus.

These facts reformulate as follows: \emph{If\/ $\bA$ is an irreducible
Artin-Tits monoid with at least two generators, the Möbius polynomial of\/
$\bA$ has a unique root of smallest modulus. This root is simple, real, lies
in~$(0,1)$, and coincides with the real $p_0$ introduced in
Proposition~~{\normalfont\ref{prop:7}}.}

\subsubsection{Multiplicative case}
\label{sec:multiplicative-case}

More generally, assume given a multiplicative and positive weight on the
elements of an Artin-Tits monoid~$\bA$, hence what we called a valuation
$\omega:\bA\to(0,+\infty)$. We associate to the pair $(\bA,\omega)$ the
following square matrix~$M$ with the same indices as in the uniform case
(Definition~\ref{def:8}), and the initial and final vectors $w^-$ and $w^+$
given by:
  \begin{gather*}
    M_{(x,i),(y,j)}=  \begin{cases}
    1,&\text{if $x=y$ and $j=i+1$}\\
\omega(y),&\text{if $x\to y$ and $i=\abs{x}$ and $j=1$}\\
0,&\text{otherwise}
  \end{cases}\\
\begin{aligned}
w^-(x,i)&=\un(i=1)\omega(x)\qquad
&w^+(x,i)&=\un(i=\abs{x})
  \end{aligned}
\end{gather*}

Then we claim that $(M,w^-,w^+)$ is a \CWG.

Indeed, if $\bA$ is not of spherical type, then the same arguments used in
the proof of Proposition~\ref{prop:8} show that $(M,w^-,w^+)$ is a \CWG\
since $M$ is primitive. Assume now that $\bA$ is of spherical type. If
$\omega$ is constant on $\Sigma$, so that $\omega (x)= p^{\abs{x}}$ for all
$x\in\bA$ and for some positive real number~$p$, then the same arguments used
in the proof of Proposition~\ref{prop:8} show that $(M,w^-,w^+)$ is a \CWG.
If, however, $\omega$~is non constant on~$\Sigma$, introduce again
$J_\Delta=\{(\Delta,i)\tq i\leq\abs{\Delta}\}$, $A$~the restriction of $M$ to
$J_\Delta\times J_\Delta$ and $\Mtilde$ the restriction of $M$ to
$(J\setminus J_\Delta)\times(J\setminus J_\Delta)$. Let $x \in \Sigma$ be
such that $\omega(x) = \max \omega(\Sigma)$. Since $\omega$ is a valuation,
and since $\Delta$ is divisible by some element $y \in \Sigma$ such that
$\omega(y) < \omega(x)$, it comes at once that $\omega(\Delta) <
\omega(x)^{\abs{\Delta}}$. Due to the loop $x \to x$, it follows that
$\rho(\Mtilde) \geq \omega(x) > \omega(\Delta)^{1/\abs{\Delta}} = \rho(A)$,
which proves that $(M,w^-,w^+)$ is a \CWG~in this case too.

Let $G_\omega$ be the generating series:
\begin{gather*}
  G_\omega=\sum_{x\in\bA}\omega(x)T^{\abs{x}}=\sum_{k\geq0}Z_\omega(k)T^k,\qquad
\text{with }Z_\omega(k)=\sum_{x\in\bA\tqs\abs{x}=k}\omega(x).
\end{gather*}
Then the coefficients $Z_\omega(k)$ have the following expression, for all
$k>0$:
\begin{gather}
\label{eq:12}
  Z_\omega(k)=w^-\cdot M^{k-1}\cdot w^+.
\end{gather}

The uniform case seen in Section~\ref{sec:uniform-case} corresponds to the
constant valuation $\omega(x)=1$, in which case $G_\omega$ is the growth
series~\eqref{eq:8} of the monoid. We note that $G_\omega=1/\mu_\omega$,
hence is rational, where $\mu_\omega\in\bbZ[T]$ is the following polynomial:
\begin{gather*}
\mu_\omega=\sum_{D\subseteqlub\SS}(-1)^{\abs{D}}\omega\bigl(\bigveel
D\bigr)T^{\abs{\bigveel D}}\,.
\end{gather*}

To prove the equality $G_\omega=1/\mu_\omega$, one has $\mu_\omega
G_\omega=\sum_{x\in\bA}a_xT^{\abs{x}}$, where $a_x$ is computed by:
\begin{multline*}
  a_x=\sum_{\substack{
D\subseteqlub\SS,\ y\in\bA\tqs\\\!\bigl(\bigveel D\bigr)\cdot y= x}}
(-1)^{\abs{D}}\omega\bigl(\bigveel
       D\bigr)\omega(y)
       \\
  =\omega(x)\Bigl(\sum_{D\subseteqlub
       \SS\tqs\bigveel D\leql x}(-1)^{\abs D}\Bigr)
  =\un(x=\unit).
\end{multline*}

\subsection{Parametrization of multiplicative measures}
\label{sec:param-mult-meas}

Assume given a valuation $\omega:\bA\to(0,+\infty)$ defined on an irreducible
Artin-Tits monoid with at least two generators. For each integer $k\geq1$,
let $\bA_k=\{x\in\bA\tq\abs{x}=k\}$ and let $m_{\omega,k}$~be the probability
distribution on $\bA_k$ proportional to~$\omega$:
\begin{gather*}
m_{\omega,k}(x)=\frac{\omega(x)}{Z_\omega(k)}\qquad\text{for $x\in\bA_k$\,.}
\end{gather*}

Then the finite probability space $(\bA_k,m_{\omega,k})$ is isomorphic to the
finite probability space of all paths of length $k-1$ in the \CWG\
$(M,w^-,w^+)$ equipped with the associated probability distribution from
Definition~\ref{def:4} (Section~\ref{sec:conv-unif-meas}). Furthermore, the
space of infinite paths in the \CWG\ $(M,w^-,w^+)$ is homeomorphic to the
boundary~$\BA$. By Theorem~\ref{thr:4}, we deduce that the sequence
$(m_{\omega,k})_{k\geq0}$ converges weakly toward a probability measure
$m_{\omega,\infty}$ on~$\BA$. The following result gives the form of this
limit measure.

\begin{theorem}
  \label{thr:3}
  Let $\omega:\bA\to(0,+\infty)$ be a valuation defined on an
  irreducible Artin-Tits monoid with at least two generators. Then the
  weak limit $m_{\omega,\infty}$ of the sequence of finite probability
  distributions $(m_{\omega,k})_{k\geq0}$ is a multiplicative measure
  on~$\BA$. Its associated valuation is given as follows, for any
  $x\in\bA$:
  \begin{gather*}
    m_{\omega,\infty}(\up x)=\lambda^{-\abs{x}}\omega(x),
  \end{gather*}
  where $\lambda$ is the Perron eigenvalue of the \CWG\ associated
  to~$\omega$.
\end{theorem}

\begin{proof}
  Recall that $\bAbar=\bA\cup\BA$ denotes the completion of~$\bA$, and
  $\Up x$ denotes the full visual cylinder with base~$x$ (see
  Definition~\ref{def:1poq}). Since the support of $m_{\omega,\infty}$
  is a subset of~$\BA$, one has
  $m_{\omega,\infty}(\up x)=m_{\omega,\infty}(\Up x)$. Since $\Up x$
  is both open and closed in~$\bAbar$, its topological boundary is
  empty, and therefore by~\cite[Th.~25.8]{billingsley95}:
\begin{gather*}
  m_{\omega,\infty}(\Up x)=\lim_{k\to\infty}m_{\omega,k}(\Up x).
\end{gather*}
Next, using that $\bA$ is left cancellative, we compute for $k\geq\abs{x}$:
\begin{align*}
  m_{\omega,k}(\Up x)&=\frac1{Z_\omega(k)}\Bigl(\sum_{y\in\bA_k\tqs x\leql y}\omega(y)\Bigr)
=\omega(x)\frac{Z_\omega(k-\abs{x})}{Z_\omega(k)}\,.
\end{align*}

Given the expression~\eqref{eq:12} for $Z_\omega(\cdot)$ on the one hand, and
the asymptotics from Lemma~\ref{lem:pokqpkqa} for the powers of~$M$ on the
other hand, we deduce $m_{\omega,\infty}(\up x)=\lambda^{-\abs{x}}\omega(x)$.
This is indeed a valuation, hence $m_{\omega,\infty}$ is a multiplicative
measure.
\end{proof}

\begin{corollary}
  \label{cor:4}
  Let\/ $\bA$ be an irreducible Artin-Tits monoid with at least two
  generators. Then there exists a unique non degenerate uniform
  measure $\nu$ on\/~$\BA$. It is characterized by
  $\nu(\up x)=p_0^{\abs{x}}$ for all $x\in\bA$, where $p_0$ is the unique
  root of smallest modulus of the Möbius polynomial of\/~$\bA$.
\end{corollary}

\begin{proof}
  The uniqueness of the non degenerate uniform measure has already
  been proved in Theorem~\ref{thr:2}. For the existence, let
  $\omega(x)=1$ be the constant uniform valuation on~$\bA$, and let
  $\nu=m_{\omega,\infty}$. Then, according to Theorem~\ref{thr:3}, we
  have $\nu(\up x)=\lambda^{-\abs{x}}$ for all $x\in\bA$ and for
  $\lambda$ the Perron eigenvalue of the \CWG\ associated
  with~$\omega$. We have seen in Section~\ref{sec:mobius-polynomial}
  that $\lambda=p_0^{-1}$, whence the result.
\end{proof}

As illustrated by the above corollary, Theorem~\ref{thr:3} provides a mean
for proving the existence of multiplicative measures. We shall see that all
multiplicative measures can be obtained as weak limits of such finite
`multiplicative distributions'. This yields in Theorem~\ref{thr:7} below a
parametrization of all multiplicative measures on the boundary.

From the operational point of view however, expressing a measure on the
boundary $\BA$ as a weak limit of finite distributions on $\bA$ does not
provide a realization result similar to Theorem~\ref{thr:1}. It is therefore
not much useful for simulation purposes for instance. Nevertheless, it yields
a way to obtain asymptotic information on these finite `multiplicative
distributions', which are of interest \emph{per se}. The latter aspect will
be developed in Section~\ref{sec:asymptotics}.

Let us first investigate the structure of valuations on an Artin-Tits monoid:
this task does not present any difficulty, and therefore the proof of the
following result is omitted.

\begin{proposition}

  \label{prop:4}
  Let\/ $\bA=\bA(\Sigma,\ell)$ be an Artin-Tits monoid. Let\/ $\Rbar$
  be the reflexive and transitive closure of the symmetric relation
  $R\subseteq\Sigma\times\Sigma$ defined by:
  \begin{gather*}
    R=\bigl\{(x,y)\in\Sigma\times\Sigma\tq x\neq
    y \text{ and } \ell(x,y)<\infty \text{ and } \ell(x,y)=1\!\!\!\mod 2\bigr\},
  \end{gather*}
and let $\R$ be the set of equivalence classes of~$\Rbar$.

Then, for any valuation $f:\bA\to(0,+\infty)$, and for any equivalence class
$r\in\R$, the value $f(a)$ is constant for $a$ ranging over~$r$. Conversely,
if $x_r\in(0,+\infty)$ is arbitrarily fixed for every $r\in\R$, then there
exists a unique valuation $f:\bA\to(0,+\infty)$ such that $f(a)=x_{r(a)}$ for
every $a\in\Sigma$, where $r(a)$ is the equivalence class of~$a$.
\end{proposition}

Valuations on $\bA$ are thus in bijection with the product set
$\FF=(0,+\infty)^K$, where $K$ is the number of equivalence classes
of~$\Rbar$. Let $\M$ denote the subset of $\FF$ corresponding to parameters
of Möbius valuations, hence those associated with a non degenerate
multiplicative measure. Theorem~\ref{thr:7} below shows that $\M$ has a
familiar topological structure.

We first introduce the following notation. If $\omega:\bA\to(0,\infty)$ is a
valuation, and if $\kappa$ is a positive real number, then $\kappa\omega$
denotes the valuation on $\bA$ defined by
$(\kappa\omega)(x)=\kappa^{\abs{x}}\omega(x)$ for all $x\in\bA$. The
half-line $(0,\infty)\omega$ is the set of valuations of the form
$\kappa\omega$ for $\kappa$ ranging over~$(0,\infty)$.

\begin{theorem}
  \label{thr:7}
  Let\/ $\bA$ be an irreducible Artin-Tits monoid with at least two
  generators. Let $K$ be the number of equivalence classes of the
  relation $\Rbar$ from Proposition~\ref{prop:4}.  Then the set
  $\M \subseteq (0,+\infty)^K$ of parameters of non degenerate
  multiplicative measures defined on\/ $\BA$ intersects any half-line
  $(0,+\infty) \omega$ in exactly one point. This gives a
  homeomorphism between $\M$ and the open simplex
  $\mathbb{P}\bigl( (0,+\infty)^K\bigr)$ of dimension~$K-1$ (if $K=1$,
  the open simplex of dimension $0$ reduces to a singleton).
\end{theorem}

We first need the following lemma, which generalizes the uniqueness result of
uniform measures proved in Theorem~\ref{thr:2} with the same technique of
proof.

\begin{lemma}
  \label{lem:7}
  Let\/ $\bA$ be an irreducible Artin-Tits monoid with at least two
  generators. Let $\nu$ and $\nu'$ be two multiplicative non
  degenerate measures on~$\BA$, with associated valuations $f$
  and~$f'$. Assume that there exists a constant $\kappa>0$ such that
  $f'=\kappa f$. Then $f=f'$.
\end{lemma}

\begin{proof}
Without loss of generality, we assume that $\kappa\leq1$. Let $\SS$ be the
  smallest Garside subset of~$\bA$. Let $B$ and $B'$ be the square
  non-negative matrices, indexed by
  $(\SS\setminus\{\unit,\Delta\})\times(\SS\setminus\{\unit,\Delta\})$,
  with $\Delta$ to be ignored if $\bA$ is not of spherical type, and
  defined by:
  \begin{align*}
    B_{x,x'}&=\un(x\to x')f(x'),&B'_{x,x'}&=\un(x\to x')f'(x').
  \end{align*}
  Then $B'\leq B$.  According to Lemma~\ref{lem:3},
  point~\ref{item:8}, both matrices are primitive of the same spectral
  radius~$1$. According to Perron-Frobenius
  Theorem~\cite{seneta81}, it implies $B=B'$ and thus $f=f'$.
\end{proof}

\begin{proof}[Proof of Theorem~\ref{thr:7}.]
  Let $K$ be defined as in the statement, and let
  $\FF=(0,+\infty)^K$. We identify the set of valuations on $\bA$ with
  the product set~$\FF$, which is justified by
  Proposition~\ref{prop:4}. Let $\omega$ be a valuation on~$\bA$, and
  let $(M,w^-,w^+)$ be the \CWG\ associated to~$\omega$. Let also
  $\lambda$ be the Perron eigenvalue of~$M$, and let
  $f=\lambda^{-1}\omega$. The valuation $f$
  corresponds to a non degenerate multiplicative measure according to
  Theorem~\ref{thr:3}. This association defines thus a mapping
  $\Phi:\FF\to\M$.

  We first prove that $\Phi(r)=r$ for every $r\in\M$.
  If $r\in\M$, then both $r$ and $\Phi(r)$ correspond to non
  degenerate multiplicative measures, and they are related for some
  constant $\kappa>0$ by $\Phi(r)=\kappa r$. According to
  Lemma~\ref{lem:7}, it implies that $\Phi(r)=r$, as claimed. In
  particular, we deduce that $\Phi$ is onto. By the same lemma, we
  also observe that $\Phi$ is constant on all half-lines $(0,\infty)\omega$, for any $\omega\in\FF$.

  Let $\R$ be the set of equivalence classes of~$\Rbar$, and let $S$
  be the open simplex of dimension $K-1$ defined by:
  \begin{gather*}
    S=\bigl\{r=(r_a)_{a\in\R}\in\FF\tq\sum_{a\in\R}r_a=1\bigr\},
  \end{gather*}
  and let $\Psi=\Phi\rest S$. Since $\Phi$ is onto, and constant on
  the half-lines of~$\FF$, it follows at once that $\Psi$ is onto. If
  $\Psi(f)=\Psi(f')$ for $f$ and $f'$ in~$S$, then $f$ and $f'$ are
  related by a relation of the form $f'=\kappa f$ for some $\kappa>0$,
  and thus obviously $f=f'$. Hence $\Psi$ is one-to-one and thus
  bijective.

  The mapping which associates to a family $r\in S$ the corresponding
  matrix $M$ is obviously continuous, as well as the one mapping a
  primitive matrix to its Perron eigenvalue (see~\cite[Fact 10 of
  Section~9.2]{rothblum07}). Hence $\Psi$ is continuous. And
  $\Psi^{-1}$ is continuous since, for $f\in\M$, the unique
  $\omega\in S$ such that $\Psi(\omega)=f$ is given by:
  \begin{gather*}
    \forall a\in\R\quad \omega(a)=\frac{f(a)}{\sum_{b\in\R}f(b)}.
  \end{gather*}
  Hence $\Psi$ is a homeomorphism.
\end{proof}

\section{Asymptotics: concentration of ergodic means\\ and Central Limit
  Theorem}
\label{sec:asympt-conc-ergod}

\subsection{Case of Conditioned Weighted Graphs}
\label{sec:case-cond-weight}

Consider a conditioned weighted graph $\GG=(M,w^-,w^+)$, with $M$ of
order~$N$, and a cost function $f:\{1,\ldots,N\}\to\bbR$. The \emph{ergodic
sums} $S_kf(z)$ and \emph{ergodic means} $R_kf(z)$ of $f$ along a path
$z=(x_0,\ldots,x_k)$ of length $k$ are defined by:
\begin{align}
\label{eq:35}
S_kf(z)&=\sum_{i=0}^kf(x_i)\,,& R_k
f(z)&=\frac1{k+1}\sum_{i=0}^kf(x_i)\,.
\end{align}

For each integer $k\geq0$, let $G_k$ denote the set of paths in $\GG$ of
length~$k$, equipped with the associated weighted distribution as in
Definition~\ref{def:4}.  The function $R_kf:G_k\to\bbR$ can be seen as a
random variable on the discrete probability space $(G_k,\mu_k)$. Since this
collection of random variables are not defined on the same probability
spaces, the only way we have to compare them is to consider their laws and
their convergence in law.

A weak variant of the Law of large numbers, adapted to the framework of
conditioned weighted graphs, is the following concentration result---to be
refined in a Central Limit Theorem next.

\begin{theorem}
  \label{thr:5}
  Let\/ $\GG=(M,w^-,w^+)$ be a conditioned weighted graph, with $M$ of
  order~$N$. Let $\pi$ be the stationary measure on\/ $\{1,\ldots,N\}$
  associated with the limit weighted measure of\/~$\GG$.

  Then, for every function $f:\{1,\ldots,N\}\to\bbR$, the sequence of
  ergodic means $(R_kf)_{k\geq0}$, with $R_kf$ defined on $(G_k,\mu_k)$
  by\/~\eqref{eq:35}, converges in distribution toward the Dirac
  measure~$\delta_{\gamma_f}$\,, where $\gamma_f$ is the constant defined by:
\begin{gather}
\label{eq:37}
  \gamma_f=\sum_{j=1}^N\pi(j)f(j)\,.
\end{gather}
\end{theorem}

\begin{proof}
  Using the characteristic functions, it is enough to show the
  following convergence, for every real~$t$\,:
  \begin{gather}
    \label{eq:36}
\lim_{k\to\infty}\esp_{\mu_k}(e^{\ii t R_kf})=e^{\ii t\gamma_f}\,,
  \end{gather}
where $\esp_{\mu_k}(\cdot)$ denotes the expectation with respect
to~$\mu_k$\,, and $\gamma_f$ is the constant defined in~\eqref{eq:37}.

We express the above expectation in matrix terms, as follows. For each
complex number~$u$, let $M_f(u)$ be the complex-valued matrix defined by:
\begin{align*}
  M_f(u)&=D_f(u)M\,,&D_f(u)=\Diag(e^{uf})\,,
\end{align*}
where the last matrix is the diagonal matrix with entry $e^{u f(i)}$ at
position $(i,i)$.

Let also $v_f(u)$ be the column vector defined by $v_f(u) = D_f(u) w^+$.
Then:
\begin{align*}
  \esp_{\mu_k}(e^{\ii t R_kf})&=\frac1{Z_k}w^-\bigl(M_f(u)\bigr)^k
  v_f(u)\,,&\text{with }u&=\frac {\ii t}{k+1}\,,
\end{align*}
where $Z_k=w^-\cdot M^k\cdot w^+$ is the normalization factor.

For small values of~$u$, $D_f(u)$ is an analytic perturbation of the
identity. By assumption, $M$~has a unique eigenvalue of maximal modulus,
say~$\lambda$, and it is simple; the same spectral picture persists for
$M_f(u)$ for small values of~$u$. Hence, according to~\cite[Theorem
III.8]{hennion_herve}, denoting by $\lambda(u)$ the eigenvalue of highest
modulus of~$M_f(u)$, by $\lev(u)$ and by $\rev(u)$ the unique left and right
associated eigenvectors normalized by the conditions $\lev\cdot\rev(u)=1$ and
$\lev(u)\cdot\rev(u)=1$, all these quantities are analytic in $u$ around
zero, and for $u$ small enough:
\begin{align*}
  M_f(u)&=\lambda(u)\Pi(u)+Q(u)\,,&
\text{with\quad}\Pi(u)&=\rev(u)\cdot\lev(u),
\end{align*}
where the spectrum of $Q(u)$ is included in a fixed disk of
radius~$<\lambda$, and $\Pi(u)\cdot Q(u)=Q(u)\cdot\Pi(u)=0$. Raising to the
power $k$ yields:
\begin{gather*}
  \bigl(M_f(u)\bigr)^k=\lambda(u)^k \Pi(u)+Q(u)^k\,,\quad \norm{Q(u)^k}= O\bigl((\lambda\varepsilon)^k\bigr)
\end{gather*}
for any spectral norm $\norm{\cdot}$ and for some $0<\varepsilon<1$.

Put $\lev=\lev(0)$ and $\rev=\rev(0)$, consistently with our previous
notation for the pair $(\lev,\rev)$.  For $t$ fixed and for $k$ large enough,
$u=\frac t{k+1}$ eventually reaches the region of validity of the above
estimate. Hence, recalling that $Z_k = w^-\cdot M^k\cdot w^+= w^-\cdot
\bigl(M_f(0)\bigr)^k\cdot w^+$,
\begin{align*}
  \esp_{\mu_k}(e^{\ii t R_kf})&=_{k\to\infty}\Biggl(\frac{\lambda(u)}{\lambda}\Biggr)^k
                                \frac{w^-\cdot\Pi(u)\cdot v_f(u)}{w^-\cdot \Pi(0)\cdot v_f(0)} + o(1)\,.
\end{align*}
As the last fraction on the right tends to $1$ when $u\to 0$, this gives
\begin{align*}
  \esp_{\mu_k}(e^{\ii t R_kf})&=_{k\to\infty}\Biggl(\frac{\lambda(u)}{\lambda}\Biggr)^k(1+o(1)) + o(1)\,.
\end{align*}

Passing to the limit as $k\to\infty$ above yields, using the development of
$\lambda(\cdot)$ around zero at order~$1$:
\begin{gather}
\label{eq:38}
  \lim_{k\to\infty}\esp_{\mu_k}(e^{\ii tR_kf})
  =\lim_{k\to\infty}\Biggl(\frac{\lambda\bigl(\frac{\ii t}{k+1}\bigr)}{\lambda}\Biggr)^k
  =e^{\ii \frac{\lambda'(0)}{\lambda}t}\,.
\end{gather}

It remains to evaluate $\lambda'(0)$. For this, we first differentiate the
equality $\ell \cdot r(u)=1$ at $0$ and obtain $\ell\cdot r'(0)=0$. We also
differentiate the equality $M_f(u)\cdot\rev(u)=\lambda(u)\rev(u)$ at $0$ and
obtain:
\begin{gather*}
  M_f'(0)\cdot\rev+M\cdot \rev'(0)=\lambda'(0)\rev+\lambda\rev'(0)\,.
\end{gather*}
We multiply both members of the above equality by $\lev$ on the left to
derive:
\begin{gather*}
  \lev\cdot M_f'(0)\cdot \rev=\lambda'(0)\,,\qquad\text{since
    $\ell\cdot r'(0)=0$, $\ell\cdot M=\lambda\ell$ and
    $\ell\cdot r=1$.}
\end{gather*}
By the definition $M_f(u)=D_f(u)\cdot M$, we have $M'_f(0)\cdot \rev
=\Diag(f)\cdot  M\cdot r = \lambda \Diag(f)\cdot r$, and thus:
\begin{gather*}
  \lambda'(0)= \lambda \lev\cdot \Diag\bigl(f\bigr)\cdot \rev=\lambda\sum_{j=1}^Nf(j)\lev(j)\rev(j)
  = \lambda\gamma_f\,.
\end{gather*}

Returning to~\eqref{eq:38}, we deduce the validity of~\eqref{eq:36}, which
completes the proof.
\end{proof}

Extending the analysis to the next order yields a Central Limit Theorem.

\begin{theorem}
  \label{thr:6}
  Let $\GG=(M,w^-,w^+)$ be a conditioned weighted graph, with $M$ of
  order~$N$. Let $\pi$ be the stationary measure on\/ $\{1,\ldots,N\}$
  associated with the uniform measure at infinity of\/~$\GG$. Let
  $f:\{1,\ldots,N\}\to\bbR$ be a function.

  Then there exists a non-negative $\sigma^2$ such that the following
  convergence in law with respect to $\mu_k$ toward a Normal law holds:
  \begin{gather*}
\frac1{\sqrt {k+1}}\bigl( S_kf-(k+1)\gamma_f\bigr)    \cvlaw{k\to\infty}\NN(0,\sigma^2)\,.
  \end{gather*}

  Moreover, $\sigma^2=0$ if and only if there exists a function
  $g:\{1,\ldots,N\}\to\bbR$ such that $f(i)=\gamma_f+g(i)-g(j)$
  whenever $M_{i,j}>0$ and $i, j \in \supp \pi$\,.
\end{theorem}
\begin{proof}
  We keep the same notations introduced in the proof of
  Theorem~\ref{thr:5}. Replacing $f$ by $f-\gamma_f$ if necessary, we
  assume without loss of generality that $\gamma_f=0$.  We evaluate
  the characteristic function of $(S_kf)/\sqrt {k+1}$\,:
\begin{align}
\label{eq:39}
    \esp_{\mu_k}\bigl(e^{\ii \frac t{\sqrt
    {k+1}}S_kf}\bigr)&=_{k\to\infty}\Biggl(\frac{\lambda\bigl(\frac{\ii t}{\sqrt{k+1}}\bigr)}{\lambda} \Biggr)^k(1+o(1)) + o(1)\, .
\end{align}

Since $\gamma_f=0$, it follows from the computations performed previously that
$\lambda'(0)=0$, and thus, for $\kappa=\lambda''(0)/(2\lambda)$, around zero:
\begin{gather*}
\lambda(u)=\lambda (1+\kappa u^2)+o(u^2)\,.
\end{gather*}

Passing to the limit in~\eqref{eq:39} yields:
\begin{gather*}
  \lim_{k\to\infty} \esp_{\mu_k}\bigl(e^{\ii \frac t{\sqrt {k+1}}S_kf}\bigr)
  =e^{-\kappa t^2}\,.
\end{gather*}

Since the left member in the above equation is uniformly bounded in $t$ and
in~$k$, it entails that $\kappa=\sigma^2\geq0$, and this proves the
convergence in law of $(S_kf)/\sqrt{k+1}$ toward~$\NN(0,\sigma^2)$.

To prove the non degeneracy criterion, we need to compute the second
derivative of $\lambda$. First, we claim that we can assume that $r(0)(i)>0$
for all $i$. Indeed, in the general case, we can partition $\{1,\ldots,N\}$
into the set where $r(0)>0$ and the set where $r(0)=0$. This gives rise to a
block-triangular decomposition of $M$, with diagonal blocks $A$ (on $r(0)>0$)
and $B$ (on $r(0)=0$). The spectrum of $M$ is the union of the spectra of $A$
and $B$, hence $A$ has $\lambda$ as a unique eigenvalue of maximal modulus,
while $B$ has strictly smaller spectral radius. When one perturbs $M$ into
$M(u)$, the block-triangular decomposition survives. By continuity of the
spectrum, the dominating eigenvalue $\lambda(u)$ of $M(u)$ is also the
dominating eigenvalue of $A(u)$, which is of the same type except that $r(0)$
is everywhere positive for $A$.

From now on, we assume that $r(0)(i)>0$ for all $i$. To compute the second
derivative of $\lambda$, it is more convenient to reduce by conjugation to a
situation where $\rev'(0)=0$, as follows. The vector $\widetilde \rev(u) =
\Diag(e^{-u \rev'(0)/\rev(0)})\cdot \rev(u)$ is equal to $\rev + O(u^2)$, and
it is an eigenvector of the matrix
\begin{equation*}
  \widetilde M_f(u) = \Diag\bigl(e^{-u \rev'(0)/\rev(0)}\bigr)\cdot  M_f(u)\cdot  \Diag\bigl(e^{u \rev'(0)/\rev(0)}\bigr)\,,
\end{equation*}
for the eigenvalue $\lambda(u)$. Differentiating twice the equality
$\lambda(u) \widetilde \rev(u) = \widetilde M_f(u)\cdot  \widetilde \rev(u)$
yields
\begin{equation*}
  \lambda''(0) \rev + 2 \lambda'(0) \widetilde \rev'(0) + \lambda \widetilde \rev''(0) =
  \widetilde M_f''(0)\cdot  \rev + 2 \widetilde M_f'(0)\cdot  \widetilde \rev'(0) + \widetilde M_f(0)\cdot  \widetilde \rev''(0)
  \,.
\end{equation*}
By construction, $\widetilde \rev'(0)=0$. Multiplying this equation on the
left by~$\lev$, the terms $\lambda \widetilde \rev''(0)$ and $\widetilde
M_f(0)\cdot  \widetilde \rev''(0)$ cancel each other as $\ell\cdot
\widetilde M_f(0) = \ell\cdot  M = \lambda\ell$. Thus,
\begin{equation*}
  \lambda''(0) = \ell\cdot  \widetilde M_f''(0)\cdot  \rev.
\end{equation*}
Moreover:
\begin{gather*}
  \bigl(\widetilde M_f(u)\bigr)_{i,j} = M_{i,j} \cdot \exp u\Bigl(f(i)-\frac{\rev'(0)(i)}{\rev(0)(i)} +
  \frac{\rev'(0)(j)}{\rev(0)(j)}\Bigr)\,.
\end{gather*}

Finally, writing $g(i) = \rev'(0)(i)/\rev(0)(i)$, we get the formula
\begin{equation*}
  \lambda''(0) = \sum_{i,j} \lev(i) \bigl( f(i) - g(i) + g(j) \bigr)^2 M_{i,j} \rev(j).
\end{equation*}
Note that $g$ is real-valued, as the matrix $D_f(u)$ for small real $u$ is
real-valued and has a dominating real-valued eigenvector.

If follows from this expression that, if the variance vanishes, then $f(i) =
g(i)-g(j)$ whenever $M_{i,j}>0$ and $\lev(i)>0$ and $\rev(j)>0$ (these last
two conditions are satisfied if $i$ and $j$ belong to the support of $\pi$).

Conversely, assume that there exists $g$ such that $f(i) = g(i)-g(j)$
whenever $i,j \in \supp \pi$ and $M_{i,j}>0$. Then, along any path
$x=(x_0,\dotsc, x_k)$ in the graph with nonzero weight and in the support of
$\pi$, the quantity $S_k f = g(x_0) - g(x_k) + f(x_k)$ is uniformly bounded.
For the limit weighted measure $\mu$, almost every path enters the support of
$\pi$ by ergodicity. Hence, $S_k f$~remains bounded along almost every path.
In particular, for any $\epsilon>0$, there is a clopen set $K_\epsilon$ of
$\mu$-measure $>1-\epsilon$ on which $S_k f$ is bounded for all $k$ by a
constant~$C(\epsilon)$. As $\mu_k$ converges weakly to $\mu$ by
Theorem~\ref{thr:4}, it follows that $\mu_k(K_\epsilon)>1-\epsilon$ for large
enough~$k$. Hence, $S_k f$ is also bounded by $C(\epsilon)$ with
$\mu_k$-probability $1-\epsilon$. This shows that $S_k f / \sqrt{k+1}$
converges in distribution with respect to $\mu_k$ towards the Dirac mass
at~$0$.
\end{proof}

\subsection{Asymptotics for Artin-Tits monoids}
\label{sec:asymptotics}

Let an Artin-Tits monoid~$\bA$, that we assume to be irreducible and with at
least two generators, be equipped with a Möbius valuation~$\omega$. Let
$(M,w^-,w^+)$ be the associated \CWG.  We have already observed that, for
every integer $k\geq0$, the finite probability distribution $m_{\omega,k}$ on
$\bA_k=\{x\in\bA\tq\abs{x}=k\}$ which is proportional to $\omega$ corresponds
to the weighted distribution on the set of paths of length $k-1$ in the \CWG.
Let $m_{\omega,\infty}$ be the weak limit on $\BA$
of~$(m_{\omega,k})_{k\geq0}$\,.

Let $\Xi$ denote the space of infinite paths in the \CWG, equipped with the
limit weighted measure~$\mtilde$. Then the natural correspondence between
$\BA$ and $\Xi$ makes the two probability measures $m_{\omega,\infty}$ and
$\mtilde$ image of each other.

By Theorems~\ref{thr:1} and~\ref{thr:4}, both measures correspond to finite
homogeneous Markov chains, but on two different finite sets of states: the
set $\SS\setminus\{\unit\}$ for~$m_{\omega,\infty}$ and the set
$J=\bigl\{(x,i)\in (\SS\setminus\{\unit\})\times\bbN\tq 1\leq
i\leq\abs{x}\bigr\}$ for~$\mtilde$. We wish to relate the stationary measures
of these chains, that is to say, the finite probability distributions which
are left invariant with respect to the transition matrices.

\begin{lemma}
  \label{lem:8}
Let $f:\bA\to(0,+\infty)$ be the  valuation corresponding to a non degenerate multiplicative measure $m$ on the boundary at infinity of an Artin-Tits monoid\/~$\bA$. Let $P$ be the transition matrix of the Markov chain on $\SS\setminus\{\unit\}$ associated to~$m$, and let $\theta$ be the unique probability vector on~$\SS\setminus\{\unit\}$ left invariant for~$P$.

Let $(M,w^-,w^+)$ be the\/ \CWG\ associated with~$f$ as described in Section~{\normalfont\ref{sec:multiplicative-case}}, and let $\Ptilde$ be the transition matrix on $J=\bigl\{(x,i)\in (\SS\setminus\{\unit\})\times\bbN\tq 1\leq
i\leq\abs{x}\bigr\}$ of the Markov chain corresponding to the limit weighted measure of the\/ \CWG\ (see Definition\/~{\normalfont\ref{def:5}}).

Then the probability vector  $\thetatilde$ on $J$ defined by:
\begin{align*}
\thetatilde(x,i)&=\frac1\kappa\theta(x),&\text{with\quad}
\kappa=\sum_{x\in\SS\setminus\{\unit\}}\abs{x}\theta(x),
\end{align*}
is left invariant for~$\Ptilde$.
\end{lemma}

\begin{proof}
We put $\SSp=\SS\setminus\{\unit\}$ to shorten the notations. Let $h$ be the Möbius transform of~$f$, and let $g:\SSp\to\bbR$ be the normalization vector defined on $\SSp$ by:
\begin{gather*}
  \forall y\in\SSp\qquad g(x)=\sum_{y\in\SSp\tqs x\to y}h(y).
\end{gather*}

It follows from Theorem~\ref{thr:1} that $f$ is a Möbius valuation, hence $h(\unit)=0$. Therefore $g$ coincides on $\SSp$ with the function $g$ defined in Lemma~\ref{lem:4}, and thus $h(x)=f(x)g(x)$ holds in particular for all $x\in\SSp$. Furthermore, since $m$ is non degenerate, $g>0$ according to Lemma~\ref{lem:3}.

Writing down the equation $\theta P=\theta$ yields, for every $x\in\SSp$, and using the expression for $P$ given by Theorem~\ref{thr:1}:
\begin{align*}
  \theta(x)&=\sum_{y\in\SSp}\theta(y)P_{y,x}=\sum_{y\in\SSp\tqs y\to x}f(y)\frac{\theta(y)}{h(y)}h(x)\,.
\end{align*}
Using the identities $h(x)=f(x)g(x)$ and $f(y)/h(y)=1/g(y)$, we obtain:
\begin{gather}
\label{eq:26}
f(x)\cdot\Biggl(\;  \sum_{y\in\SSp\tqs y\to x}\frac{\theta(y)}{g(y)}\Biggr)=\frac{\theta(x)}{g(x)}\,.
\end{gather}

Now we claim:
\begin{intermediate}{$(\dag)$}
   The  vectors $u$ and $v$ defined, for $(x,i)\in J$, by $u_{(x,i)}=\theta(x)/g(x)$ and $v_{(x,i)}=g(x)$, are respectively left and right invariant for~$M$, hence satisfy $u\cdot M=u$ and $M\cdot v=v$.
\end{intermediate}
The claim follows from the following computations, referring to the definition for the matrix $M$ given in Section~\ref{sec:multiplicative-case}:
\begin{align*}
  (u\cdot M)_{(x,1)}&=\sum_{y\in\SSp\tqs y\to x}\frac{\theta(y)}{g(y)}f(x)=u_{(x,1)}\,,
\end{align*}
where the last equality derives from~\eqref{eq:26}. And for $i>1$:
\begin{align*}
  (u\cdot M)_{(x,i)}&=u_{(x,i-1)}=u_{(x,i)}\,.
\end{align*}
This proves that $u$ is left invariant for~$M$. To prove the right invariance of~$v$, we compute as follows:
\begin{align*}
\text{for $i=|x|$:}&\hspace{-6pt}&(M\cdot v)_{(x,\abs x)}&=\sum_{y\in\SSp\tqs x\to y\hspace{-1.5pt}}f(y)g(y)=\sum_{y\in\SSp\tqs x\to y\hspace{-1.5pt}}h(y)=g(x)=v_{(x,\abs x)}\\
\text{for $i<|x|$:}&\hspace{-6pt}&(M\cdot v)_{(x,i)}&=g(x)=v_{(x,i)}
\end{align*}
This proves the claim~$(\dag)$.

Thus, the stationary distribution of $\Ptilde$ is proportional to the vector $u(x,i)v(x,i)=\theta(x)$ and is a probability vector; whence the result.
\end{proof}

\begin{definition}
  \label{def:1}
Let\/ $\bA$ be an irreducible Artin-Tits monoid with at least two generators,
equipped with a valuation $\omega:\bA\to(0,+\infty)$. The \emph{speedup} of
$\omega$ is the quantity:
\begin{gather*}
  \kappa=\sum_{x\in\SS\setminus\{\unit\}}\abs{x}\theta(x),
\end{gather*}
where $\SS$ is the smallest Garside subset of\/~$\bA$, and $\theta$ is the
stationary distribution of the Markov chain on $\SS\setminus\{\unit\}$
associated with the multiplicative measure~$m_{\omega,\infty}$ on\/~$\BA$.
\end{definition}

We note that $\theta(\Delta)=0$ if $\bA$ is of spherical type.  We now come
to the study of asymptotics for combinatorial statistics defined on
Artin-Tits monoids.

\begin{definition}
  \label{def:9}
  Let\/ $\bA$ be an Artin-Tits monoid. A function $F:\bA\to\bbR$ is
  said to be:
\begin{enumerate}
\item \emph{Additive} if $F(x\cdot y)=F(x)+F(y)$ for all $x,y\in\bA$.
\item \emph{Normal-additive} if $F(x_1\cdot\ldots\cdot
    x_n)=F(x_1)+\ldots+F(x_n)$ whenever $(x_1,\dots,x_n)$ is a normal
    sequence of\/~$\bA$.
\end{enumerate}
\end{definition}

Additive functions are normal additive, but the converse needs not to be
true. For instance, the height function is normal additive without being
additive in general.

\begin{theorem}
  \label{thr:8}
Let\/ $\bA$ be an irreducible Artin-Tits monoid with at least two generators. Let $\omega:\bA\to(0,+\infty)$ be a valuation, and let $F:\bA\to\bbR$ be a normal-additive function.

If $k>0$ is an integer, we let $x$ denote a random element in $\bA_k$
distributed according to the finite distribution~$m_{\omega,k}$\,. Let
$\theta$ be the stationary distribution on $\SS\setminus\{\unit\}$ of the
Markov chain associated with the weak limit
$m_{\omega,\infty}=\lim_{k\to\infty} m_{\omega,k}$\,.

Then the following convergence in distribution holds:
\begin{align*}
  \frac{F(x)}{\abs{x}}&\cvlaw{k\to\infty}\delta_\gamma\,,&\text{with\quad}
\gamma&=\frac1\kappa\sum_{x\in\SS\setminus\{\unit,\Delta\}}\theta(x)F(x),
\end{align*}
where $\kappa$ is the speedup of~$\omega$, and where $\Delta$ is to be
ignored if\/ $\bA$ is not of spherical type.

Assume furthermore that:
\begin{inparaenum}[(1)]
  \item $\bA$~has at least three generators in case that it is of spherical type, and
  \item $F$~is not proportional on~$\SS\setminus\{\unit,\Delta\}$ to
the length function (with $\Delta$ to be ignored if\/ $\bA$ is
not of spherical type).
\end{inparaenum}
Then there exists a constant $s^2>0$ such that the
following convergence in distribution toward a Normal law holds:
\begin{align*}
  \frac1{\sqrt k}\bigl( F(x)-k\gamma\bigr)&\cvlaw{k\to\infty}\NN(0,s^2).
\end{align*}
\end{theorem}

\begin{proof}
We assume without loss of generality that $\omega$ is a M\"obius valuation. For otherwise, we normalize it by putting $\omega'=\kappa\omega$, for the unique positive real $\kappa$ such that, according to Theorem~\ref{thr:7}, the resulting valuation $\omega'$ is a Möbius valuation. Then the weighted distributions associated with $\omega'$ are the same as the weighted distributions associated with~$\omega$, and the convergences that we shall establish for $\omega'$ correspond to the convergences for~$\omega$.

Let $(M,w^-,w^+)$ be the \CWG\ associated to~$\omega$. We express $F(x)$ for
$x\in\bA_k$ as an ergodic sum in order to apply Theorem~\ref{thr:5}. For
this, let $\Ftilde:J\to\bbR$, with
$J=\{(x,i)\in(\SS\setminus\{\unit\})\times\bbN\tq 1\leq i\leq \abs{x}\}$, be
defined by:
\begin{gather*}
  \Ftilde(x,i)=
  \begin{cases}
    0,&\text{if $i<\abs{x}$}\\
F(x)&\text{if $i=\abs{x}$}
  \end{cases}
\end{gather*}
Then, for $x\in\bA_k$, if $\xtilde$ is the corresponding path in the \CWG\ as
in~\eqref{eq:20}, one has $F(x)=S_{k-1}\Ftilde(\xtilde)$, where
$S_{k-1}\Ftilde$ denotes the ergodic sums associated to~$\Ftilde$, and thus
$F(x)/\abs{x}=R_{k-1}\Ftilde(\xtilde)$ where $R_{k-1}\Ftilde$ denotes the
ergodic means associated to~$\Ftilde$. Let $\thetatilde$ be the stationary
distribution given by Theorem~\ref{thr:4} applied to the \CWG\ $(M,w^-,w^+)$.
Then Theorem~\ref{thr:5} entails the convergence in distribution:
\begin{align*}
  \frac{F(x)}{\abs{x}}&\cvlaw{k\to\infty}\delta_\gamma\,,&\text{with\quad}\gamma&=
\sum_{(x,i)\in J}\thetatilde(x,i)\Ftilde(x,i)=\frac1\kappa\sum_{x\in\SS\setminus\{\unit\}}\theta(x)F(x),
\end{align*}
where the last equality comes from Lemma~\ref{lem:8}. If
$\bA$ is of spherical type, then $\theta(\Delta)=0$ hence the sum ranges over
$\SS\setminus\{\unit,\Delta\}$.

We now aim at applying the Central Limit Theorem~\ref{thr:6}.  For this, let
$H(x)=F(x)-\gamma\abs{x}$ and let $\Htilde$ be defined with respect to $H$ in
the same way than $\Ftilde$ was defined with respect to~$F$. We prove that
$\Htilde$ satisfies the non degeneration criterion stated in
Theorem~\ref{thr:6}. For this, since $\gamma_H=0$, assume that there exists a
function $g(\cdot,\cdot)$ on $J$ such that $\Htilde(x,i) = g(x,i)-g(y,j)$
whenever $M_{(x,i),(y,j)}>0$ and $(x,i), (y,j)$ belong to the support of the
invariant measure of the Markov chain, \ie, $x, y \in
\SS\setminus\{\unit,\Delta\}$ (with $\Delta$ appearing only if $\bA$ is of
spherical type).

For any $x\in\SS\setminus\{\unit,\Delta\}$, and for $i<\abs{x}$, one has $\Htilde(x,i)=0=g(x,i)-g(x,i+1)$, hence the value of $g(x,i)$ is independent of $i\in\{1,\dots,\abs{x}\}$. And for all $y$ such that $x\to y$ holds, $\Htilde(x,\abs{x})=g(x,\abs{x})-g(y,1)$. Therefore the value of $g(y,1)$ is independent of~$y$, provided that $x\to y$ holds. 

To prove that $g$ is globally constant, it is thus enough to show that the transitive closure on $\SS\setminus\{\unit,\Delta\}$ of the relation relating $x$ and $y$ if there exists $z$ such that $z\to x$ and $z\to y$, is $\SS\setminus\{\unit,\Delta\}$. In turn, this derives from the following claim:
\begin{intermediate}{$(\dag)$}
  For every $x\in\SS\setminus\{\unit,\Delta\}$ (with $\Delta$ appearing only if $\bA$ is of spherical type), let $D(x)=\bigl\{y\in\SS\setminus\{\unit,\Delta\}\tq x\to y\bigr\}$. Then there exists a subset $I$ of $\SS\setminus\{\unit,\Delta\}$ such that:
  \begin{gather}
    \label{eq:25}
\forall x,y\in I\quad x\neq y\implies D(x)\cap D(y)\neq\emptyset,\\
\label{eq:27}
\forall y\in\SS\setminus\{\unit,\Delta\}\quad\exists x\in I\quad y\in D(x).
  \end{gather}
\end{intermediate}

We prove~$(\dag)$. If $\bA$ is not of spherical type, any singleton $I=\{x\}$ is suitable if $x$ is maximal in the finite poset $(\SS,\leql)$. Indeed, \eqref{eq:25}~is trivially true, and for~\eqref{eq:27}, by maximality of~$x$, any $y\in\SS\setminus\{\unit\}$ satisfies $x=\bigveel\{\zeta\in\SS\tq\zeta\leql x\cdot y\}$ and thus $y\in D(x)$ by Lemma~\ref{lem:2-FC}.

For $\bA$ of spherical type, we put $I=\bigl\{\Delta_A\tq\exists a\in\Sigma\quad A=\Sigma\setminus\{a\}\bigr\}$, where $\Delta_A=\bigveel A$. Let $A=\Sigma\setminus\{a\}$ and $B=\Sigma\setminus\{b\}$ with $a\neq b$. Then there exists $c\notin\{ a,b\}$ since $|\Sigma|>2$, and then $L(c)=\{c\}\subseteq R(\Delta_A), R(\Delta_B)$, hence $c\in D(\Delta_A)\cap D(\Delta_B)$ using Corollary~\ref{cor:4-FC}. This proves~\eqref{eq:25}. To prove~\eqref{eq:27}, consider any $x\in\SS\setminus\{\unit,\Delta\}$. Then $L(x)\neq\Sigma$ hence $L(x)\subseteq R(\Delta_A)$ for any $A=\Sigma\setminus\{a\}$ such that $a\notin L(x)$, and thus $x\in D(\Delta_A)$.

Hence we have proved the claim~$(\dag)$, and thus that $g$ is globally constant. It entails that $\Htilde$ is itself constant, equal to zero, and that $F$ is proportional on $\SS\setminus\{\unit, \Delta\}$ to the length function, contradicting our assumption. Consequently, Theorem~\ref{thr:6} applies and we derive the stated convergence.
\end{proof}

\begin{remark}
Let $\NN(0,0)$ denote the Dirac measure on $\bbR$ concentrated on~$0$. We remark that, in the cases excluded in the last paragraph of Theorem~\ref{thr:8}, the same convergence holds with $s^2=0$, hence toward~$\NN(0,0)$.

For functions $F$ proportional to the length function, the quantity $F(x)-k\gamma$ identically vanishes, hence the convergence is trivial. The statement also excludes the case of an Artin-Tits monoid $\bA$ of spherical type and with two generators only. Since $\bA$ is assumed to be irreducible, it is the case of $\bA=\bA\bigl(\{a,b\},\ell\bigr)$ with $3\leq\ell(a,b)<\infty$. Such a monoid is easily investigated. Putting $n=\ell(a,b)$, one has  $\SS=\{\unit,x_1,\dots,x_{n-1},y_1,\dots,y_{n-1},\Delta\}$ with $x_i=abab\cdots$, $y_i=baba\cdots$ and $\abs{x_i}=\abs{y_i}=i$, and $\Delta=x_n=y_n$\,. Furthermore, putting $X=\{x_j\tq1\leq j<n\}$ and $Y=\{y_j\tq1\leq j<n\}$, one has:
\begin{align*}
\text{for $2i+1<n$:}&& D(x_{2i+1})&=X,&D(y_{2i+1})&=Y,\\
\text{for $2i<n$:}&&D(x_{2i})&=Y,&D(y_{2i})&=X.
\end{align*}
Consider a M\"obius valuation~$\omega$, and let $G$ be the normal-additive function defined by $G(x_{2i})=1$, $G(y_{2i})=-1$, $G(x_{2i+1})=G(y_{2i+1})=0$. Observe that $\omega(x_{2i})=\omega(y_{2i})$, hence $\gamma=0$ for symmetry reasons. For any $x\in\bA$, with normal form $x=\Delta^{i_x}\cdot s_1\cdot\ldots\cdot s_m$\,, one has $G(x)=i_x G(\Delta)+\varepsilon_m$ with $\varepsilon_m\in\{-1,0,1\}$, where $i_x$ is the number of $\Delta$s in the normal form of~$x$. Hence, for any $a>0$:
\begin{align*}
  \omega_k\Bigl(\Bigl|\frac{G(x)}{\sqrt k}\Bigr|\geq a\Bigr)\leq\omega_k\Bigl(i_x\geq\frac{a\sqrt k-1}{1+|G(\Delta)|}\Bigr).
\end{align*}
Since $i_x$ converges in law toward a geometric law, it follows that the right-hand member above, and thus the left-hand member, converges toward~$0$, and so $G(x)/\sqrt k$ converges in law toward~$\NN(0,0)$, as expected.

An inspection of the proof of Theorem~\ref{thr:8} shows that the functions for which the non degeneracy criterion for the convergence applies, are exactly those in the two-dimensional vector space generated by the length function and~$G$. Hence, for normal-additive functions outside this vector space, the convergence stated in Theorem~\ref{thr:8} applies with $s^2>0$.
\end{remark}

Let us apply Theorem~\ref{thr:8} to obtain information on the following statistics: the ratio height over length of large elements in an irreducible Artin-Tits monoid.

\begin{corollary}
  \label{cor:2}
  Let\/ $\bA$ be an Artin-Tits monoid. We assume that $\bA$ is irreducible, has at least two generators and is not a free monoid. Let $\kappa$ be the speedup associated to the constant valuation $\omega(x)=1$ on~$\bA$.

Then, for all reals $a<b$ distinct from~$\kappa^{-1}$, one has:
\begin{align*}
  \lim_{k\to\infty}\frac{\#\{x\in \bA_k\tq
  a<\height(x)/k<b\}}{\#\bA_k}&=\un(a<\kappa^{-1}<b),\\
\intertext{and for $a$ and $b$ distinct from~$\kappa$:}
\lim_{k\to\infty}\frac{\#\{x\in\bA_k\tq
a<k/\height(x)<b\}}{\#\bA_k}&=\un(a<\kappa<b).
\end{align*}

It entails the convergence of the following expectations, where\/ $\esp_k$
denotes the expectation with respect to the uniform distribution
on\/~$\bA_k$:
\begin{align*}
  \lim_{k\to\infty}\frac1k\esp_k\height(\cdot)&=\kappa^{-1},&
\lim_{k\to\infty}k\esp_k\Bigl(\frac1{\height(\cdot)}\Bigr)&=\kappa.
\end{align*}

Furthermore, there is a positive constant $s^2$ such that, for any two reals
$a<b$, one has:
\begin{align}
\label{eq:14}
&\lim_{k\to\infty}
  \frac{\#\bigl\{x\in\bA_k\tq a\sqrt k<\height(x)-k\kappa^{-1}<b\sqrt
  k\bigr\}}{\#\bA_k}=\frac1{\sqrt{2\pi
                      s^2}}\int_a^b\exp\Bigl(-\frac{t^2}{2s^2}\Bigr)\,dt,\\
\label{eq:15}&\lim_{k\to\infty}\frac
{\#\bigl\{
x\in\bA_k\tq
a<\sqrt k\bigl(k/\height(x)-\kappa)<b
\bigr\}}
{\#\bA_k}
=\frac1{\sqrt{2\pi s^2\kappa^{-4}}}\int_a^b\exp\Bigl(
-\frac{t^2}{2s^2\kappa^{-4}}
\Bigr)\,dt.
\end{align}
\end{corollary}

\begin{proof}
  We apply Theorem~\ref{thr:8} with the height function
  $F(x)=\height(x)$, which is normal-additive. It is not
  proportional on $\SS\setminus\{\unit,\Delta\}$ to the length
  function since $\bA$ is not a free monoid. If $\bA$ is of spherical type with only two generators, the function $F$ does not belong to the vector space generated by the length function and the function $G$ described in the above Remark. Hence, Theorem~\ref{thr:8} applies.

The ratios $\height(x)/\abs{x}$ converge in distribution toward $\delta_\gamma$ with:
  \begin{gather*}
    \gamma=\frac1\kappa\sum_{x\in\SS\setminus\{\unit\}}\theta(x)=\frac1\kappa\,.
  \end{gather*}

  Since it is a constant, it entails the convergence in distribution
  of the inverse ratios $\abs{x}/\height(x)$ toward the constant
  $\gamma^{-1}=\kappa$. The two first points derive at once, as well
  as the convergence of the expectations.

The convergence~\eqref{eq:14} is the reformulation of the convergence in
distribution $\sqrt
k\bigl(\height(x)/\abs{x}-\kappa^{-1}\bigr)\cvlaw{}\NN(0,s^2)$. It is then
well know how to derive the following convergence:
\begin{gather*}
  \sqrt
  k\Bigl(\frac{\abs{x}}{\height(x)}-\kappa\Bigr)\cvlaw{k\to\infty}\NN(0,s^2\gamma^4),
\end{gather*}
based on the Delta method, see~\cite[p.~359]{billingsley95}. The
convergence~\eqref{eq:15} follows.
\end{proof}

\section{Similar results for other monoids}
\label{sec:simil-results-other}

So far, our results only focused on Artin-Tits monoids and their Garside
normal forms, with the help of conditioned weighted graphs. However, the
arguments developed above may apply to other monoids as well, provided that
their combinatorial structure is similar to that of Artin-Tits monoids.

In Section~\ref{sec:more-general-working}, we list the different properties
on which our arguments rely, and then we state the corresponding results. In
Section~\ref{sec:bound-mult-lmeas}, we underline that the main objects that
we construct---the boundary at infinity and the class of multiplicative
probability measures on the boundary---are intrinsically attached to the
monoid, although the construction uses non intrinsic objects such as a
particular Garside subset. Finally,  we identify in
Section~\ref{sec:exampl-outs-artin} several examples of monoids fitting into
this more general framework, and not falling into the class of Artin-Tits
monoids.

\subsection{A more general working framework}
\label{sec:more-general-working}

The properties that must be satisfied by the monoid $\bA$ to follow the
sequence of arguments that we developed are the following:
\begin{enumerate}[(P1)]
 \item\label{item:15} There exists a \emph{length} function on the
     monoid~$\bA$, \ie, a function $\abs{\cdot} : \bA \to \bbZ_{\geq0}$
     such that:
\begin{align*}
\forall x,y \in \bA\quad \abs{x \cdot y} = \abs{x} + \abs{y} &&\text{and}&&\forall x\in\bA\quad x= \unit \iff \abs{x}=0.
\end{align*}
\item\label{item:16} The monoid $\bA$ is both left and right cancellative,
    meaning:
\begin{gather*}
\forall x,y,z \in \bA\quad (z \cdot x = z \cdot y \implies x = y) \text{ and } (x \cdot z = y \cdot z \implies x = y).
\end{gather*}
\item\label{item:17} The ordered set $(\bA,\leql)$ is a lower semi-lattice,
    \ie, any non-empty set has a greatest lower bound in~$\bA$.
 \item\label{item:18} There exists a finite Garside subset~$\SS$, \ie, a
     finite subset of $\bA$ which generates $\bA$ and is closed under
     existing $\veel$ and downward closed under~$\leqr$ (the condition that
     $\SS$ contains $\Sigma$ is dropped: it does not make sense in our
     context as we have not singled out a generating set $\Sigma$ in our
     assumptions).
 \item\label{item:19} The Charney graph $(\Ch,\to)$ is strongly connected,
     where $\Ch$ is the subset of $\SS$ and $\to$ is the relation on $\SS
     \times \SS$ defined by:
 \begin{align*}
   \Ch &= \begin{cases}
    \SS \setminus \{\Delta,\unit\}, & \text{if $(\SS,\leql)$ has a maximum $\Delta$} \\
    \SS \setminus \{\unit\}, & \text{otherwise}
  \end{cases} \\
   x \to y&\iff x = \bigveel\bigl\{z \in \SS \tq z \leql x \cdot y\bigr\}.
  \end{align*}
 \item\label{item:20} The integers in the set:
   \begin{gather*}
     \bigl\{\abs{z_1}+\ldots+\abs{z_k} \tq k\geq1,\quad z_1,\dots,z_k\in\SS,\quad
 z_1 \to z_2 \to \ldots \to z_k \to z_1\bigr\},
   \end{gather*}
 are setwise coprime, \ie, have $1$ as greatest common divisor.
 \item\label{item:21} If $(\SS,\leql)$ has a maximum, then the Charney
     graph $(\Ch,\to)$ has at least one vertex with out-degree two or more.
\end{enumerate}

We observe that all irreducible Artin-Tits monoids with at least two
generators satisfy the above axioms.

Property~\qref{item:15} states the conditions that must be met by the notion
of length of an element, and then Properties~\qref{item:16}
to~\qref{item:18} lead to the notion of Garside normal form and its variants,
such as the generalized Garside normal form.

Property~\qref{item:19} is then equivalent with Theorem~\ref{thr:1qqlknaa}.
Property~\qref{item:20} is used to obtain aperiodic, and therefore primitive
matrices. Property \qref{item:21} ensures that the $\Delta$ component, if it
exists, will have a small spectral radius. In the case of irreducible
Artin-Tits monoids, \qref{item:21}~holds for monoids with at least two
generators, and only for them.

The results obtained in the previous sections for irreducible Artin-Tits
monoids with at least two generators, and regarding the construction of
multiplicative probability measures at infinity, generalize to all monoids
satisfying Properties~\qref{item:15}
 to~\qref{item:21}.

\begin{theorem}\label{thr:axiomatization}
Let\/ $\bA$ be a monoid satisfying Properties~\qref{item:15}--\qref{item:21}.
Then Theorem~{\normalfont\ref{thr:2}}, Proposition~{\normalfont\ref{prop:7}},
Theorem~{\normalfont\ref{thr:3}}, Corollary~{\normalfont\ref{cor:4}},
Theorem~{\normalfont\ref{thr:7}} and Theorem~{\normalfont\ref{thr:8}} hold
for the monoid\/~$\bA$.
\end{theorem}

In view of Theorem~\ref{thr:axiomatization}, two questions naturally arise:
\begin{itemize}
\item Assume given a monoid $\bA$
    satisfying~\qref{item:15}--\qref{item:21}. The central objects that we
    consider are the boundary at infinity of~$\bA$, the multiplicative
    measures and the uniform measure on the boundary. How much of these
    objects are intrinsic to~$\bA$? And on the contrary, how much depend on
    the specific length function and on the Garside subset that were
    chosen?
\item What are typical examples of monoids
    satisfying~\qref{item:15}--\qref{item:21} outside the family of
    Artin-Tits monoids?
\end{itemize}

We answer the first question below in Section~\ref{sec:bound-mult-lmeas},
showing that most of our objects of interest are intrinsic to the monoid. The
answer to the second question is the topic of
Section~\ref{sec:exampl-outs-artin}.

\subsection{The boundary at infinity and the multiplicative measures are intrinsic}
\label{sec:bound-mult-lmeas}

We have constructed in Section~\ref{sec:boundary-elements-an} a
compactification $\bAbar$ of an Artin-Tits monoid $\bA$ as the set of
infinite paths in the graph $(\SS,\to)$, where $\SS$ is the smallest finite
Garside subset of~$\bA$. Then $\bA$ identifies with the set of paths that
hit~$\unit$, whereas the \emph{boundary} is the set of paths that never
hit~$\unit$.

The same construction is carried over for an arbitrary monoid $\bA$ equipped
with a finite Garside subset~$\SS$, and this was implicitly understood in the
statement of Theorem~\ref{thr:axiomatization}.  We have thus an operational
description of the compactification and of the boundary at infinity relative
to the Garside subset~$\SS$, say $\bAbar_\SS$ and~$\BA_\SS$\,. It is not
obvious however to see that the two compact spaces thus obtained are
essentially independent of~$\SS$. We sum up an alternative construction of
the boundary at infinity, already used in~\cite{abbes15a}, and the fact that
it is equivalent to the previous construction in the following result, the
proof of which we omit.

\begin{proposition}
  \label{prop:3}
Let\/ $\bA$ be a monoid, and define the preorder $\leql$ on $\bA$ by putting
$x\leql y\iff(\exists z\in\bA\quad y=x\cdot z)$. Let
$\H=\bigl\{(x_i)_{i\geq0}\tq \forall i\geq0\quad x_i\in\bA\quad x_i\leql
x_{i+1}\bigr\}$. Equip $\H$ with the preordering relation defined, for
$x=(x_i)_{i\geq0}$ and $y=(y_i)_{i\geq0}$\,, by:
\begin{gather*}
  x\sqsubseteq y\iff\bigl(
\forall i\geq0\quad\exists j\geq0\quad x_i\leql y_j\bigr).
\end{gather*}
Let finally $(\bAbar,\leql)$ be the \emph{collapse} of~$(\H,\sqsubseteq)$.
That is to say, $\bAbar$~is the partially ordered set obtained as the
quotient of $\H$ by the equivalence
relation~$\sqsubseteq\cap(\sqsubseteq)^{-1}$.

Assume that $\bA$ satisfies Properties~\qref{item:15}, \qref{item:16}
and~\qref{item:18}.
\begin{enumerate}
\newcounter{mycount}%
\item\setcounter{mycount}{\theenumi} Then $(\bA,\leql)$ is a  partial order
    that identifies with its image in $\bAbar$ through the composed mapping
    $\bA\to\H\to\bAbar$, where the mapping $\bA\to\H$ sends an element $x$
    to the constant sequence $(x,x,\dots)$.
\end{enumerate}
One equips $\bAbar$ with the smallest topology containing as open sets all
sets of the form $\up x$ with $x\in\bA$ and all sets of the form
$\bAbar\;\setminus\up x$ with $x\in\bAbar$ (in Domain theory, this
corresponds to the \emph{Lawson
topology}~{\normalfont\cite[p.211]{gierz03}}). Finally, the \emph{boundary at
infinity} of\/ $\bA$ is  the topological space defined by
$\BA=\bAbar\setminus\bA$.
\begin{enumerate}
\setcounter{enumi}{\themycount}
\item There is a canonical mapping between $\bAbar_\SS$ and~$\bAbar$. This
    mapping makes  $(\bAbar_\SS,\leql)$ and $(\bAbar,\leql)$  isomorphic as
    partial orders, and the topological spaces homeomorphic. Its
    restriction to $\BA_\SS$ induces a homeomorphism from $\BA_\SS$
    to~$\BA$.
\end{enumerate}
\end{proposition}

Multiplicative measures on the boundary are those measures $m$ on $\BA$
satisfying $m\bigl(\up(x\cdot y)\bigr)=m(\up x)\cdot m(\up y)$. We deduce the
following corollary.

\begin{corollary}
\label{cor:1}
  Let\/ $\bA$ be a monoid satisfying Properties~\qref{item:15} to~\qref{item:21}. Then the notions of boundary at infinity and of multiplicative measure are intrinsic to\/~$\bA$, and do not depend either on the specific length function nor on the specific finite Garside subset considered.
\end{corollary}

By contrast, the notion of uniform measure does depend on the specific length
function one considers. Nevertheless, the uniform measure associated to any
length function belongs to the class of multiplicative measures---and this
class is intrinsic to~$\bA$ according to Corollary~\ref{cor:1}.

\subsection{Examples outside the family of Artin-Tits monoids}
\label{sec:exampl-outs-artin}

We now mention two families of monoids matching the general working framework
introduced in Section~\ref{sec:more-general-working}, and yet outside the
family of Artin-Tits monoids.

\subsubsection{Dual braid monoids}

Braid monoids, which are among the foremost important Artin-Tits monoids, can
be seen as sub-monoids of \emph{braid groups}. The braid group with $n$
strands, for $n \geq 2$, is the group $B_n$ defined by the following
presentation:
\begin{gather*}
B_n = \bigl\langle \sigma_1,\ldots,\sigma_{n-1} \mid
\sigma_i \sigma_j = \sigma_j \sigma_i \text{ if } \abs{i-j} > 1,\
\sigma_i \sigma_j \sigma_i = \sigma_i \sigma_j \sigma_i \text{ if } i = j \pm 1\bigr\rangle\,.
\end{gather*}

The associated braid monoid is just the sub-monoid positively generated by
the family $\{\sigma_i \tq 1 \leq i \leq n\}$.

The \emph{dual braid monoid} is another sub-monoid of the braid group,
strictly greater than the braid monoid when $n \geq 3$. This monoid
was introduced in~\cite{birman1998new}. It is the sub-monoid
of $B_n$ generated by the family $\{\sigma_{i,j} \tq 1\leq i<j\leq n\}$\,,
where $\sigma_{i,j}$ is defined by:
\begin{align*}
\sigma_{i,j}&=\sigma_i\,,&&\text{for $1\leq i<n$ and $j=i+1$,}\\
\sigma_{i,j}&=\sigma_i \sigma_{i+1} \ldots \sigma_{j-1}
\sigma_{j-2}^{-1} \sigma_{j-3}^{-1} \ldots
\sigma_i^{-1}\,,&&\text{for $1\leq i< n-1$ and $i+2\leq j\leq n$}\,.
\end{align*}

Alternatively, the dual braid monoid is the monoid generated by the elements
$\sigma_{i,j}$ for $1 \leq i < j \leq n$, and subject to the relations:
\[\begin{cases}\sigma_{i,j} \cdot \sigma_{j,k} = \sigma_{j,k} \cdot \sigma_{i,k} = \sigma_{i,k} \cdot \sigma_{i,j} & \text{for } 1 \leq i < j < k \leq n; \\
\sigma_{i,j} \cdot \sigma_{k,\ell} = \sigma_{k,\ell} \cdot \sigma_{i,j} & \text{for } 1 \leq i < j < k < \ell \leq n; \\
\sigma_{i,j} \cdot \sigma_{k,\ell} = \sigma_{k,\ell} \cdot \sigma_{i,j} & \text{for } 1 \leq i < k < \ell < j \leq n.
\end{cases}\]

It is proved in~\cite{abbes17} that dual braid monoids satisfy Properties~\qref{item:15} to~\qref{item:21}, where the length of an element $x$ of the monoid is the number of generators $\sigma_{i,j}$ that appear in any word representing~$x$, and where $\SS$ is the smallest Garside subset of the monoid.

In fact, this construction can be generalised to all irreducible Artin-Tits monoids
of spherical type. Indeed, for any such monoid~$\bA$, Bessis
also developed in~\cite{bessis2003dual} a notion of dual monoid, which subsumes the notion of dual braid monoid
in case $\bA$ is the braid monoid~$B_n^+$. % In fact, once $\bA$ is fixed,
% such dual monoids are defined up to conjugation, hence they have the same
% combinatorial properties.
 In particular, Bessis proved that these dual monoids
satisfy Properties~\qref{item:15} to~\qref{item:18},~\qref{item:20} and~\qref{item:21},
with the same definitions of element length and of Garside set~$\SS$.

However, Property~\qref{item:19} was not investigated. Hence, we prove here
the following result.

\begin{proposition}\label{pro:dual-braid-prop-5}
Let\/ $\bA$ be an irreducible Artin-Tits monoid of spherical type,
with at least two generators.
The dual monoid associated with\/ $\bA$ satisfies Properties~\qref{item:15}
to~\qref{item:21}.
\end{proposition}

\begin{proof}
Based on the above discussion, it only remains to prove that the dual monoid
$\bB$ satisfies Property~\qref{item:19}.
It follows from the classification of finite Coxeter groups~\cite{coxeter1935complete}
that $\bA$ must fall into one of the following finitely many families.
Three families consist of so-called monoids of type~$A$, $B$ and~$D$,
one is the family  of dihedral monoids~$I_2(m)$ (with $m \geqslant 3$),
and other families are finite
(these are families of \emph{exceptional type}).

Monoids of type $A$ are braid monoids, which are already treated
in~\cite{abbes17}. Furthermore, using computer algebra (for instance the package CHEVIE of GAP~\cite{gap3}), it is easy to prove~\qref{item:19} in the case of families
of exceptional type. Hence, we focus on the three infinite families of monoids.

If $\bA$ is of type~$I_2(m)$, the monoid $\bB$ has the following presentation (this presentation is
the one given by~\cite{picantin2002explicit}, with reversed product order):
\begin{gather*}
\bB = \bigl\langle \sigma_1,\ldots,\sigma_m \mid
\sigma_i \sigma_j = \sigma_m \sigma_1 \text{ if } j = i+1 \bigr\rangle^+\,.
\end{gather*}
Then, its Garside set is $\SS = \{\unit,\sigma_1,\ldots,\sigma_m,\Delta\}$,
where $\Delta = \sigma_m \sigma_1$ is the only element of $\SS$ with length~$2$.
In particular, we have $\Ch = \{\sigma_1,\ldots,\sigma_m\}$, and 
$\sigma_i \to \sigma_j$ for all $j \neq i+1$, which proves~\qref{item:19} in this case.

In case $\bA$ is of type $B$ or~$D$, combinatorial descriptions
of $\bB$ are provided in~\cite{reiner1997non,brady2002k,picantin2002explicit,
bessis2003dual,athanasiadis2004noncrossing,armstrong2009generalized}.
They include descriptions
of the generating family~$\Sigma$, of the Garside set~$\SS$, and of the multiplication relations in~$\SS$.
These descriptions lead, in all dual braid monoids, 
to the fact that $(\SS,\leql)$ is a lattice, and to the following property:
\begin{intermediate}
    {\normalfont$(\dag)$}%
   For all $\sigma, \tau \in \Sigma$ such that $\sigma\cdot \tau \in \SS$,
   it holds that $\tau \leql \sigma \cdot \tau$ and 
   $\sigma \leqr \sigma \cdot \tau$.
  \end{intermediate}

Property~$(\dag)$ leads
to Lemma~\ref{lem:left-right-dual} and
Corollary~\ref{cor:left-right-dual-2}.
From there, and by following a proof that is
very similar to the one in~\cite{abbes17} for monoids of type~$A$,
we show Lemmas~\ref{lem:dual-braid-case-b}
and~\ref{lem:dual-braid-case-d}, thereby demonstrating
Proposition~\ref{pro:dual-braid-prop-5}.
\end{proof}

\begin{lemma}\label{lem:left-right-dual}
The left and right divisibility relations coincide on\/~$\bB$, \ie:
\begin{gather}\label{eq:left-right-dual}
\forall x,y \in \bB \quad x \leql y \iff x \leqr y.
\end{gather}
\end{lemma}

\begin{proof}
We proceed by induction on~$\abs{y}$.
The statement is immediate if $\abs{x} = 0$ or $\abs{x} = \abs{y}$,
hence we assume that $0 < \abs{x} < \abs{y}$.

If $x \leql y$, let us factor $x$ and $y$ as products of the form
$x = x_1 \cdot x_2$ and $y = x \cdot y_3 \cdot y_4$,
where $x_2$ and $y_4$ belong to~$\Sigma$.
Using the induction hypothesis and Property~$(\dag)$,
there exist elements $y'_3$, $y'_4$ and $y''_4$ of $\bB$
such that
\begin{gather*}
y = x_1 \cdot x_2 \cdot y_3 \cdot y_4 = y'_3 \cdot x_1 \cdot x_2 \cdot y_4
= y'_3 \cdot x_1 \cdot y'_4 \cdot x_2 = y'_3 \cdot y''_4 \cdot x_1 \cdot x_2,
\end{gather*}
which proves that $x \leqr y$.
We prove similarly that, if $x \leqr y$, then $x \leql y$.
\end{proof}

\begin{corollary}\label{cor:left-right-dual-2}
For all $x \in \SS$, let $L(x) = \{\sigma \in \Sigma \tq \sigma \leql x\}$.
Then, for all $x$ and $y$ in~$\SS$:
\begin{gather*}
\bigl(\,\forall \tau \in L(y) \quad \exists
\sigma \in L(x) \quad \sigma \to \tau\bigr)\implies (x\to y).
\end{gather*}
\end{corollary}

\begin{proof}
Assume that the relation $x \to y$ does not hold.
This means that we can factor $y$ as a product
$y = \tau \cdot y_2 \cdot y_3$, with $\tau \in \Sigma$,
such that the element $z = x \cdot \tau \cdot y_2$ belongs to~$\SS$.
Then, let $\sigma \in L(x)$ be such that $\sigma \to \tau$.

By Lemma~\ref{lem:left-right-dual}, we can also factor $x$ as a product
$x = x_1 \cdot \sigma$. Hence, the element
$\sigma \cdot \tau$ is a factor of~$z$, and therefore
$\sigma \cdot \tau \leqr z$.
Since $\SS$ is downward closed under~$\leqr$,
this means that $\sigma \cdot \tau \in \SS$,
contradicting the fact that $\sigma \to \tau$.
\end{proof}

\begin{remark}
\label{rem:1}
The above corollary indicates how to infer a relation $x\to y$ for $x,y\in\SS$ based on the knowledge of the restriction~$\to\rest{\Sigma\times\Sigma}$\,. In turn, the explicit presentations given in \cite{picantin2002explicit} for dual monoids of monoids of type $B$ and $D$ have the following property: for any two generators $\sigma,\tau\in\Sigma$, the relation $\sigma\to\tau$ holds if and only if the product $\sigma\cdot\tau$ does not appear as a member of the presentation rules of the dual monoid.  In particular,  $\sigma \to \sigma$ holds for all $\sigma \in \Sigma$. 
\end{remark}

\begin{lemma}\label{lem:dual-braid-case-b}
Let\/ $\bA$ be an irreducible Artin-Tits monoid of type~$B$, with two generators or more. The dual monoid associated with $\bA$ satisfies Property~\qref{item:19}.
\end{lemma}

\begin{proof}
Thanks to the combinatorial descriptions mentioned above, we identify elements of $\SS$
with \emph{type B non-crossing partitions} of size~$n$,
where $n$ is the number of generators of~$\bA$.
These are the partitions
$\mathbf{T} = \{T^1,\ldots,T^m\}$ of $\bbZ/(2n)\bbZ$ such that,
for every set~$T^i$, the set $n + T^i$ is also in~$\mathbf{T}$, and
the sets $\{\exp(\mathbf{i}\;\! k \pi / n) \tq
k \in T^i\}$ have pairwise disjoint convex hulls in the complex plane. Note that $\mathbf{T}$ is thus globally invariant by central symmetry. Both relations $\leql$ and $\leqr$ then coincide with the
partition refinement relation.

Thus, we respectively identify the extremal elements $\unit$ and $\Delta$ of $\SS$ 
with the partitions
$\bigl\{\{1\},\{2\},\ldots,\{2n\}\bigr\}$ and $\bigl\{\{1,2,\ldots,2n\}\bigr\}$,
and $\Sigma$ consists of those partitions
\begin{gather*}
\sigma_{i,j} = \bigl\{\{i,j\},\{i + n,j + n\}\bigr\} \cup
\bigl\{\{k\} \tqs k \neq i,j,i+n,j+n\bigr\}\,,
\end{gather*}
with $1 \leq i \leq n$ and $i < j \leq n+i$. They are pictured in Figure~\ref{fig:typeB} for $n=3$.
The left divisors of a partition~$\mathbf{T}$,
\ie, the elements of~$L(\mathbf{T})$, are then those partitions
$\sigma_{i,j}$ that refine~$\mathbf{T}$.

\begin{figure}
  \centering
  \begin{tabular}{ccccc}
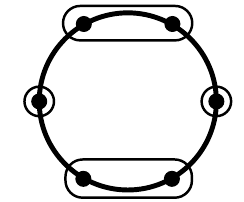&
\includegraphics[scale=0.74]{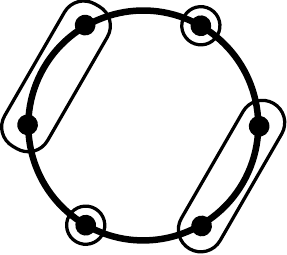}&
\includegraphics[scale=0.74]{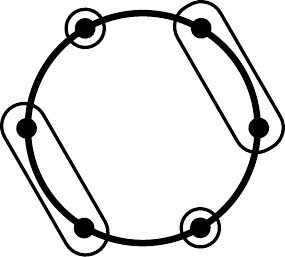}&
\includegraphics[scale=0.74]{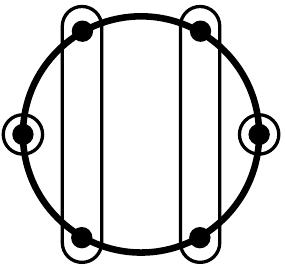}&
\includegraphics[scale=0.74]{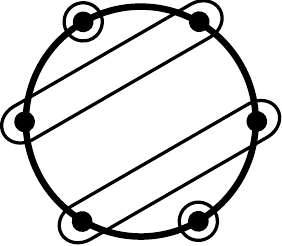}\\
$\sigma_{1,2}$&$\sigma_{2,3}$&$\sigma_{3,4}$&$\sigma_{2,4}$&$\sigma_{1,3}$\\[1em]
\includegraphics[scale=0.74]{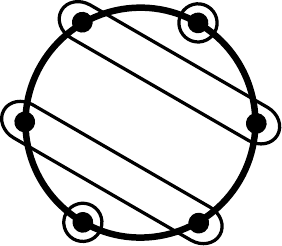}&
\includegraphics[scale=0.74]{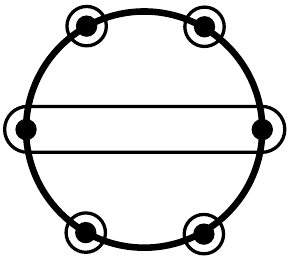}&
\includegraphics[scale=0.74]{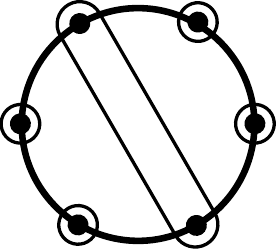}&
\includegraphics[scale=0.74]{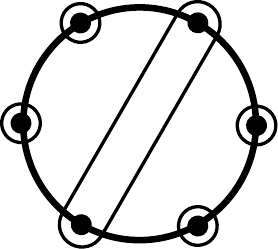}
\\
$\sigma_{3,5}$
&$\sigma_{3,6}$&$\sigma_{2,5}$&$\sigma_{1,4}$\\
\end{tabular}

  \caption{Generators of the dual monoid of type $B$ for $n=3$}
  \label{fig:typeB}
\end{figure}

Finally, keep in mind the following observation. Let $v$ be
a maximal proper factor of $\Delta$, \ie, 
a type $B$ non-crossing partition of which $\Delta$ is the immediate successor in the partition refinement order. After application of a bijection of the form $i\in\bbZ/(2n)\bbZ\mapsto i+j$, for some integer $j\in\{0,2n-1\}$, $v$~is of the form \begin{gather}
\label{eq:29}
v_i=\bigl\{\{1,\ldots,i\},\{n+1,\ldots,n+i\},\{i+1,\ldots,n,n+i+1,\ldots,2n\}\bigr\}\,,
\end{gather}
for some integer $i\in\{1,\dots,n\}$.

Now, consider two elements $y, z$ of $\SS \setminus \{\unit,\Delta\}$, and let
us prove that $y \to^\ast z$. Since
$z \lel \Delta$, and since the map $\sigma_{i,j} \mapsto \sigma_{i+1,j+1}$ induces
an automorphism of the monoid~$\bB$, corresponding to the rotation of the complex plane of angle~$\pi/n$, we assume without loss of generality
that $z$ is a refinement of some partition~$v_i$ from~\eqref{eq:29}, for an integer $i\in\{1,\ldots,n\}$. Corollary~\ref{cor:left-right-dual-2} implies in particular that $v_i\to z$ holds.

Based on Remark~\ref{rem:1}, one sees that:
\begin{gather*}
 \sigma_{i,j} \to \sigma_{u,v}\iff\bigl( (u < i \leq v < j)\text{ or }(i \leq u < j \leq v)\text{ or }
(i \leq u+n < j \leq v+n)\bigr)\,.
\end{gather*}
Using Corollary~\ref{cor:left-right-dual-2}, it implies: $v_j\to v_{j+1}$ when $1 \leq j \leq n-1$, and $v_n\to w\to v_1$ with $w=\bigl\{
\{1,n+2,\ldots,2n\},\{2,\ldots,n+1\}
\bigr\}$. It follows that holds: $v_n \to^\ast v_i \to z$.

Finally, let $\sigma_{a,b}$ be some generator in~$L(y)$. 
It comes that $y \to \sigma_{1,a} \to \sigma_{1,n+1}$ if $1 < a$, or
$y \to \sigma_{1,n+1}$ if $a = 1$. We then observe that
$\sigma_{1,n+1} \to x \to v_n$, where
$x = \bigl\{\{j,n+1-j\} \tqs 1 \leq j \leq n\bigr\}$,
and therefore that $y \to^\ast z$.
\end{proof}

\begin{lemma}\label{lem:dual-braid-case-d}
Let\/ $\bA$ be an irreducible Artin-Tits monoid of type~$D$, with two generators or more. The dual monoid associated with\/ $\bA$ satisfies Property~\qref{item:19}.
\end{lemma}

\begin{proof}
Let $n+1$ be the number of generators of~$\bA$.
If $n \leq 2$, $\bA$~is also a braid monoid.
Hence, we assume that $n \geq 3$.

We identify elements of $\SS$
with \emph{type $D$ non-crossing partitions} of size~$n+1$.
Here, we associate every element $k$ of $\mathbb{Z}/(2n)\mathbb{Z}$ with the
complex point $p_k = \exp(\mathbf{i}\;\! k \pi / n)$.
We also consider a two-element set $\{\bullet,\bullet+n\}$,
with the convention that $\bullet + 2n = \bullet$,
and associate both $\bullet$ and $\bullet+n$
with the complex point $p_\bullet = p_{\bullet+n} = 0$.

Then, type $D$ non-crossing partitions of size $n$
are the partitions $\mathbf{T} = \{T^1,\ldots,T^m\}$ of
$\mathbb{Z}/(2n)\mathbb{Z} \cup \{\bullet,\bullet+n\}$
such that (i)~$\mathbf{T}$ does \emph{not} contain the set $\{\bullet,\bullet+n\}$,
(ii)~for every set~$T^i$, the set $n + T^i$ is also in~$\mathbf{T}$,
and (iii)~the sets $\{p_k \tq k \in T^i\}$ have pairwise disjoint convex
hulls in the complex plane, with the exception that they may share the point
$p_\bullet = p_{\bullet+n} = 0$.

Again, both relations $\leql$ and $\leqr$ coincide with the
partition refinement relation,
extremal elements $\unit$ and $\Delta$ of $\SS$
are identified with the respective partitions
\begin{gather*}
\bigl\{\{1\},\{2\},\ldots,\{2n\},\{\bullet\},\{\bullet+n\}\bigr\} \text{\quad and\quad}
\bigl\{\{1,2,\ldots,2n,\bullet,\bullet+n\}\bigr\},
\end{gather*}
and $\Sigma$ consists of those partitions
\begin{align*}
\sigma_{i,j} & = \bigl\{\{k\} \tqs k \neq i,j,i+n,j+n\bigr\} \cup
\bigl\{\{\bullet\},\{\bullet+n\},\{i,j\},\{i + n, j + n\}\bigr\} \\
\tau_\ell & = \bigl\{\{k\} \tqs k \neq \ell,\ell+n\bigr\} \cup
\bigl\{\{\ell,\bullet\},\{\ell + n,\bullet+n\}\bigr\}\,,
\end{align*}
with $1 \leq i \leq n$, $i < j < n+i$, and $1 \leq \ell \leq 2n$.

Moreover, proceeding as in Remark~\ref{rem:1}, one finds that:
\begin{gather*}
\begin{cases}
\sigma_{i,j} \to \tau_w &
\text{if } w \in \{i,\ldots,j-1\} \text{ or } w+n \in \{i,\ldots,j-1\}; \\
\sigma_{i,j} \to \sigma_{u,v} &
\text{if } u < i \leq v < j \text{, }
i \leq u < j \leq v \text{ or } i \leq u+n < j \leq v+n; \\
\tau_\ell \to \tau_w &
\text{if } w \in \{\ell,\ldots,\ell+n-1\}; \\
\tau_\ell \to \sigma_{u,v} &
\text{if } \{\ell,\ell+n\} \cap \{u+1,\ldots,v\} \neq \emptyset. 
\end{cases}
\end{gather*}

Now, consider two elements $y, z$ of $\SS \setminus \{\unit,\Delta\}$, and let
us prove that $y \to^\ast z$. Since
$z \lel \Delta$, and since the map $\sigma_{i,j} \mapsto \sigma_{i+1,j+1},\quad\tau_\ell\mapsto\tau_{\ell+1}$ induces
an automorphism of the monoid~$\bB$, we assume without loss of generality
that $z$ is refinement of a partition $v$ that is either equal to
$\tau_{2n}^{-1} \Delta =
\{\{1,\ldots,n,\bullet\},\{n+1,\ldots,2n,\bullet+n\}\}$ or to
\begin{align*}
\sigma_{n,n+i}^{-1} \Delta = \{\{1,\ldots,i\},\;\!& \{n+1,\ldots,n+i\}, \\
& \{i+1,\ldots,n,n+i+1,\ldots,2n,\bullet,\bullet+n\}\}
\end{align*}
for some integer $i \in \{1,\ldots,n-1\}$.

Using Corollary~\ref{cor:left-right-dual-2}, one checks that $\sigma_{n,n+j}^{-1} \Delta \to \sigma_{n,n+j+1}^{-1} \Delta$
when $1 \leq j \leq n-1$, and that $\sigma_{n,2n}^{-1} \Delta \to \tau_{2n}^{-1} \Delta
\to \tau_1^{-1} \Delta \to \sigma_{n,n+1}^{-1} \Delta$, where
\begin{gather*}
\tau_1^{-1} \Delta = \bigl\{\{2,\ldots,n+1,\bullet\},\{n+2,\ldots,1,\bullet+n\}\bigr\}.
\end{gather*}
It follows that
$\tau_{2n}^{-1} \Delta \to^\ast v \to z$.
Finally, let $\lambda$ be some generator in $L(y)$. 
It comes that
\begin{gather*}
\begin{cases}
y \to \sigma_{1,a} \to \sigma_{1,n} &
\text{if } \lambda = \sigma_{a,b} \text{ with } 2 \leq a; \\
y \to \sigma_{1,n} &
\text{if } \lambda = \sigma_{a,b} \text{ with } a = 1; \\
y \to \sigma_{1,\ell} \to \sigma_{1,n} &
\text{if } \lambda = \tau_\ell \text{ with } \ell \leq n; \\
y \to \sigma_{1,\ell+n} \to \sigma_{1,n} &
\text{if } \lambda = \tau_\ell \text{ with } n+1 \leq \ell.
\end{cases}
\end{gather*}
We then observe that
$\sigma_{1,n+1} \to w \to \tau_{2n}^{-1}\Delta$, where
\begin{gather*}
w=
\begin{cases}
 \bigl\{\{j,2n+1-j\} \tqs 1 \leq j \leq n \text{ and } j \neq m\bigr\} \cup
\bigl\{\{m,m+n,\bullet,\bullet+n\}\bigr\} \\
\hfill \text{if $n$ is odd and $m = (n+1)/2$;} \\[.5em]
 \bigl\{\{j,2n+1-j\} \tqs 1 \leq j \leq n\bigr\} \cup \bigl\{\{\bullet\},\{\bullet+n\}\bigr\}
\hfill \text{if $n$ is even}.
\end{cases}
\end{gather*}
It follows that $y \to^\ast z$.
\end{proof}

\subsubsection{Free products of Garside monoids}

Garside monoids are extensively studied
structures~\cite{dehornoy1999gaussian,dehornoy2002groupes,picantin2005garside,dehornoy2013foundations}.
Their definition is as follows~\cite{dehornoy2002groupes}. First recall that
an \emph{atom} of a monoid $\M$ is an element $x\in\M$, different from the
unit element~$\unit$, and such that $x=yz\implies(y=\unit\vee z=\unit)$. A
monoid is \emph{atomic} if it is generated by its atoms. \emph{A Garside
monoid is an atomic monoid~$\M$, left and right cancellative, such that
$(\M,\leql)$ and $(\M,\leqr)$ are two lattices, and such that $\M$ contains a
Garside element.} A \emph{Garside element} is an element $\Delta\in\M$ such
that $\{x\in\M\tq x\leql \Delta\}=\{x\in\M\tq x\leqr\Delta\}$, and such that
this set is finite and generates~$\M$.
By~\cite[Prop.~1.12]{dehornoy2002groupes}, this set is then necessarily a
Garside subset of~$\M$.

Garside monoids do not necessarily have a length function as
in~\qref{item:15}. Hence, in the following result, we have to assume its
existence as an additional assumption in order to fit with our previous
setting.

\begin{proposition}\label{pro:free-prod-1-7}
Let\/ $\bA_1$ and\/ $\bA_2$ be non-trivial Garside monoids satisfying
Property~\qref{item:15}. Then the monoid\/ $\bA_0$ defined as the free
product\/ $\bA_0 = \bA_1 * \bA_2$ satisfies Properties~\qref{item:15}
to~\qref{item:21}.
\end{proposition}

\begin{proof}
  This is a consequence of the properties of Garside monoids recalled above, together with the following elementary lemma.
\end{proof}

\begin{lemma}
  \label{lem:5}
Let\/ $\bA_1$ and\/ $\bA_2$ be two non-trivial monoids satisfying
Properties~\qref{item:15}--\qref{item:18}. Then the free product monoid\/
$\bA_0 = \bA_1 * \bA_2$ satisfies Properties~\qref{item:15}
to~\qref{item:21}.
\end{lemma}

\begin{proof}
Let $u=x_1y_1\dots x_py_p$\,, with $x_1,\dots,x_p\in\bA_1$ and
$y_1,\dots,y_p\in\bA_2$\,, denote a generic element of~$\bA_0$\,. The length
of $u$ is defined by $\abs
{u}_0=\abs{x_1}_1+\dots+\abs{x_p}_1+\abs{y_1}_2+\dots+\abs{y_p}_2$\,, and the
length function thus defined on $\bA_0$ satisfies~\qref{item:15}.  The free
product of cancellative monoids is itself cancellative, hence $\bA_0$
satisfies~\qref{item:16}. Let $u'=x'_1y'_1\dots x'_qy'_q\in\bA_0$\,. Put
$k=\max\{j\leq p,q\tq x_1y_1\dots x_jy_j=x'_1y'_1\dots x'_jy'_j\}$. If $k=p$
then $u\wedge u'=u$ and if $k=q$ then $u\wedge u'=u'$. Otherwise, it is
readily seen that:
\begin{gather*}
  u\wedge u'=
  \begin{cases}
    x_1 y_1\dots x_ky_k(x_{k+1}\wedge x'_{k+1}),&\text{if $x_{k+1}\neq x'_{k+1}$\,,}\\
    x_1 y_1\dots x_ky_kx_{k+1}(y_{k+1}\wedge y'_{k+1}),&\text{if $x_{k+1}=x'_{k+1}$\,.}
  \end{cases}
\end{gather*}
Hence $u\wedge u'$ exists in all cases, and $\bA_0$ satisfies~\ref{item:17}.

Identify the two monoids $\bA_1$ and $\bA_2$ with their images in $\bA_0$
through the canonical injections $\bA_1\to\bA_0$ and $\bA_2\to\bA_0$\,. Then
the (disjoint) union $\SS_1\cup\SS_2$ is a finite Garside subset
of~$\bA_0$\,, which satisfies thus~\qref{item:18}.

Consider the \emph{augmented} Charney graph $(\Ch'_1,\to)$ of $\bA_1$, where
$\Ch'_1 = \SS_1 \setminus \{\unit_{\bA_1}\}$, \ie, $\Ch'_1 = \Ch_1 \cup
\{\Delta_1\}$ if $\SS_1$ has a maximum~$\Delta_1$, and $\Ch'_1 = \Ch_1$
otherwise. We define the augmented Charney graph $(\Ch'_2,\to)$ of $\bA_2$
similarly. Then, the Charney graph $(\Ch_0,\to)$ of $\bA_0$ is the bipartite
complete graph with parts $(\Ch'_1,\to)$ and $(\Ch'_2,\to)$. Therefore $\bA_0$
satisfies necessarily~\qref{item:19}. And since the two monoids $\bA_1$ and
$\bA_2$ are assumed to be non trivial, both parts contain edges. Hence
$\bA_0$ satisfies~\qref{item:21}.

Up to rescaling the length function $\abs{\cdot}_0$ by a multiplicative
factor, we assume that the integers in the set $\{\abs{x}_0 \tq x \in
\bA_0\}$ are setwise coprime. We show below that $\bA_0$
satisfies~\qref{item:20} by proving that, for all prime numbers~$p$, there
exists a cycle in $(\Ch_0,\to)$ whose total length is not divisible by~$p$.

We pick an element  $a\in\SS_1\setminus\{\unit_1\}$ such that $p$ does not
divide~$\abs{a}_1$\,. Such an element exists; otherwise, since $\SS_1$
generates~$\bA_1$\,, it would contradict our assumption that the integers in
$\{\abs{x}_0 \tq x \in \bA_0\}$ are setwise coprime. We also pick a cycle
$x_1 \to x_2 \to \ldots \to x_k \to x_1$  in $(\Ch_0,\to)$, such that $x_1$
and $x_k$ belong to~$\SS_2$. Then, both $x_1 \to x_2 \to \ldots \to x_k \to
x_1$ and $x_1 \to x_2 \to \ldots \to x_k \to a \to x_1$ are cycles in the
Charney graph $(\Ch_0,\to)$. Their total lengths differ by~$\abs{a}_1$, hence
$p$ does not divide both of them. This completes the proof.
\end{proof}

%\printbibliography
\bibliographystyle{plain}
\bibliography{biblio.bib}

\end{document}

%% file: sigma_1-2.pdf_t
\begin{picture}(0,0)%
\includegraphics{sigma_1-2.pdf}%
\end{picture}%
\setlength{\unitlength}{3108sp}%
\begingroup\makeatletter\ifx\SetFigFont\undefined%
\gdef\SetFigFont#1#2#3#4#5{%
  \reset@font\fontsize{#1}{#2pt}%
  \fontfamily{#3}\fontseries{#4}\fontshape{#5}%
  \selectfont}%
\fi\endgroup%
\begin{picture}(1470,1213)(2371,-2920)
\put(3826,-2356){\makebox(0,0)[lb]{\smash{{\SetFigFont{9}{10.8}{\rmdefault}{\mddefault}{\updefault}{\color[rgb]{0,0,0}$6$}%
}}}}
\put(2386,-2356){\makebox(0,0)[lb]{\smash{{\SetFigFont{9}{10.8}{\rmdefault}{\mddefault}{\updefault}{\color[rgb]{0,0,0}$3$}%
}}}}
\put(3601,-1906){\makebox(0,0)[lb]{\smash{{\SetFigFont{9}{10.8}{\rmdefault}{\mddefault}{\updefault}{\color[rgb]{0,0,0}$1$}%
}}}}
\put(3601,-2851){\makebox(0,0)[lb]{\smash{{\SetFigFont{9}{10.8}{\rmdefault}{\mddefault}{\updefault}{\color[rgb]{0,0,0}$5$}%
}}}}
\put(2611,-2851){\makebox(0,0)[lb]{\smash{{\SetFigFont{9}{10.8}{\rmdefault}{\mddefault}{\updefault}{\color[rgb]{0,0,0}$4$}%
}}}}
\put(2611,-1906){\makebox(0,0)[lb]{\smash{{\SetFigFont{9}{10.8}{\rmdefault}{\mddefault}{\updefault}{\color[rgb]{0,0,0}$2$}%
}}}}
\end{picture}%

%% file: ATM-R1.1.bbl
\begin{thebibliography}{10}

\bibitem{abbes17}
S.~Abbes, S.~Gou\"ezel, V.~Jug\'e, and J.~Mairesse.
\newblock Uniform measures on braid monoids and dual braid monoids.
\newblock {\em J. Algebra}, 473(1):627--666, 2017.

\bibitem{abbes15a}
S.~Abbes and J.~Mairesse.
\newblock Uniform and {B}ernoulli measures on the boundary of trace monoids.
\newblock {\em J. Combin. Theory Ser. A}, 135:201--236, 2015.

\bibitem{armstrong2009generalized}
D.~Armstrong.
\newblock {\em Generalized Noncrossing Partitions and Combinatorics of
  {C}oxeter Groups}.
\newblock Mem. Amer. Math. Soc., 2009.

\bibitem{athanasiadis2004noncrossing}
C.A. Athanasiadis and V.~Reiner.
\newblock Noncrossing partitions for the group {$D_n$}.
\newblock {\em SIAM J. Discrete Math.}, 18(2):397--417, 2004.

\bibitem{bessis2003dual}
D.~Bessis.
\newblock The dual braid monoid.
\newblock {\em Ann. Sci. \'Ec. Norm. Sup{\'e}r.}, 36(5):647--683, 2003.

\bibitem{bestvina1999non}
M.~Bestvina.
\newblock Non-positively curved aspects of {A}rtin groups of finite type.
\newblock {\em Geom. Topol.}, 3(1):269--302, 1999.

\bibitem{billingsley95}
P.~Billingsley.
\newblock {\em Probability and Measure, 3rd edition}.
\newblock Wiley, 1995.

\bibitem{birman1998new}
J.~Birman, K.H. Ko, and S.J. Lee.
\newblock A new approach to the word and conjugacy problems in the braid
  groups.
\newblock {\em Adv. Math.}, 139(2):322--353, 1998.

\bibitem{brady2002k}
T.~Brady and C.~Watt.
\newblock K$(\pi,1)$'s for {A}rtin groups of finite type.
\newblock {\em Geom. Dedicata}, 94(1):225--250, 2002.

\bibitem{brieskorn72}
E.~Brieskorn and K.~Saito.
\newblock Artin-{G}ruppen und {C}oxeter-{G}ruppen.
\newblock {\em Invent. Math.}, 17:245--271, 1972.

\bibitem{cartier69}
P.~Cartier and D.~Foata.
\newblock {\em Probl\`emes combinatoires de commutation et
  r\'ear\-ran\-ge\-ments}, volume~85 of {\em Lecture Notes in Math.}
\newblock Springer, 1969.

\bibitem{charney07}
R.~Charney.
\newblock An introduction to right-angled {A}rtin groups.
\newblock {\em Geom. Dedicata}, 129:1--13, 2007.

\bibitem{collet13}
P.~Collet, S.~Mart\'inez, and J.~San~Mart\'in.
\newblock {\em Quasi-Stationary Distributions. {M}arkov Chains, Diffusions and
  Dynamical Systems}.
\newblock Springer, 2013.

\bibitem{coxeter1935complete}
H.~Coxeter.
\newblock The complete enumeration of finite groups of the form $r_i^2=(r_i
  r_j)^{k_{ij}}= 1$.
\newblock {\em J. Lond. Math. Soc.}, 1(1):21--25, 1935.

\bibitem{darroch65}
J.N. Darroch and E.~Seneta.
\newblock On quasi-stationary distributions in absorbing discrete-time {M}arkov
  chains.
\newblock {\em J. Appl. Probab.}, 2:88--100, 1965.

\bibitem{dehornoy2002groupes}
P.~Dehornoy.
\newblock Groupes de {G}arside.
\newblock {\em Ann. Sci. \'Ec. Norm. Sup{\'e}r.}, 35(2):267--306, 2002.

\bibitem{dehornoy2013foundations}
P.~Dehornoy, F.~Digne, E.~Godelle, D.~Krammer, and J.~Michel.
\newblock {\em Foundations of {G}arside theory}, volume~22 of {\em EMS Tracts
  Math.}
\newblock EMS, 2015.

\bibitem{dehornoy2015garside}
P.~Dehornoy, M.~Dyer, and C.~Hohlweg.
\newblock Garside families in {A}rtin--{T}its monoids and low elements in
  {C}oxeter groups.
\newblock {\em C. R. Math. Acad. Sci. Paris}, 353(5):403--408, 2015.

\bibitem{dehornoy1999gaussian}
P.~Dehornoy and L.~Paris.
\newblock Gaussian groups and {G}arside groups, two generalisations of {A}rtin
  groups.
\newblock {\em Proc. Lond. Math. Soc.}, 79(3):569--604, 1999.

\bibitem{digne06}
F.~Digne.
\newblock Présentations duales des groupes de tresses de type affine
  $\widetilde {A}$.
\newblock {\em Comment. Math. Helv.}, 8(1):23--47, 2006.

\bibitem{flajolet09}
P.~Flajolet and R.~Sedgewick.
\newblock {\em Analytic {C}ombinatorics}.
\newblock Cambridge Univ. Press, 2009.

\bibitem{garside1969braid}
F.~Garside.
\newblock The braid group and other groups.
\newblock {\em Quarterly J. Math.}, 20(1):235--254, 1969.

\bibitem{gebhardt14}
V.~Gebhardt and S.~Tawn.
\newblock Normal forms of random braids.
\newblock {\em J. Algebra}, 408:115--137, 2014.

\bibitem{gierz03}
G.~Gierz, K.H. Hofmann, K.~Keimel, J.D. Lawson, M.W. Mislove, and D.S. Scott.
\newblock {\em Continuous Lattices and Domains}, volume~93 of {\em Encyclopedia
  Math. Appl.}
\newblock Cambridge Univ. Press, 2003.

\bibitem{godelle03}
E.~Godelle.
\newblock Parabolic subgroups of {A}rtin groups of type {FC}.
\newblock {\em Pacific J. Math.}, 208:243--254, 2003.

\bibitem{hennion_herve}
H.~Hennion and L.~Herv{\'e}.
\newblock {\em Limit {T}heorems for {M}arkov {C}hains and {S}tochastic
  {P}roperties of {D}ynamical {S}ystems by {Q}uasi-compactness}, volume 1766 of
  {\em Lecture Notes Math.}
\newblock Springer, 2001.

\bibitem{juge16:_combin}
V.~Jug\'e.
\newblock {\em Combinatorics of braids}.
\newblock PhD thesis, Univ. Paris Diderot, 2016.

\bibitem{kitchens97}
B.P. Kitchens.
\newblock {\em {S}ymbolic {D}ynamics. {O}ne-sided, {T}wo-sided and {C}ountable
  {S}tate {M}arkov Shifts}.
\newblock Springer, 1997.

\bibitem{lind95}
D.~Lind and B.~Marcus.
\newblock {\em An {I}ntroduction to {S}ymbolic {D}ynamics and {C}oding}.
\newblock Cambridge Univ. Press, 1995.

\bibitem{parry64}
W.~Parry.
\newblock Intrinsic {M}arkov chains.
\newblock {\em Trans. Amer. Math. Soc.}, 112(1):55--66, 1964.

\bibitem{picantin2002explicit}
M.~Picantin.
\newblock Explicit presentations for the dual braid monoids.
\newblock {\em C. R. Math. Acad. Sci. Paris}, 334(10):843--848, 2002.

\bibitem{picantin2005garside}
Matthieu Picantin.
\newblock Garside monoids vs divisibility monoids.
\newblock {\em Math. Structures Comput. Sci.}, 15(2):231--242, 2005.

\bibitem{reiner1997non}
V.~Reiner.
\newblock Non-crossing partitions for classical reflection groups.
\newblock {\em Discrete Math.}, 177(1--3):195--222, 1997.

\bibitem{rota64}
G.-C. Rota.
\newblock {On the foundations of combinatorial theory~{I}. {T}heory of
  {M}\"obius functions}.
\newblock {\em Z. Wahrscheinlichkeitstheorie}, 2:340--368, 1964.

\bibitem{rothblum07}
U.G. Rothblum.
\newblock Nonnegative and stochastic matrices.
\newblock In L.~Hogben, editor, {\em Handbook of Linear Algebra}, chapter~9.
  Chapman \& Hall, 2007.

\bibitem{gap3}
M.~Schönert et~al.
\newblock {\em {GAP} -- {Groups}, {Algorithms}, and {Programming} -- version 3
  release 4 patchlevel 4}.
\newblock Lehrstuhl D für Mathematik, Rheinisch Westfälische Technische
  Hoch\-schule, Aachen, Germany, 1997.

\bibitem{seneta81}
E.~Seneta.
\newblock {\em Non-negative Matrices and {M}arkov Chains. {R}evised printing}.
\newblock Springer, 1981.

\bibitem{viennot86}
X.~Viennot.
\newblock Heaps of pieces,~{I} : basic definitions and combinatorial lemmas.
\newblock In {\em Combinatoire \'enum\'erative}, volume 1234 of {\em Lecture
  Notes in Math.}, pages 321--350. Springer, 1986.

\end{thebibliography}
